\documentclass[11pt]{amsart}

\usepackage{epigamath}

\usepackage[english]{babel}

\numberwithin{equation}{section}

\usepackage[shortlabels]{enumitem}
\setlist[enumerate,1]{label={\rm(\arabic*)}, ref={\rm\arabic*}} 

\usepackage{amsmath, amssymb, mathrsfs, amsthm,  xparse, xspace, wasysym, manfnt}

\usepackage{mathtools}
\usepackage[all]{xy}
\usepackage{tabularx}
\usepackage{longtable}

\usepackage{cleveref}
\usepackage{aliascnt}
\usepackage[normalem]{ulem}
\crefname{enumi}{}{}

\crefformat{equation}{(#2#1#3)}
\let\oref\ref
\AtBeginDocument{\renewcommand{\ref}[1]{\Cref{#1}}}

\usepackage{pgf,tikz}
\usetikzlibrary{matrix,arrows,calc}
\usetikzlibrary{positioning}
\usetikzlibrary{decorations.pathreplacing,decorations.markings,decorations.pathmorphing}
\usetikzlibrary{positioning,arrows,patterns}
\usetikzlibrary{cd}
\usetikzlibrary{intersections}

\tikzset{
	every loop/.style={very thick},
	comp/.style={circle,fill,black,,inner sep=0pt,minimum size=5pt},
	order bottom left/.style={pos=.05,left,font=\tiny},
	order top left/.style={pos=.9,left,font=\tiny},
	order bottom right/.style={pos=.05,right,font=\tiny},
	order top right/.style={pos=.9,right,font=\tiny},
	order node dis/.style={text width=.75cm},
	circled number/.style={circle, draw, inner sep=0pt, minimum size=12pt},
	below left with distance/.style={below left,text height=10pt},
    below right with distance/.style={below right,text height=10pt}
	}

\tikzset{
  symbol/.style={
    draw=none,
    every to/.append style={
      edge node={node [sloped, allow upside down, auto=false]{$#1$}}}
  }
}

\newtheorem{lemma}{Lemma}[section]
\newtheorem{theorem}[lemma]{Theorem}
\newtheorem{thm}[lemma]{Theorem}
\newtheorem{proposition}[lemma]{Proposition}

\newtheorem{conjecture}[lemma]{Conjecture}

\theoremstyle{definition}
\newtheorem{dfx}[lemma]{Definition}
\newtheorem{df}[lemma]{Definition}

\theoremstyle{remark}
\newtheorem{remark}[lemma]{Remark}
\newtheorem{example}[lemma]{Example}

\def\={\;=\;}  \def\+{\,+\,}

\def\be{\begin{equation}}   \def\ee{\end{equation}}
\def\bes{\begin{equation*}}    \def\ees{\end{equation*}}
\def\ba{\be\begin{aligned}} \def\ea{\end{aligned}\ee}
\def\bas{\bes\begin{aligned}}  \def\eas{\end{aligned}\ees}

\definecolor{Mgreen}{RGB}{140,180,0} %

\newcommand*{\lspl}{\widetilde{\psi}} 

\renewcommand{\tilde}{\widetilde}

\newcommand*{\lcm}{\operatorname*{lcm}}

\newcommand*{\Pic}{\operatorname{Pic}}

\newcommand{\rom}[1]{\textup{\uppercase\expandafter{\romannumeral#1}}}

\newcommand{\ol}{\overline}

\newcommand*{\calH}{\mathcal H}

\newcommand*{\calL}{\mathcal L}
\newcommand{\calM}{{\mathcal M}}
\newcommand{\calO}{{\mathcal O}}

\newcommand{\calN}{{\mathcal N}}

\newcommand*{\calX}{\mathcal X}

\newcommand{\calZ}{{\mathcal Z}}
\newcommand{\calY}{{\mathcal Y}}

\newcommand{\CC}{{\mathbb{C}}}

\newcommand{\PP}{{\mathbb{P}}}

\newcommand{\ZZ}{{\mathbb{Z}}}

\newcommand{\bfs}{{\boldsymbol{s}}}

\newcommand{\bft}{{\boldsymbol{t}}}
\newcommand{\bfz}{{\boldsymbol{z}}}
\newcommand{\bfsigma}{{\boldsymbol{\sigma}}}

\newcommand{\bfomega}{{\boldsymbol{\omega}}}

\newcommand{\barmoduli}[1][g]{{\overline{\mathcal M}}_{#1}}

\DeclareDocumentCommand{\teich}{ O{\Sigma} O{\bfs}}{{\mathcal T}_{(#1,#2)}}

\DeclareDocumentCommand{\Pteich}{ O{\Sigma} O{\bfs} O{\mu}}{\mathcal{T}_{(#1,#2)}(#3)}
\DeclareDocumentCommand{\bigPteich}{ O{\Sigma} O{\bfs} O{\mu}}{\mathcal{T}^+_{(#1,#2)}(#3)}
\DeclareDocumentCommand{\Oteich}{ O{\Sigma} O{\bfs} O{\mu}}{\Omega \mathcal{T}_{(#1,#2)}(#3)}
\DeclareDocumentCommand{\bigOteich}{ O{\Sigma} O{\bfs} O{\mu}}{\Omega\mathcal{T}^+_{(#1,#2)}(#3)}
\DeclareDocumentCommand{\augteich}{ O{\Sigma} O{\bfs}}{{\overline{\mathcal T}}_{(#1,#2)}}

\DeclareDocumentCommand{\BSteich}{ O{\Sigma} O{\bfs}}{{\overline{\mathcal T}}^{\rm BS}_{(#1,#2)}}
\DeclareDocumentCommand{\Bteich}{ O{\Lambda} }{{{\mathcal T}}_{#1}}
\DeclareDocumentCommand{\BBSteich}{ O{\Lambda} }{{{\mathcal T}}^{\rm BS}_{#1}}
\DeclareDocumentCommand{\BBBSteich}{ O{\Lambda} }{{\widehat{{\mathcal T}}}^{\rm BS}_{#1}}
\DeclareDocumentCommand{\BPteich}{ O{\Lambda} O{\mu}}{\mathcal{T}_{#1}(#2)}
\DeclareDocumentCommand{\OBnoGRC}{ O{\Lambda} O{\mu}}{\Omega^{no} {{\mathcal T}}_{#1}(#2)}
\DeclareDocumentCommand{\OBteich}{ O{\Lambda} O{\mu}}{\Omega {{\mathcal T}}_{#1}(#2)}
\DeclareDocumentCommand{\OBBSteich}{ O{\Lambda} O{\mu}}{\Omega{{\mathcal T}}^{\rm BS}_{#1}(#2)}

\DeclareDocumentCommand{\flatMod}{ O{\Sigma} O{\bfs} O{\mu}}{\operatorname{Mod}_{(#1,#2)}(\mu)}
\DeclareDocumentCommand{\TNP}{ O{\Sigma} O{\bfs} O{\mu}}{\operatorname{TNP}_{(#1,#2)}(\mu)}

\newcommand{\aG}{\Gamma}                  
\newcommand{\eG}{\Gamma}       
\newcommand{\eL}{\Lambda}                   
\newcommand{\ltop}[1][e]{\ell(#1^{+})}                   
\newcommand{\lbot}[1][e]{\ell(#1^{-})}                  

\DeclareDocumentCommand{\pin}{ O{z} O{\omega}}{P^{\rm in}_{#1}}
\DeclareDocumentCommand{\pout}{ O{z} O{\omega}}{P^{\rm out}_{#1}}

\newcommand*{\Textd}[1][\Lambda]{T^{\bullet}_{#1}}

\newcommand*{\Tnorm}[1][\Lambda]{\overline{T}_{#1}^{n}}
\newcommand*{\Tsnorm}[1][\Lambda]{\overline{T}_{#1}^{s}}

\newcommand*{\Tw}[1][\Lambda]{\mathrm{Tw}_{#1}}  
\newcommand*{\sTw}[1][\Lambda]{\mathrm{Tw}_{#1}^s}  
\DeclareDocumentCommand{\sTwi}{ O{\Lambda} O{i}}{\mathrm{Tw}_{#1,#2}^{sv}}  

\newcommand*{\msds}{multi-scale differentials\xspace}

\DeclareDocumentCommand{\kmoduli}{ O{\mu} O{g} O{n} }{\mathcal{M}_{#2,#3}(#1)}
\DeclareDocumentCommand{\barkmoduli}{ O{\mu} O{g} O{n} }{\mathcal{\overline{M}}_{#2,#3}(#1)}
\DeclareDocumentCommand{\obarkmoduli}{ O{\mu} O{g} O{n} }{\Omega\mathcal{\overline{M}}_{#2,#3}(#1)}

\DeclareDocumentCommand{\LMS}{ O{\mu} O{g,n}} {\Xi\overline{\mathcal{M}}_{#2}(#1)}
\DeclareDocumentCommand{\RBLMS}{ O{\mu} O{g,n}}
{\Xi\widehat{\mathcal{M}}_{#2}(#1)}

\DeclareDocumentCommand{\MDstratum}{ O{\Lambda} O{\Lambda'}}{\mathcal{MD}_{#1}^{#2}}
\DeclareDocumentCommand{\OMDstratum}{  O{\Lambda} O{\Lambda'}}{\Omega\mathcal{MD}_{#1}^{#2}}

\DeclareDocumentCommand{\MDstratums}{ O{\Lambda} O{\Lambda'}}{\mathcal{MD}_{#1}^{s,#2}}

\DeclareDocumentCommand{\PDstratum}{ O{\Lambda} O{\Lambda'}}{\mathcal{PD}_{#1}^{#2}}

\DeclareDocumentCommand{\ptwT}{ O{\Lambda} O{\mu}}{{\Omega\mathcal{T}^{pm}_{#1}(#2)}}
\DeclareDocumentCommand{\ptwTm}{ O{\Lambda} O{\mu}}{{\Omega\mathcal{T}^{pm}_{#1}(#2)^-}}

\DeclareDocumentCommand{\kptwT}{ O{\Lambda} O{\mu}}{{\Omega^{k}\mathcal{T}^{pm}_{#1}(#2)}}

\DeclareDocumentCommand{\ocpM}{ O{g} O{\mu}}{\Omega{\overline{\mathcal{M}}_{#1}({#2})}}

\DeclareDocumentCommand{\cpM}{ O{g} O{\mu}}{{\overline{\mathcal{M}}_{#1}({#2})}}  

\DeclareDocumentCommand{\Oaugteich}{ O{\Sigma} O{\bfs} O{\mu}}{\Omega \overline{\mathcal{T}}_{(#1,#2)}(#3)}
\DeclareDocumentCommand{\Okaugteich}{ O{\Sigma} O{\bfs} O{\mu}}{\Omega^{k} \overline{\mathcal{T}}_{(#1,#2)}(#3)}

\DeclareDocumentCommand{\Paugteich}{ O{\Sigma} O{\bfs} O{\mu}}{\PP \Omega \overline{\mathcal{T}}_{(#1,#2)}(#3)} 

\DeclareDocumentCommand{\Dstratum}{ O{\Lambda} O{\Lambda'}}{{\PP\Xi\mathcal D}_#1^{#2}}
\DeclareDocumentCommand{\Dstratums}{ O{\Lambda} O{\Lambda'}}{{\PP\Xi\mathcal D}_#1^{#2,s}}
\DeclareDocumentCommand{\ODstratum}{ O{\Lambda} O{\Lambda'}}{{\Xi\mathcal D}_#1^{#2}}
\DeclareDocumentCommand{\ODstratums}{ O{\Lambda} O{\Lambda'}}{{\Xi\mathcal D}_#1^{#2,s}}
\DeclareDocumentCommand{\ODstratumso}{ O{\Lambda} O{\Lambda'}}{{\Xi\mathcal D}_{#1,\circ}^{#2,s}}
\DeclareDocumentCommand{\ODstratumo}{ O{\Lambda} O{\Lambda'}}{{\Xi\mathcal D}_{#1,\circ}^{#2}}

\DeclareDocumentCommand{\sMSfmark}{ O{\mu} O{\eL}}{\mathbf{MS}^s_{(#1,#2)}}
\DeclareDocumentCommand{\MSfmark}{ O{\mu} O{\eL}}{\mathbf{MS}_{(#1,#2)}}
\DeclareDocumentCommand{\sMSfun}{ O{\mu} }{\mathbf{MS}^s_{#1}}
\DeclareDocumentCommand{\MSfun}{ O{\mu} }{\mathbf{MS}_{#1}}
\DeclareDocumentCommand{\MSfun}{ O{\mu} }{\mathbf{MS}_{#1}}
\DeclareDocumentCommand{\MSgrp}{ O{\mu} }{\mathcal{MS}_{#1}}
\DeclareDocumentCommand{\GMS}{ O{\mu} }{\mathcal{GMS}_{#1}}
\DeclareDocumentCommand{\GSMS}{ O{\mu} O{g,n}} {G\Xi\overline\calM_{#2}(#1)}

\DeclareDocumentCommand{\sMDfmark}{ O{\mu} O{\eL}}{\mathbf{MD}^s_{(#1,#2)}}
\DeclareDocumentCommand{\MDfmark}{ O{\mu} O{\eL}}{\mathbf{MD}_{(#1,#2)}}
\DeclareDocumentCommand{\MDfun}{ O{\mu} O{\eG}}{\mathbf{MD}_{(#1,#2)}}

\newcommand{\whg}{\widehat{g}}

\newcommand{\wh}[1]{{\widehat{#1}}}

\DeclareDocumentCommand{\Thick}{ O{(X,\bfz)} O{\epsilon}}{#1_{#2}}

\DeclareDocumentCommand{\LMSk}{ O{\mu} O{g}} {\Xi^{k}\overline\calM_{#2}(#1)}
\DeclareDocumentCommand{\LMSc}{ O{\wh\mu} O{\whg}} {\Xi\overline\calM_{#2}(#1)}
\DeclareDocumentCommand{\LMSkc}{ O{\wh\mu} O{\whg}} {\Xi^{\sim}\overline\calM_{#2}(#1)}

\newcommand{\TWGRC}[1][\oG]{TWGRC}
\newcommand{\TW}[1][\oG]{TW}

\makeatletter
\newcommand*{\doublerightarrow}[2]{\mathrel{
  \settowidth{\@tempdima}{$\scriptstyle#1$}
  \settowidth{\@tempdimb}{$\scriptstyle#2$}
  \ifdim\@tempdimb>\@tempdima \@tempdima=\@tempdimb\fi
  \mathop{\vcenter{
    \offinterlineskip\ialign{\hbox to\dimexpr\@tempdima+1em{##}\cr
    \rightarrowfill\cr\noalign{\kern.5ex}
    \rightarrowfill\cr}}}\limits^{\!#1}_{\!#2}}}
\newcommand*{\triplerightarrow}[1]{\mathrel{
  \settowidth{\@tempdima}{$\scriptstyle#1$}
  \mathop{\vcenter{
    \offinterlineskip\ialign{\hbox to\dimexpr\@tempdima+1em{##}\cr
    \rightarrowfill\cr\noalign{\kern.5ex}
    \rightarrowfill\cr\noalign{\kern.5ex}
    \rightarrowfill\cr}}}\limits^{\!#1}}}
\makeatother

\newcommand{\on}[1]{\operatorname{#1}}
\newcommand{\bb}[1]{{\mathbb{#1}}}

\newcommand{\ca}[1]{{\mathcal{#1}}}
\newcommand{\bd}[1]{{\mathbf{#1}}}
\newcommand{\ul}[1]{{\underline{#1}}}

\newcommand{\Span}[1]{\left<#1\right>}

\newcommand{\lra}{\longrightarrow}
\newcommand{\hra}{\hookrightarrow}
\newcommand{\sub}{\subseteq}

\def\lowsim{\vbox to 0pt{\vss\hbox{$\scriptstyle\sim$}\vskip-1.5pt}}
\newcommand{\longiso}{\overset{\lowsim}\lra}
\newcommand{\iso}{\xrightarrow{\lowsim}}

\usepackage{letltxmacro}
\LetLtxMacro{\phiorig}{\phi}
\renewcommand{\phi}{\varphi}

\newcommand{\Mbar}{\overline{\ca M}}

\newcommand{\isom}{\iso}
\newcommand{\longisom}{\longiso}

\newcommand{\cat}[1]{\bd{#1}}

\newcommand{\trop}{\mathsf{trop}}

\newcommand{\gp}{\mathsf{gp}}

\newcommand{\Picabs}{\mathfrak{Pic}}

\newcommand{\aj}{\mathsf{aj}}

\newcommand{\M}{\mathsf{M}}
\newcommand{\ghost}{\overline{\mathsf{M}}}
\newcommand{\sltelement}{\xi}

\newcommand{\vedges}{E^v}

\newcommand{\coarse}{{\sf{coarse}}}

\newcommand{\MW}{\mathrm{MW}}
\newcommand{\pt}{\mathrm{pt}}

\DeclareRobustCommand\longtwoheadrightarrow
    {\relbar\joinrel\twoheadrightarrow}

\newcommand{\supth}[1]{\ensuremath{#1^{\mathrm{th}}}} 

\EpigaVolumeYear{9}{2025} \EpigaArticleNr{21} \ReceivedOn{May 5, 2023}
\InFinalFormOn{October 16, 2024}
\AcceptedOn{May 11, 2025}

\title{A tale of two moduli spaces: Logarithmic \\and multi-scale differentials}
\titlemark{Moduli of logarithmic and multi-scale differentials}

\author{Dawei Chen}
\address{Department of Mathematics, Boston College, Chestnut Hill, MA 02467, USA}
\email{dawei.chen@bc.edu}

\author{Samuel Grushevsky}
\address{Simons Center for Geometry and Physics, and Department of Mathematics, Stony Brook University, Stony Brook, NY 11794-3636,~USA}
\email{sam@math.stonybrook.edu}

\author{David Holmes}
\address{Mathematisch Instituut, Universiteit Leiden, Einsteinweg 55, 2333 CC Leiden, Netherlands} 
\email{holmesdst@math.leidenuniv.nl}

\author{Martin M\"oller}
\address{Institut f\"ur Mathematik, Goethe-Universit\"at Frankfurt,
Robert-Mayer-Str. 6-8, 60325 Frankfurt am Main, Germany}
\email{moeller@math.uni-frankfurt.de}

\author{Johannes Schmitt}
\address{ETH Zurich, Department Mathematik, Raemistrasse 101, CH-8092 Zurich, Switzerland} 
\email{johannes.schmitt@math.uzh.ch}

\authormark{D.~Chen, S.~Grushevsky, D.~Holmes, M.~M\"oller, and J.~Schmitt}

\AbstractInEnglish{\footnotesize{Multi-scale differentials were constructed by M.~Bainbridge, D.~Chen, Q.~Gendron, S.~Grushevsky, and M.~M\"oller, from the viewpoint of flat and complex geometry, for the purpose of compactifying moduli spaces of curves together with a differential with prescribed orders of zeros and poles. Logarithmic differentials were constructed by S.~Marcus and J.~Wise, as a generalization of stable rubber maps from Gromov--Witten theory. Modulo the global residue condition that isolates the main components of the compactification, we show that these two kinds of differentials are equivalent, and establish an isomorphism of their (coarse) moduli stacks. Moreover, we describe the rubber and multi-scale spaces as an explicit blowup of the moduli space of stable pointed rational curves in the case of genus zero, and as a global blowup of the incidence variety compactification for arbitrary genera, which implies their projectivity. We also propose a refined double ramification cycle formula in the twisted Hodge bundle which interacts with the universal line bundle class.}}


\acknowledgement{Research of D.\,C.\ is supported in part by the National Science Foundation under the grants DMS-20-01040 and DMS-23-01030, as well as by the Simons Foundation through Travel Support for Mathematicians and a Simons Fellowship.  Research of S.\,G.\ is supported in part by the National Science Foundation under the grants DMS-18-02116 and DMS-21-01631.  Research of D.\,H.\ is supported by grants 613.009.103 and VI.Vidi.193.006 of the Dutch Research Council (NWO). Research of M.\,M.\ is supported by the DFG-project MO 1884/2-1 and by the Collaborative Research Centre TRR 326 `Geometry and Arithmetic of Uniformized Structures'. Research of J.\,S.\ is supported by the Swiss National Science Foundation under the SNF Early Postdoc Mobility grant SNF-184245 and the grant SNF-184613. Moreover, part of the research was conducted at the Max-Planck Institute for Mathematics in Bonn, which the author wants to thank warmly.}

\begin{document}


\removeabove{1cm}
\removebetween{0.8cm}
\removebelow{1cm}

\maketitle

\begin{prelims}

\DisplayAbstractInEnglish

\medskip

\DisplayKeyWords

\smallskip

\DisplayMSCclass

\end{prelims}


\newpage

\setcounter{tocdepth}{1}

\tableofcontents


\section{Introduction}

\subsection{Background and main results}
Let $\mu = (m_1, \ldots, m_n)$ be a tuple of integers with $\sum_{i=1}^n m_i = 2g-2$.  The (projectivized) \emph{stratum of differentials of type} $\mu$ is the moduli space of smooth curves $X$ of genus $g$ with distinct marked points $z_1,\dots,z_n\in X$ such that $\sum_{i=1}^n m_i z_i$ is a (possibly meromorphic) canonical divisor.

The study of differentials with prescribed zeros and poles is important for at least two reasons. On the one hand, a (holomorphic) differential induces a flat metric with conical singularities at its zeros, such that the underlying Riemann surface can be realized as a polygon with edges pairwise identified by translations. Varying the shape of the polygons by affine transformations of the plane induces an action on the strata of differentials (called Teichm\"{u}ller dynamics), whose orbit closures (called affine invariant subvarieties) govern intrinsic properties of surface dynamics. On the other hand, a differential (up to multiplication by a scalar) corresponds to a canonical divisor in the underlying complex curve. Hence the union of the moduli spaces of differentials with all possible configurations of zeros stratifies the Hodge bundle over the moduli space of curves, thus producing a number of remarkable questions to investigate from the viewpoint of algebraic geometry, such as compactification, enumerative geometry, and cycle class calculations. The interplay of these aspects has brought the study of differentials to an exciting new stage (see, \textit{e.g.}, \cite{zorichsurvey, wrightsurvey, chensurvey} and the references therein for an introduction to this subject).

As in
many other moduli problems, having a geometrically meaningful compactification plays a crucial role in the study of the strata of differentials. Extending the setup of canonical divisors with prescribed zeros and poles to (pre-)stable curves, we define an algebraic stack $\GSMS$, the moduli space of {\em generalized {simple} multi-scale differentials of type~$\mu$}. The relative coarse moduli space $\GMS$ over $\barmoduli[g,n]$ of this stack is defined the same way as the multi-scale differentials in~\cite{LMS}, but dropping the global residue condition.\footnote{Our definition thus solves a task left open in~\cite{LMS}, namely to describe the smooth stack~$\LMS$ dominating the stack of multi-scale differentials~$\MSgrp$ without invoking Teichm\"uller markings.}  Compared to the multi-scale space, the key player in \cite{LMS}, the stack $\GSMS$ has additional irreducible components whose generic elements parameterize differentials on (strictly) nodal curves. Indeed $\GSMS$ maps onto the space of twisted canonical divisors constructed by Farkas--Pandharipande \cite{fapa}. A minimal logarithmic structure on the space of twisted canonical divisors was described in \cite{ChenChenLog}, which extracts the information of meromorphic differentials from lower levels, but does not specify the full level structure. The precise definitions on these related objects are recalled in~\ref{sec:famGMS}.

On the logarithmic side, Marcus and Wise \cite{MarcusWiseLog} defined, for any line bundle $\ca L$ on the universal curve $X_{g,n}$ over $\Mbar_{g,n}$, a space $\cat{Rub}_\ca L$ over $\Mbar_{g,n}$. The fiber of $\cat{Rub}_\ca L$ over a curve $X$ is the set of piecewise linear functions $\beta$ on the tropicalization of $X$, together with an isomorphism of line bundles from $\ca O_X(\beta)$ to $\ca L$. The natural $\bb C^*$ quotient, which forgets the data of the isomorphism, is denoted by $\bb P(\cat{Rub}_\ca L)$. When $\ca L = \ca O_{X_{g,n}}(\sum_i m_i z_i)$, this space $\bb P(\cat{Rub}_\ca L)$ is the space of rubber maps to $\bb P^1$ with zeros and poles prescribed by the $m_i$, giving an alternative definition to that of Li, Graber and Vakil \cite{Li2001Stable-morphism, GV}. This machinery gives an extremely clean and functional definition of the double ramification cycle, as well as its logarithmic, pluricanonical, universal,
and iterated variants; see \cite{BHPSS, HolmesSchmitt, MPS, MR, HMPPS}. 

To connect this space with moduli of differentials, we define the line bundle
\[\ca L_\mu = \omega_{X_{g,n}/\Mbar_{g,n}}\left(-\sum_{i=1}^n m_i z_i\right)\]
on $X_{g,n}$, leading to the stack $\cat{Rub}_{\ca L_\mu}$, together
with its relative coarse moduli space $\cat{Rub}_{\ca L_\mu}^{\sf{coarse}}$ over $\Mbar_{g,n}$.\footnote{See \cite{Abramovich2011Twisted-stable-} for the definition of relative coarse moduli spaces. Moreover, note that one can replace $\omega$ with any power $\omega^{\otimes k}$ in the formula for $\ca L_\mu$, extending the theory to $k$-canonical divisors.} The virtual fundamental class of $\bb P(\cat{Rub}_{\ca L_\mu}^{\sf{coarse}})$ is the `canonical' double ramification cycle described in~\cite{HolmesSchmitt}.

The definitions of the spaces $\cat{Rub}_{\ca L_\mu}$ and $\GMS$ look very different. They can be found in~\Cref{sec:rubber,sec:famGMS}, respectively. The main aim of this paper is to show that these definitions are in fact \emph{essentially equivalent}. More precisely, we prove the following theorem.

\begin{thm} \label{intro:mainiso} For any tuple of integers $\mu = (m_1, \ldots, m_n)$
with $\sum_{i=1}^n m_i = 2g-2$, there is an isomorphism of algebraic stacks over $\Mbar_{g,n}$
\bes
F \colon \cat{Rub}_{\ca L_\mu} \lra \GSMS 
\ees
which induces an isomorphism of the corresponding relative coarse
moduli spaces over $\Mbar_{g,n}$
\bes
\ol{F} \colon \cat{Rub}_{\ca L_\mu}^{\sf{coarse}} \lra \GMS\,.
\ees
\end{thm}

Note that the global residue condition described in \cite{BCGGM1} can isolate the main component of $\GMS$, called the {\em multi-scale} space and denoted by ${\mathcal MS}_\mu$. In other words, a generalized multi-scale differential not satisfying the global residue condition is not smoothable while preserving the prescribed zero and pole orders. (This global residue condition arises from applying Stokes' theorem to subcurves of the limiting nodal curve when differentials degenerate from nearby smooth curves, thereby imposing that certain sums of residues at the nodes vanish. See \textit{op.~cit.}~for further details.) Moreover, in \cite{LMS} the space of multi-scale differentials ${\mathcal MS}_\mu$ was shown to possess nice geometric properties, such as smoothness (as a stack), normal crossings boundary, and extension of the ${\rm GL}_2(\mathbb R)$-action to the boundary (after a real oriented blowup). It would be interesting to see whether these properties can be obtained directly by using rubber differentials and logarithmic geometry. 
 
\subsection{Applications and related topics}
In what follows we address several constructions, results, and conjectures related to the main result above. 
\subsubsection{A blowup description of the space of multi-scale differentials}
First, describing a modular compactification via {\em blowups} can be useful in many aspects, \textit{e.g.}, for projectivity and intersection calculations. There is a natural action of $\CC^*$ on generalized multi-scale differentials by simultaneous rescaling of all differentials, and we denote the quotient, the space of `projectivized generalized multi-scale differentials', by $\PP(\GMS)$; \ref{intro:mainiso} induces an isomorphism $\PP(\cat{Rub}_{\ca L_\mu}^{\sf{coarse}}) \isom \PP(\GMS)$. 

In the case of genus zero, we can identify $\PP(\cat{Rub}_{\ca L_\mu}^{\sf{coarse}})$ with a blowup of $\overline{\mathcal M}_{0,n}$.  

\begin{thm}[\ref{thm:blowup-0}]
\label{intro:blowup-0}
For $g=0$ there exists an explicit sheaf of ideals in $\overline{\mathcal M}_{0,n}$ such that the normalization of its blowup is  $\PP(\cat{Rub}_{\ca L_\mu}^{\sf{coarse}})$. 
\end{thm}

We recall that the projectivized stratum of differentials can be compactified in different ways. Firstly, one can consider simply its closure in the Deligne--Mumford compactification $\Mbar_{g,n}$. Secondly, one can consider the closure of the stratum in the total space of the projectivized Hodge bundle over $\Mbar_{g,n}$ (twisted by the polar parts). This compactification is described completely in \cite{BCGGM1}, and is called the incidence variety compactification (IVC). The IVC clearly admits a morphism onto the Deligne--Mumford closure of the stratum, while $\PP({\mathcal MS}_\mu)$ maps onto the IVC, and in general both these morphisms are `forgetful', \textit{i.e.}~contract some loci in the compactifications. We further write NIVC for the normalization of the IVC.

In \cite{nguyen} Nguyen showed that, in the case of genus zero, the IVC can be described as an explicit blowup of $\Mbar_{0,n}$. From the above theorem, one can also retrieve Nguyen's result, which we do in~\ref{prop:blowup-0}.  
 
In arbitrary genus, recall that the multi-scale space ${\mathcal MS}_\mu$ is the main component of $\GMS$, whose generic element parameterizes differentials with prescribed zero and pole orders on smooth curves. 
 
\begin{thm}[\ref{thm:blowup-g}]
For arbitrary genus there exists a global sheaf of ideals on the NIVC such that the normalization of the blowup of the NIVC along this ideal  gives the projectivized multi-scale space $\PP({\mathcal MS}_\mu)$.  Consequently, the coarse moduli space of the stack $\PP({\mathcal MS}_\mu)$ is a projective variety.  
\end{thm}

In \cite{LMS} a local blowup construction to obtain $\PP({\mathcal MS}_\mu)$ from the normalization of the IVC was described. That construction did not glue to a global sheaf of ideals on IVC, and hence did not yield the projectivity of $\PP({\mathcal MS}_\mu)$. In \cite{ccm} the projectivity of $\PP({\mathcal MS}_\mu)$ was established by constructing an explicit ample divisor class on it. Thus the above theorem provides a distinct conceptual understanding of the projectivity result.

\subsubsection{A Hodge double ramification cycle}\label{sec:intro:hodge_DR}
Next we propose a refined version of the {\em double ramification} ({\em DR}\,) {\em cycle} in the twisted Hodge bundle and conjecture a Pixton-style formula for this class, involving coefficients of higher powers of the regularizing parameter `$r$'. For this purpose, we also generalize our considerations to $k$-differentials. 
 
Let $A = (a_1, \ldots, a_n)\in \bb Z^n$, where $|A| \coloneqq \sum_{i=1}^n a_i = k(2g-2+n)$ for some $k > 0$, and denote by
\begin{equation*}
 \calL_A \coloneqq \omega^{\otimes k}\left(-\sum_{i=1}^n (a_i - k) z_i\right)
\end{equation*}
the associated degree zero line bundle on $X_{g,n}$, where $\pi\colon X_{g,n}\to \overline{\calM}_{g,n}$ is the universal curve with sections $z_i$, and $\omega$ is the relative canonical bundle.\footnote{Here we switch to the logarithmic version of indices to match the notation in \cite{Janda2016Double-ramifica}. In other words, as a signature of $k$-differentials, each of the zero and pole orders is given by $a_i-k$. In particular, by slight abuse of notation, $\calL_A$ is simply the bundle we denoted by $\calL_\mu$ in the previous convention.} Taking 
$$
\calH \coloneqq \omega^{\otimes k}\left(-\sum_{\substack{i \in \{1,\ldots,n\}: \\  a_i < 0}} a_i z_i + \sum_{i=1}^n kz_i\right)
$$
to be the relative logarithmic $k$-canonical bundle twisted by the polar part, 
we obtain a natural diagram 
\begin{equation}\label{eq:map_to_hodge}
\begin{tikzcd}
\PP\left(\cat{Rub}_{\calL_A}\right) \arrow[r,"F"] \arrow[dr,"p",swap] &  \PP(\pi_*\calH) \arrow[d,"q"]\\
& \Mbar_{g,n}
\end{tikzcd}
\end{equation}
(see the discussion leading to~\ref{lem:map-hodge} for more details). Pushing forward the virtual fundamental class of $\PP(\cat{Rub}_{\calL_A})$ gives a lift
$$
\widetilde{{\rm DR}}^k_A \= F_* \left[\PP\left(\cat{Rub}_{\calL_A}\right) \right]^\mathrm{vir}
$$
of the twisted DR cycles to $\bb P(\pi_*\calH)$, which we call the twisted {\em Hodge DR cycle}.  
 
Let $H = c_1(\calO(1))$ be the universal line bundle class on $\bb P(\pi_*\calH)$, and let $\eta = F^*H$ be its pullback to $\PP(\cat{Rub}_{\calL_A})$.\footnote{In the literature sometimes $\xi$ denotes the universal line bundle class on the space of $k$-differentials and  $\eta$ denotes the tautological line bundle class $c_1(\calO(-1))$.} By the projective bundle formula associated to the map $q$, to determine the class of $\widetilde{{\rm DR}}^k_A$ in the Chow ring $\operatorname{CH}^\bullet(\bb P(\pi_*\calH))$, it suffices to determine the cycle class   
\begin{equation} \label{eqn:qpushforwardetau} q_{*}\left(\widetilde{{\rm DR}}^k_A \cdot H^u\right) \= p_*\left(\left[\PP\left(\cat{Rub}_{\calL_A}\right) \right]^\mathrm{vir} \cdot \eta^u \right) \in \operatorname{CH}^{g+u}\left(\overline{\calM}_{g,n}\right) \end{equation}
for every $u$.  

Before proceeding to give a conjectural formula for these cycles, let us make a remark about the case $k=0$. When trying to follow the construction above, we encounter the issue that in general the higher cohomology of $\calH$ will not vanish, so that $\bb P(\pi_*\calH)$ is not a projective bundle. In~\ref{sect:universal_bundle} we explain how this can be remedied.
However, there is also an alternative approach to defining $\eta$, which makes clearer a connection to relative Gromov--Witten theory: there both the space $\PP(\cat{Rub}_{\calL_A})$ and its forgetful map $p$ to~$\Mbar_{g,n}$ still make sense, and it was proven in \cite[Proposition 50]{BHPSS} that there is a natural isomorphism
$$
\PP\left(\cat{Rub}_{\calL_A}\right) \,\cong \, \Mbar_{g,A}\left(\mathbb{P}^1, 0, \infty\right)^\sim
$$
with the space of stable maps to \emph{rubber} $\mathbb{P}^1$ relative to $0, \infty$, with contact orders specified by the vector $A$. This space of stable maps parameterizes maps from prestable curves to a chain of rational curves, with marked points $0,\infty$ at opposite ends of the chain (see \cite[Section 0.2.4]{Janda2016Double-ramifica} for details). What is important for us is that it still carries a natural divisor class $\eta = \Psi_{\infty}$ defined as the class of the cotangent line bundle at the marked point~$\infty$ on the chain of rational curves.
 
Continuing in the general case $k \geq 0$, consider the space of twisted $r$-spin structures $\overline{\calM}_{g;A}^{r,k}$ constructed in \cite{Chiodo_twisted, Jarvis_higher_spin}. 
This is a compactification of the moduli space of smooth marked curves~$X$ and line bundles~$L$ on~$X$ with an isomorphism $L^{\otimes r} \cong \omega_X^{\otimes k} (-\sum (a_i - k) z_i)$. In the compactification, the curve~$X$ is allowed to acquire nodal singularities that are stacky points with some finite stabilizer group, making~$X$ a \emph{twisted curve} in the sense of \cite{MR1862797}. 
The moduli space then carries a universal curve $\pi\colon \calX \to \overline{\calM}_{g;A}^{k,r}$ with a line bundle~$\calL$ and isomorphism $\calL^{\otimes r} \cong \omega_\pi^{\otimes k} (-\sum (a_i - k) z_i)$. Here we follow the notation of \cite{Janda2016Double-ramifica}. Forgetting the line bundle and the stacky structure on $\calX$ gives a map  $\epsilon\colon \overline{\calM}_{g;A}^{k,r}\to \overline{\calM}_{g,n}$.  
Define the following Chiodo's class as the first cycle class given in \cite[Proposition~5]{Janda2016Double-ramifica}: 
$$ {\rm Ch}_{g,A}^{k,r,d} \coloneqq r^{2d-2g+1} \epsilon_{*}c_d\left(-R^{*}\pi_{*}\calL\right) \in {\rm CH}^d\left(\overline{\calM}_{g,n}\right), $$
where $R^*$ denotes derived push-forward. It is a polynomial in $r$ (for $r$ sufficiently large). Following com\-pu\-ta\-tions of Chiodo \cite{Chiodo_twisted}, the class ${\rm Ch}_{g,A}^{k,r,d}$ can be computed explicitly as a sum over stable graphs, decorated with polynomials in $\kappa$ and $\psi$-classes (see \cite[Corollary~4]{Janda2016Double-ramifica}). We propose the following conjecture, giving a formula for the cycle classes~\ref{eqn:qpushforwardetau}.
 
\begin{conjecture}[Hodge DR] 
\label{conj:HDR}
For every $g,k,u \geq 0$ and every $A \in \mathbb{Z}^n$ satisfying $|A|=k(2g-2+n)$, the following relation holds: 
$$ p_*\left(\left[\PP\left(\cat{Rub}_{\calL_A}\right) \right]^\mathrm{vir} \cdot \eta^u \right) \= [r^u]{\rm Ch}_{g,A}^{k, r, g+u} \in \mathrm{CH}^{g+u}\left(\Mbar_{g,n}\right),$$ 
where $[r^u]$ means taking the coefficient of $r^u$. 
\end{conjecture}
 
For $u =0$, by definition the left-hand side of this equation
is the usual twisted DR cycle ${\rm DR}^k_A$, and by \cite[Proposition 5]{Janda2016Double-ramifica} the right-hand side agrees with Pixton's formula for this cycle. Therefore, the conjecture is true for $u=0$ by the results of \cite{BHPSS}.
 
For $u>0$, the conjecture can be verified computationally in many examples for
the special case $g=0$. Indeed, in this case, the space $\cat{Rub}_{\calL_A}$ agrees with the space of multi-scale $k$-differentials by
\ref{intro:mainiso} (since the global residue condition is automatically satisfied in the case $g=0$). Then the software package \texttt{diffstrata}, see \cite{CoMoZadiffstrata}, can compute powers of $\eta$ on this space using relations in its Picard group, and express the left-hand side of the conjecture in terms of tautological classes. On the other hand, the right-hand side of the conjecture can be computed in \texttt{admcycles}, see \cite{DSvZ}, using the graph-sum formula from \cite{Janda2016Double-ramifica}. Using this, the prediction of the conjecture has been verified for several example vectors $A$, giving many non-trivial equalities in the Chow group of $\Mbar_{0,n}$. The calculations in \texttt{diffstrata} for $k>1$ rely on some code in development related to the forthcoming paper \cite{CMS2023}.

On the other hand, for $k=0$, the left-hand side of the conjecture has been computed in \cite[Corollary~4.3]{FWY}. The formula given there is similar, but not equal, to the one above. However, using properties of the Chiodo class proven in \cite[Theorem 4.1(ii)]{DaniloEulerChar}, a short computation shows that the formula from \cite{FWY} can be simplified to the one we give above.\footnote{Special thanks go to Longting Wu for patiently explaining their formula and to Danilo Lewa\'nski for informing us of the above property of the Chiodo class.}

\begin{theorem}\label{thm:HDR_k_is_0}
\ref{conj:HDR} is true for $k=0$: for any $g,u \geq 0$ and any vector $A \in \mathbb{Z}^n$ with sum $|A|=0$, we have
\[
p_*\left(\left[\PP\left(\cat{Rub}_{\calL_A}\right) \right]^\mathrm{vir} \cdot \eta^u \right) 
= p_*\left(\left[\Mbar_{g,A}\left(\mathbb{P}^1, 0, \infty\right)^\sim\right]^\mathrm{vir} \cdot \Psi_\infty^u \right)  
\= [r^u]{\rm Ch}_{g,A}^{0, r, g+u} .
\]
\end{theorem}

\subsection{Sketch of the comparison}
We hope that this paper will foster more communications between two
groups of researchers, those working in logarithmic geometry for moduli spaces and those working in moduli of differentials for Teichm\"uller dynamics. With this in mind, we have written out
definitions of objects on both sides of the story in a rather detailed way,
in particular assuming minimal background knowledge about logarithmic structures.
We now give an overview of the comparison in~\ref{intro:mainiso}.

\smallskip
The definition of generalized multi-scale differentials on a stable curve $X$ is
geometrically very concrete but quite lengthy. The {\em level structure} (or {\em full
order}\,) on the vertices of the dual graph $\Gamma$ of $X$, corresponding to the irreducible components of $X$, encodes the vanishing orders of a  differential in a family of differentials on smooth curves that degenerates to
a given multi-scale differential on a nodal curve. One can twist differentials that vanish identically, on the irreducible
components of the same level, by a rescaling parameter for that level, to obtain {\em
twisted differentials} that are not identically zero on the components on
that level. A multi-scale differential contains the combinatorial
data of the zero and pole orders of twisted differentials at the nodes.
Moreover, the {\em prong-matchings} of a multi-scale differential are
combinatorial data that arise from choices of smoothing a nodal differential
with matching zero and pole orders at the two branches at a node, under the flat metric
induced by the differential. Lastly,  a multi-scale differential stores
the smoothing parameters of the nodes in a way consistent with the level
structure, packaged in the notion of a {\em rescaling ensemble}.
On all these parameters, a certain {\em level rotation torus} acts and
induces a notion of equivalence that forgets the extra information due to
various choices being made in the above process, \textit{e.g.}, how simultaneously
scaling twisted differentials on the same level affects prong-matchings.

The definition of an element of $\cat{Rub}_{\ca L_\mu}$
is very concise; it is simply a piecewise linear function on the tropicalization subject to certain conditions (see Definitions \oref{def:rub} and \oref{def:rub_L}). However, it may seem cryptic at a first reading. In particular,
it may not be immediately apparent why the data of a log curve, a piecewise
linear function, and an isomorphism of line bundles should yield up all
the above data of an equivalence class of multi-scale differentials. Some parts
of the comparison (such as the enhanced level graph) are obtained essentially
by some bookkeeping, but extracting the \emph{level rotation torus} and
\emph{prong-matchings} from the logarithmic data requires significantly more care. 

Our first key insight about prong-matchings is~\ref{lem:prong_matching_comparison}, giving a new, coordinate-free characterization of prong-matching via the residue isomorphism. The second key insight exhibits the reason for the equivalence relation given by the level rotation torus in log language.  We define a \emph{log splitting} of a point in $\cat{Rub}_{\ca L_\mu}(B)$ essentially as a section of the quotient map $\M_B \to \ghost_B$, from the sheaf of monoids~$\M_B$ built into the log structure, to the ghost sheaf $\ghost_B$. The precise statement is given in~\ref{df:logsplitting}. We show that the set of log splittings is closely related to the level rotation torus, and in particular changing the choice of log splitting corresponds to the action of the level rotation torus.

Finally, we remark that an analogue of~\ref{intro:mainiso} should also hold for rubber $k$-differentials and multi-scale $k$-differentials. Indeed, on the logarithmic side the generalization is straightforward, as noted earlier. Moreover, the space of multi-scale $k$-differentials was described similarly in \cite{CoMoZa}. Thus the arguments in the current paper can be adapted directly to compare the two versions of $k$-differentials. We leave the details to the interested reader.  
 
\subsection*{Outline of the paper}
In~\ref{sec:rubber} we give the basic definitions of logarithmic rubber maps, and in~\ref{sec:famGMS} we do the same for generalized multi-scale differentials. In the somewhat technical~\ref{sec:minimal_log_str}, we describe the underlying algebraic stack that comes from the logarithmic definition in~\ref{sec:rubber}, which will be essential for what follows. \Cref{sec:log_to_GMS} is the technical heart of our comparison theorem, where we show how to construct a multi-scale differential from a logarithmic one, and vice versa. In~\ref{sect:universal_bundle} we discuss several constructions of the universal line bundle class $\eta$ that appears in the Hodge DR conjecture and prove the conjecture in the case of $k = 0$. In~\ref{sec:blowup_descriptions} we describe some of the moduli spaces concerned via blowup constructions. Finally, the sign conventions generally adopted in the logarithmic and multi-scale worlds are unfortunately opposite to one another; in the appendix we explain a small variation on the logarithmic definitions which makes the signs match. 
 
\subsection*{Acknowledgments}
The authors would like to thank Quentin Gendron for many helpful discussions. We also thank Qile Chen, Honglu Fan, Danilo Lewa\'nski, Rahul Pandharipande, Adrien Sauvaget, Johannes Schwab, and Longting Wu for sharing comments and insights on a preliminary version of the current paper. This project was initiated during an online seminar organized by D.\,H.\ and J.\,S.; we would like to thank all the participants for making the seminar a success. Finally, we thank the anonymous referees for carefully reading the paper and providing many helpful comments that helped clarify some arguments and presentation.

\section{Logarithmic rubber maps}\label{sec:rubber}

In the following we recall the relevant notions from logarithmic geometry that are needed to define logarithmic rubber maps and make the comparison to the space of generalized multi-scale differentials. Instead of a broader introduction, we focus on the precise parts of the theory that are needed.  We refer the reader to \cite{Kato1989Logarithmic-str,Ogus2018Lectures-on-Log} for a more detailed treatment of the basic notions of log geometry, and mention more specialized references where appropriate later in the text.

\subsection{Overview of log divisors}
\label{sec:logoverview}

A log scheme is a pair
\begin{equation}\label{eq:log_str}
(B, \alpha\colon \M_B \to \calO_B),
\end{equation}
where $B$ is a scheme, $\M_B$ is a sheaf of monoids on $B$, and $\alpha$ is a map of monoids, where $\calO_B$ is equipped with the multiplicative monoid structure, and where we assume that $\alpha$ induces an isomorphism $\alpha^{-1}(\calO_B^\times) \to \calO_B^\times$. We write $\ghost_B \coloneqq \M_B/\alpha^{-1}(\calO_B^\times)$; this is called the \emph{ghost sheaf} or \emph{characteristic sheaf}. We write the monoid operation on $\M_B$ as multiplication, and that on $\ghost_B$ as addition.  Recall that a monoid $M$ is called \emph{saturated} if the natural map $M \to M^\gp$ to its \emph{groupification} is injective, and if, for every $n \in \bb Z_{\ge 1}$ and $g \in M^\gp$ with $ng \in M$, we have $g \in M$.  A log structure is called saturated if all its stalks are saturated. We work throughout only with fine saturated log structures (log structures admitting charts; see \cite[Section~III.1.2]{Ogus2018Lectures-on-Log} by finitely generated saturated monoids). 

If $\beta \in \Gamma(B, \ghost_B^\gp)$, then the preimage of $\beta$ in the short exact sequence
\begin{equation}\label{eq:stdsequence}
1 \lra \calO_B^\times \lra \M_B^\gp \lra \ghost_B^\gp \lra 1
\end{equation}
is a $\bb G_m$-torsor, which we denote by $\calO_B^\times(\beta)$. We write $\calO_B(\beta)$ for the associated line bundle (see the appendix for our sign convention here).

Following \cite{Kato2000Log-smooth-defo}, the formal definition of a \emph{log curve} is a morphism of log schemes\footnote{The reader concerned about the case $g=1$, $n=0$ should rather take log algebraic spaces.} $\pi\colon X \to B$ that is proper, saturated, log smooth, and  has geometric fibers which are reduced and of pure dimension 1. This definition is rarely important to us, so rather than explicating the terms involved, we present a crucial structure result (to be found in \cite[Section~1.8]{Kato2000Log-smooth-defo}). If $\pi\colon X \to B$ is a log curve, then the underlying map of schemes is a prestable curve, and if $x$ is a geometric point of $X$ mapping to a geometric point $b$ of $B$, then exactly one of the following three cases holds: 
\begin{enumerate}
    \item $x$ is a smooth point of $X$, and the natural map $\ghost_{B,b} \to \ghost_{X,x}$ is an isomorphism; 
    \item $x$ is a smooth point of $X$, and $\ghost_{X, x} \cong \ghost_{B, b} \oplus \bb N$ with the natural map $\ghost_{B, b} \to \ghost_{X, x}$ corresponding to the inclusion of the first summand (in this case we say $x$ is a marked point, and we choose a total ordering on our markings to be compatible with the standard definition of marked prestable curves);
    \item $x$ is not a smooth point of the fiber $X_b$ (\textit{i.e.}~$x$ is a node), and there exist a unique element $\delta_x \in \ghost_{B,b}$ and an isomorphism 
    \begin{equation}\label{eq:smoothing_param_1}
        \ghost_{X,x} \cong \left\{(u,v) \in \ghost_{B,b}^2 \textrm{ such that } \delta_x \textrm{ divides } u-v\right\}. 
    \end{equation}
\end{enumerate}
We warn the reader that the ghost sheaf $\ghost_X$ does \emph{not} fully determine the log structure; the units contain important additional information. 

We write $\mathfrak M$ for the fibered category over $\cat{LogSch}$ whose objects are log curves $X/B$, with the projection taking $X/B$ to $B$. This is representable by an algebraic stack with log structure, see \cite[Appendix A]{Gross2013Logarithmic-Gro}, generalizing the construction of \cite{Kato2000Log-smooth-defo} in the stable case. As shown in those references, the underlying algebraic stack of $\mathfrak M$ is naturally isomorphic to the stack of prestable curves. The stack $\mathfrak M$ naturally contains all $\Mbar_{g,n}$ as open substacks, by equipping a stable curve $X/B$ with its basic log structure (see \cite{Kato2000Log-smooth-defo,Gross2013Logarithmic-Gro}), equivalently, with the log structure coming from the boundary divisor. 

Given a log scheme, we define
\begin{equation*}
\mathbb G_m^{\trop}(B) \,\coloneqq\, \Gamma\left(B, \ghost_B^{\gp}\right),
\end{equation*}
which we call the tropical multiplicative group.  It can naturally be extended to a presheaf $\mathbb G_{m,B}^{\trop}$ on the category $\cat{LogSch}_B$ of log schemes over $B$,
and admits a log smooth cover by log schemes (with subdivision $[\bb P^1/\bb G_m]$); see \cite[Section 4.1]{MarcusWiseLog}.

\begin{df}\label{def:rub}
We define $\cat{Rub}$ as the stack in groupoids over $\mathfrak M$, with objects being tuples
\bes
\left(\pi \colon X \to B, \, \beta\colon X \to \bb G_{m,B}^\trop\right)
\ees
with $X/B$ a log curve, satisfying two conditions on each strict geometric fiber:
\begin{enumerate}
\item\label{d:r-1}
The image of $\beta$ is fiberwise totally ordered,\footnote{Here we mean that for any two elements in the image of $\beta$, one of their differences is contained in~$\ghost_B$.} with largest element~$0$.
\item\label{d:r-2}
Writing $R$ for the stack obtained from $\bb G_m^\trop$ by subdividing at the image of $\beta$, we require that the fiber product $X \times_{\beta, \bb G_m^\trop} R$ is a log curve.
\end{enumerate}
The morphisms are defined by pullback.
\end{df}

Over a given geometric point of $B$, write $N+1$ for the cardinality of the image of $\beta$; since the latter is totally ordered, there is a unique isomorphism~$\tau$ of totally ordered sets between the image of $\beta$ and $\{0,-1,\ldots, -N\}$. The composition \be \label{eq:defell} \ell \coloneqq \tau \circ \beta \ee is then called the \emph{normalized level function} associated with~$\beta$.

\begin{remark}
This definition will be unpacked in~\ref{sec:unpacking_rub_definition}, but for now we make a couple of remarks on how it differs from that given in Marcus--Wise \cite{MarcusWiseLog}. Firstly, they declare the image of~$\beta$ to have \emph{smallest} element 0; this makes no material difference, and the reason for our choice of conventions is explained in the appendix. 

More significantly, condition~\eqref{d:r-2} is not stated by Marcus and Wise. However, it is assumed, for example in datum (R1) in Section 5.5 of their paper. Most of their results go through without this condition, but it is necessary for making a connection to the spaces of rubber maps of Li, Graber--Vakil, \textit{etc.}, and is also necessary for the comparison results in the present paper.

In fact, dropping condition~\eqref{d:r-2} (and thus passing to the space of Marcus and Wise) is exactly the same as taking the coarse moduli space of $\cat{Rub}$ relative to $\mathfrak M$. We write $\cat{Rub}^{\MW}$ for this space. Combining Theorem~4.3.2 and Proposition 5.1.2 of \cite{MarcusWiseLog} shows that the space $\cat{Rub}^{\MW}$ is an algebraic space over the relative Picard stack over $\mathfrak M$. However, because the line bundle $\ca O(\beta)$ is canonically trivial along the locus where $\beta$ takes value 0, the map from $\cat{Rub}^{\MW}$ to the relative Picard stack factors via the relative Picard space, so that $\cat{Rub}^{\MW}$ is an algebraic space over $\mathfrak M$. On the other hand, $\cat{Rub} \to \cat{Rub}^{\MW}$ is a root stack (see the proof of~\ref{thm:rub_smooth} for more details), and so the relative coarse space of $\cat{Rub}$ is exactly $\cat{Rub}^{\MW}$.
\end{remark}

\begin{theorem}[\textit{cf.} \cite{MarcusWiseLog}]
  The category $\cat{Rub}$ is a log algebraic stack locally of finite
  presentation.
\end{theorem}

Marcus and Wise prove this in the absence of condition~\eqref{d:r-2} above, but imposing this condition simply corresponds to a root stack construction, and does not affect the result. One benefit of imposing condition~\eqref{d:r-2} is the following theorem, which did not hold for the version of $\cat{Rub}$ considered by Marcus and Wise.

\begin{theorem}\label{thm:rub_smooth}
The algebraic stack $\cat{Rub}$ is smooth. 
\end{theorem}
 
The proof of~\ref{thm:rub_smooth} will be given in~\ref{sec:smoothness_of_Rub}.

Given $\beta \in \ghost_X^\gp(X)$, then taking the preimage in the standard exact sequence~\ref{eq:stdsequence} applied to $X$ yields the line bundle $\ca O_X(\beta)$; in other words, it yields an Abel--Jacobi
map
\bes
\aj\colon \cat{Rub} \lra \Picabs
\ees
to the Picard stack of the universal curve over $\mathfrak M$ (the stack of pairs $(X/B, \ca F)$, where $X/B$ is a log curve and~$\ca F$ is a line bundle on $X$).  One of the main results of \cite{MarcusWiseLog} is that the composite of this Abel--Jacobi map with the forgetful map $\Picabs \to \on{Pic}$ to the relative Picard space is proper.

\begin{df}
Write $n$ for the locally constant function on $\mathfrak M$ giving the number of markings. Then there is an \emph{outgoing slopes} map
\bes
\cat{Rub} \lra \bb Z^n
\ees
sending a point $(X/B, \beta)$ to the outgoing slopes of $\beta$, \textit{i.e.}~the values of $\beta$ in the groupifications of the stalks $\ghost_{X/B}(z_i) \coloneqq \ghost_X(z_i) / \ghost_B(\pi(z_i)) = \bb N$ at the markings.

Given a tuple $\mu = (m_1, \ldots, m_n)$ of integers, we define $\cat{Rub}_\mu$ to be the open-and-closed substack of $\cat{Rub}$ where the log curve has $n$~markings and the outgoing slopes are given by~$\mu$.
\end{df}
Note that the forgetful map from $\cat{Rub}_\mu$ to the locus in $\mathfrak M$ where the curves carry exactly $n$ markings is birational (it is an isomorphism over the locus of smooth curves); in particular, if we fix a genus and a number of markings, then $\cat{Rub}_\mu$ is connected.

Writing $d\coloneqq \sum_{i=1}^n m_i$, the image of $\cat{Rub}_\mu$ under the Abel--Jacobi
map~$\aj$ lands in the connected component $\Picabs^d$ of $\Picabs$ consisting of line bundles of (total) degree $d$.

\begin{remark}
In fact one can show that the map $\cat{Rub}_\mu \to \mathfrak M$ is not only birational onto the locus in $\mathfrak M$ with $n$ markings but also \emph{log \'etale}. 
This is a type of map basically consisting of an iterated blowup of boundary strata, followed by root constructions\footnote{We cannot assume that the degrees of these roots are invertible in the base ring, so this map should probably not be considered to be log \'etale outside of characteristic zero (though conventions in the literature differ).} on some of these strata, and then followed by taking an \'etale map. For the details, we refer the reader \textit{e.g.}~to the paper \cite{HMPPS}, where such morphisms are used extensively. An important point there is that they can be described uniquely by an (incomplete) subdivision of the tropicalization of $\mathfrak M$. While again we do not explain the details, one consequence is that one can obtain a smooth local model of the morphism $\cat{Rub}_\mu \to \mathfrak M$ by the toric map induced via some explicit subdivision of a cone.

In~\ref{fig:Rub_subdivision_fan}
we use this to illustrate the importance of condition~\eqref{d:r-2} in~\ref{def:rub}. For this, consider a point of $\mathfrak M$ where the curve has genus zero and the stable graph $\Gamma$ has three vertices and two edges $e_1, e_2$ as illustrated. Assume that each vertex carries one marking and that $\mu$ is chosen so that the unique slopes of a piecewise linear function on the edges are $1,2$ for $e_1, e_2$ (see~\ref{df:PL} for a discussion of piecewise linear functions). 

Then the tropicalization of $\mathfrak M$ contains a cone $\sigma_\Gamma = \mathbb{R}_{\geq 0}^2$ parameterizing the ways of putting lengths $\ell_1, \ell_2$ on the two edges. Depending on which of the quantities $\ell_1$ or $2 \ell_2$ is greater, a piecewise linear function on $\Gamma$ with the given slopes will take a larger value on either $v_2$ or $v_1$.  Then the smooth local picture of $\cat{Rub}_\mu \to \mathfrak M$ is given by the map of toric varieties associated to the subdivision of $\sigma_\Gamma$ along the ray spanned by $(\ell_1, \ell_2) = (2,1)$.

However, there is a subtlety: for the standard integral structure (black dots), the upper cone is simplicial, but not smooth.
Indeed, the primitive generators $(0,1)$, $(2,1)$ of its rays form a rational basis, but not an integral basis. Hence, the toric variety associated to this cone has a singularity, which would contradict~\ref{thm:rub_smooth}. And indeed, this is precisely what happens for the variant of $\cat{Rub}$ defined by omitting condition~\eqref{d:r-2} from~\ref{def:rub}. Putting this condition forces us to adjoin the element 
$\ell_1/2$
to the dual of the lattice on the upper cone. Correspondingly, on that upper cone we take the sublattice of all points $(\ell_1, \ell_2)$ with $\ell_1$ an even integer (depicted by circled dots).\footnote{Note that in contrast to the toric situation, not all cones in the tropicalization of $\cat{Rub}$ lie in the same ambient vector space with integral structure, so that it is possible to change this integral structure on different cones of the tropicalization.} Then the new ray generators are $(0,1/2)$, $(1,1/2)$, which indeed form a basis of the integral structure $\mathbb{Z} \oplus (1/2) \mathbb{Z}$, so that $\cat{Rub}$ is smooth as claimed.

\usetikzlibrary{shapes.misc}

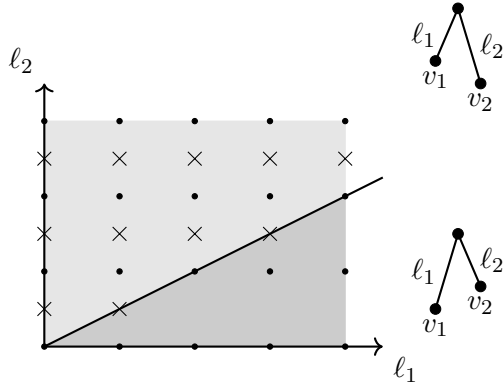
\begin{figure}
\[
\begin{tikzpicture}
\tikzset{cross/.style={cross out, draw=black, minimum size=2*(#1-\pgflinewidth), inner sep=0pt, outer sep=0pt},
cross/.default={3pt}}

\filldraw[black!20] (0,0) -- (4,0) -- (4,2) -- (0,0);
\filldraw[black!10] (0,0) --  (4,2) -- (4,3) -- (0,3) -- (0,0);

\draw[thick, ->] (0,0) -- (4.5,0) node[below right]{$\ell_1$};
\draw[thick, ->] (0,0) -- (0,3.5) node[above left] {$\ell_2$};
\draw[thick] (0,0) -- (4.5,2.25);

 \foreach \x in {0,...,4}
    \foreach \y in {0,...,3} 
       {\filldraw[black] (\x, \y) circle (1pt);}

\foreach \p in {(0,1),(0,2),(0,3),(2,1),(2,2),(2,3), (4,2), (4,3)}
 {\draw \p circle (2.3pt) {};}

\filldraw (5.5,1.5) circle (2pt) {};
\filldraw (5.2,0.5) circle (2pt) node[below]{$v_1$};
\filldraw (5.8,0.8) circle (2pt) node[below]{$v_2$};
\draw[thick] (5.5,1.5) to node[midway, left] {$\ell_1$} (5.2,0.5);
\draw[thick] (5.5,1.5) to node[midway, right] {$\ell_2$} (5.8,0.8);

\filldraw (5.5,4.5) circle (2pt) {};
\filldraw (5.2,3.8) circle (2pt) node[below]{$v_1$};
\filldraw (5.8,3.5) circle (2pt) node[below]{$v_2$};
\draw[thick] (5.5,4.5) to node[midway, left] {$\ell_1$} (5.2,3.8);
\draw[thick] (5.5,4.5) to node[midway, right] {$\ell_2$} (5.8,3.5);
\end{tikzpicture}
\]
    \caption{Subdivision associated to the drawn stable graph, with slope $1$ at edge $e_1$ (of length $\ell_1$) and slope $2$ at edge $e_2$ (of length $\ell_2$).}
    \label{fig:Rub_subdivision_fan}
\end{figure}
\end{remark}

\subsection{Logarithmic rubber differentials} \label{sec:rubdiffspace}

The stack $\cat{Rub}$ is in some sense the universal space of logarithmic rubber maps. In this section we specialize to the case of logarithmic rubber differentials. For this we fix $g$, $n$ and write $X_{g,n}/\Mbar_{g,n}$ for the universal curve, with markings $\bfz = (z_1,\ldots,z_n)$. Fix a tuple~$\mu = (m_1, \dots, m_n)$ of integers such that $d= \sum_{i=1}^n m_i = 2g-2$.
We define a line bundle on the universal curve $X_{g,n}$ over $\Mbar_{g,n}$ by
the formula
\bes
\calL \,\coloneqq\, \calL_\mu \,\coloneqq\,
\omega_{X_{g,n}/\Mbar_{g,n}}\left(-\sum_{i=1}^n m_i z_i\right),
\ees
where $\omega = \omega_{X_{g,n}/\Mbar_{g,n}}$ is the relative dualizing sheaf of $X_{g,n}\to\Mbar_{g,n}$. Then $\calL$ induces a morphism
\bes
\phi_{\calL}\colon \Mbar_{g,n} \lra \Picabs . 
\ees

\begin{df}\label{def:rub_L}
We define the \emph{space of logarithmic rubber differentials} to be
\begin{equation}\label{eq:rub_L_def}
\cat{Rub}_\calL \coloneqq \cat{Rub}_{\ul 0} \times_{\Picabs, \phi_{\calL}} \Mbar_{g,n}.
\end{equation}
\end{df}

\begin{remark}
If we had taken the fiber product over the relative Picard space (instead of the Picard stack), we would have obtained the projectivized space $\PP(\cat{Rub}_\calL)$. This is the approach taken in \cite{MarcusWiseLog,BHPSS}, as the space $\PP(\cat{Rub}_\calL)$ is what is needed for the study of the double ramification cycle. 
\end{remark}
\begin{remark}\label{rk:no_leg_weights}
There are two equivalent descriptions of the rubber differential space as
\bes
  \cat{Rub}_\calL \= \cat{Rub}_{\ul 0} \times_{\Picabs, \phi_\calL} \Mbar_{g,n}
  \= \cat{Rub}_\mu \times_{\Picabs, \phi_{\omega}} \Mbar_{g,n}. 
\ees
\end{remark}

\subsection{Local description}\label{sec:unpacking_rub_definition}

In what follows we will make the definition of the space $\cat{Rub}$ more explicit for log curves over `sufficiently small' bases, more precisely, for \emph{nuclear} log curves as defined in \cite{Holmes2020Models-of-Jacob}. This is a slight refinement of asking for the base to be atomic (in the sense of \cite{AbraWise}), and is needed because a log curve even over a point does not have a well-defined dual graph unless the residue field is sufficiently large. We omit the details of the definition of a nuclear log curve, mentioning only the key properties we use:
\begin{enumerate}
\item\label{Kprop-1}
  For any family of log curves $X/B$ with $B$ locally of finite type, there exists a strict\footnote{A map $f\colon X \to Y$ of log schemes is \emph{strict} if the log structure on $X$ is the pullback of the log structure on $Y$. In particular, the strict \'etale topology on log schemes reflects very closely the usual \'etale topology on schemes. } \'etale cover $\bigsqcup_{i\in I} B_i \to B$ such that each $X \times_B B_i \to B_i$ is nuclear.
\item\label{Kprop-2}
For $X/B$ a nuclear log curve and for any $b \in B$, the curve $X_b$ has a well-defined dual graph $\Gamma_b$, with edges labelled by non-zero elements of $\ghost_{B,b}$; we denote the \emph{label} (also called
\emph{length}) of $e$ by $\delta_e$; this was denoted by $\delta_x$ in~\ref{eq:smoothing_param_1}. If $\delta_e' \in \M_B(B)$ is a lift of $\delta_e$, then $\alpha(\delta_e') \in \ca O_B(B)$ is a \emph{smoothing parameter} for $e$, in the sense that $X$ can be described locally around the corresponding point by an equation $uv = \alpha(\delta_e')$. 
The stalk of $\ghost_{X}$ at the corresponding node $q$ of the fiber over $b\in B$  
is given by
\begin{equation}\label{eq:ghost-at-singular_point}
\ghost_{X,q} = \left\{(u, v) \in \ghost_{B, b} \oplus  \ghost_{B, b} \text{ such that } \delta_e \mid (u-v)\right\}. 
\end{equation}
\item \label{Kprop-3}
For $X/B$ nuclear, the base $B$ has a unique closed stratum,\footnote{Every log scheme comes with a decomposition into locally closed subschemes (called \emph{strata}), where the ghost sheaf is locally constant.} and, for any $b$ in that closed stratum, the restriction gives an isomorphism $\Gamma(B, \ghost_{B}) \isom \ghost_{B,b}$.

\item\label{Kprop-4} If $X/B$ is nuclear and $b$, $b' \in B$, with $b$ in the closed stratum, there is a natural identification (of labelled graphs) of $\Gamma_{b'}$ with the graph obtained from $\Gamma_b$ by mapping every label to $\ghost_{B,b'}$, and then contracting all edges that are labelled by 0. We often abuse notation by writing $\ghost_B\coloneqq\ghost_{B,b}$ (for $b$ in the closed stratum) in place of $\Gamma(B, \ghost_{B})$. We often write $\Gamma$ for the graph over any point in the closed stratum, which comes with an $\ghost_B$-metric. 
\end{enumerate}
If $B$ is the spectrum of a noetherian strictly Henselian local ring with atomic log structure (for example, if $B$ is the spectrum of a separably closed field), then by \cite[Lemma 3.40]{Holmes2020Models-of-Jacob} any log curve $X/B$ is nuclear.

Let $X/B$ be a nuclear log curve. Let $b \in B$ be a point in the closed stratum, with associated dual graph $\Gamma$ with \emph{vertex set}
$V = V(\Gamma)$, {\em set of half-edges} $H=H(\Gamma)$ (including legs), and {\em set of non-leg half-edges} $H'=H'(\Gamma)$.

\begin{df}\label{df:PL}
A \emph{piecewise linear} (\emph{PL}\,) \emph{function} on $X/B$ is an element of $\Gamma(X,\ghost_X^\gp)$.

A \emph{combinatorial PL function} on $X/B$ consists of the data:
\begin{enumerate}
\item
a function $\beta'\colon V(\Gamma) \to \ghost_{B, b}^\gp$ (the \emph{values} on the vertices), and
\item a function $\kappa\colon H'(\Gamma) \to \bb Z$ (the \emph{slopes} on the non-leg\footnote{In this paper we do not include slopes on the legs, as we are interested only in the case where these slopes are equal to $0$ (since we work throughout with $\cat{Rub}_{\ul 0}$). Recall that, as discussed in~\ref{rk:no_leg_weights}, we have moved the data of the zeros and poles into the line bundle $\ca L_\mu$. } half-edges)
\end{enumerate}
such that if $h_1$ and $h_2$ are half-edges forming an edge $e$, with $h_i$
attached to vertex $v_i$, we have
\bes
\kappa(h_2) \delta_e \= \beta'(v_2) - \beta'(v_1)
\ees
(so that in particular $\kappa(h_1)+\kappa(h_2)=0$). 
Edges of~$\Gamma$ with slope zero (that is, where both half-edges have slope zero) are called \emph{horizontal}; all the other edges of~$\Gamma$ are called \emph{vertical}.
\end{df}

\medskip
We want to show that these two types of PL functions are in natural bijection.  First, we construct a combinatorial PL function from any PL function.  At generic points $\eta$ of $X_b$, there is a natural isomorphism $\ghost_{B, b} = \ghost_{X,\eta}$, so the section $\beta \in H^0(X,\ghost_X^\gp)$ determines a function $\beta'\colon V \to \ghost_{B, b}^\gp$.  To complete the definition of~$\kappa$, we first show the following. 

\begin{lemma} \label{le:divisibility}
 If\, $h_1$ and $h_2$ are half-edges forming an edge $e$, with $h_i$ attached to vertex $v_i$, then for the function $\beta'$ constructed from $\beta$ as above, the value $\beta'(v_2) - \beta'(v_1)$ is an integer multiple of\, $\delta_e$.
\end{lemma}

\begin{proof} 
This follows from~\ref{eq:ghost-at-singular_point} and the fact that the images of $\beta$ under the two projections to $\ghost_{B, b}^\gp$ are exactly given by $\beta'(v_1)$ and $\beta'(v_2)$.
\end{proof}

In the notation of \Cref{le:divisibility}, we can then \emph{define}
\begin{equation} \label{eq:defkappa}
\kappa(h_2) \,\coloneqq\, \frac{\beta'(v_2) - \beta'(v_1)}{\delta_e}
\end{equation}
(which is unique because $\ghost_{B,b}$ is torsion-free). This accomplishes one direction of the following lemma.

\begin{lemma} \label{le:PLbijection}
The above construction induces a bijection between the set of\, PL functions and the set of combinatorial PL functions.
\end{lemma}

\begin{proof} Let $\beta'$ be a combinatorial PL function; we build a PL function $\beta$ giving the inverse image of $\beta'$ under the construction above. If $x$ is a smooth point of $X_b$, then $\ghost_{X, x} = \ghost_{B, b}$, and we define the value of $\beta$ at $x$ to be $\beta'(v)$, where $v$ corresponds to the irreducible component of $X_b$ containing $x$. The presentation~\ref{eq:ghost-at-singular_point} makes it clear that there is a unique way to extend this section to all non-smooth points $x \in X_b$. For any other point $b' \in B$, the combinatorial PL function can naturally be transferred (using property (4) of the definition of a nuclear log curve) to the fiber $X_{b'}$, and we repeat the above argument to give a PL function on $X_{b'}$. These then fit together to define a global PL function on $X/B$.
\end{proof}

Our concrete local description of $\cat{Rub}$ is now given by the next
proposition. 

\begin{proposition}\label{prop:rub_translation}
For $X/B$ nuclear and $b\in B$ in the closed stratum, let $V$ be the vertex set of the associated dual graph of\, $b$. Then there is a natural bijection between the set of\, $X/B$-points of\, $\cat{Rub}_{\ul 0}$ $($i.e.~the set of maps $B \to \cat{Rub}_{\ul 0}$ lying over $X/B)$
and the set of maps
\begin{equation}\label{eq:PL_literal}
\beta'\colon V \lra \ghost_{B, b}^\gp
\end{equation}
satisfying the following conditions:
\begin{enumerate}
\item\label{p:r_t-1} The divisibility condition $\delta_e \mid \beta'(v_2) - \beta'(v_1)$ holds at every edge $e$ in $\Gamma_b$ connecting vertices $v_1,v_2 \in V$.
\item\label{p:r_t-2} The image of $\beta'$ is a totally ordered subset of\, $\ghost_{B, b}^\gp$ with largest element being $0$.
\item\label{p:r_t-3} For every edge $e$ connecting vertices $v_1$ and $v_2$, with slope $\kappa_e$ $($defined as the absolute value of\, {\rm\ref{eq:defkappa}}$)$, and for every $y \in \on{Image}(\beta')$ with $\beta'(v_1) < y < \beta'(v_2)$, the monoid $\ghost_{B, b}$ contains the element $\frac{y - \beta'(v_1)}{\kappa_e}$. 
\end{enumerate}
\end{proposition}

\begin{proof}
Conditions~\eqref{p:r_t-1}  and~\eqref{p:r_t-2} are translations of point~\eqref{d:r-1} of~\ref{def:rub}. Condition~\eqref{p:r_t-3} corresponds to point~\eqref{d:r-2} of~\ref{def:rub}, as explained in \cite[Section~6.2]{BHPSS}.
\end{proof}

\begin{remark}
If $\beta_1'$ and $\beta_2'$ are combinatorial PL functions with the same slopes $\kappa_e$, then there exists an element $c \in \ghost_{B,b}^\gp$ such that $\beta_1' = \beta_2' + c$. In the definition of $\cat{Rub}$, we restrict to PL functions whose values are totally ordered and take maximum value $0$, and such functions are completely determined by the values of their slopes $\kappa$. 
\end{remark}

We would like to characterize in a similar spirit when a point of $\cat{Rub}$ lifts to $\cat{Rub}_\calL$. More concretely, this means describing explicitly the line bundle $\calO_X(\beta)$ associated to a PL function. The next lemma describes the \emph{restriction} of $\calO_X(\beta)$ to the irreducible components of the curve $X_b$ (in the case where $\beta$ comes from $\cat{Rub}_{\ul 0}$, \textit{i.e.}~has vanishing outgoing slopes). To describe the gluing between irreducible components would require us to get into quite a few more details of log geometry, and is not necessary for what we do in this paper. 

\begin{lemma}[\textit{cf.} {\cite[Lemma 2.4.1]{RSPWI}}] \label{lem:Obetarestriction}
 Let $Y$ be the normalization of an irreducible component of $X_b$, corresponding to a vertex $v$. For each half-edge $h$ attached to $v$, write $\kappa_h$ for the slope $($in the sense of\, {\rm\ref{eq:defkappa}}$)$ and $z_h \in Y$ for the associated preimage of a node of\, $X_b$. Then there is a canonical isomorphism 
 \bes
 \calO_X(\beta)|_{Y} \= \pi^*\calO_B(\beta'(v))\otimes_{\calO_Y} \calO_Y\left(\sum_h \kappa_h z_h\right).
 \ees
\end{lemma}
In particular, for a point $(X_b/b, \beta)$ of $\cat{Rub}_{\ul 0}$ to lie in $\cat{Rub}_{\calL}$, it is necessary (though not in general sufficient) to require that, on the normalization $Y$ of any irreducible component of $X_b$, there exists an isomorphism 
\begin{equation*}
    \calL_Y \,\cong\, \calO_Y\left(\sum_h \kappa_h z_h \right),
\end{equation*}
where the sum runs over all half-edges $h$ attached to $v$. 

\section{Generalized \msds} \label{sec:famGMS}

We recall basic notions from \cite{LMS}, in order to define the groupoids $\GSMS$ of \emph{simple generalized multi-scale differentials} and $\GMS$ of \emph{generalized multi-scale differentials}, where $\mu = (m_1, \ldots, m_n)$ is a tuple of integers with sum~$2g-2$. The adjective `generalized' refers to the fact that we do not impose the global residue condition.

\subsection{Enhanced level graphs}
\label{sec:tori}

The boundary strata of the stack of generalized multi-scale differentials are indexed by \emph{enhanced level graphs}. Such an enhanced level graph, typically denoted by~$\Gamma$, is the dual graph of a stable curve, with legs corresponding to the marked points, with a level structure (\textit{i.e.}~a weak full order, equality being permitted) on the set of vertices $V(\Gamma)$, and with enhancements $\kappa_e$, which are non-negative integers attached to the edges. The edges~$E(\Gamma)$ are grouped into the set of horizontal edges~$E^h(\Gamma)$ joining vertices at the same level, and the set of vertical edges~$E^v(\Gamma)$.  The enhancements are required to be zero precisely for horizontal edges.  We thus may consider an enhancement as a function
\bes
\kappa\colon H(\Gamma) \lra \bb Z
\ees
on the set of half-edges of~$\Gamma$, assigning~$\kappa_e>0$ to the upper half and $-\kappa_e<0$ to the lower half of a vertical edge, assigning zero to both halves of a horizontal edge, and letting $\kappa$ agree with $m_i$ at the legs of the graph.  We normalize the set of levels so that the top level is zero, and let $L(\Gamma)$ be the set of levels below zero, usually given by consecutive negative integers $L(\Gamma)=\{-1,\dots,-N\}$, where $N\coloneqq |L(\Gamma)|$, so that we typically use the \emph{normalized level function}
\begin{equation}\label{eq:normlev}
\ell\colon V(\aG) \longtwoheadrightarrow \{0,-1,\dots,-N\}.
\end{equation}
Occasionally, we use $L^\bullet(\aG)$ for the set of all levels including the zero level. In what follows  we will only consider enhancements that are
{\em admissible} in the sense that the degree equality
\be \label{eq:admenhancement}
\deg(v) \,\coloneqq\, \sum_{j\mapsto v} m_j + \sum_{e \in E^+(v)}  (\kappa_e-1)
- \sum_{e \in E^-(v)} (1+\kappa_e) - h(v) \= 2g(v)-2
\ee
holds, where $j\mapsto v$ means the leg of order $m_j$ is attached to the vertex $v$; \textit{i.e.}~the first sum goes over all legs attached to $v$,
where $E^+(v)$ (resp.\ $E^-(v)$) is the set of vertical edges whose upper
(resp.\ lower) end is the vertex~$v$ (we often write $e^+ = v$, 
  resp.\ $e^- = v$, to express this adjacency),
and $h(v)$ is the number of 
horizontal half-edges adjacent to $v$. 

Enhanced level graphs come with two kinds of undegeneration maps. First, for any subset $I = \{i_1,\dots,i_n\}$ of $\{-1,\dots,-N\}$, there is the vertical undegeneration map $\delta_{i_1,\dots,i_n}$, a map of graphs that contracts all vertical edges except those that go from the level at or above $i_{k}+1$ to a level at or below $i_{k}$, for some $i_k \in I$. Especially important among those are the two-level undegenerations~$\delta_i$, which contract all vertical edges except those that cross a level passage above~$i$, \textit{i.e.}~go from a vertex at level $i+1$ or above to a vertex at level $i$ or below.  Second, for $S \subset E^h(\Gamma)$ there is the horizontal undegeneration maps $\delta^h_S$ that contract all the horizontal edges except those in~$S$. An \emph{undegeneration} of a level graph is a composition of a vertical and a horizontal undegeneration.  Undegenerations determine the adjacency of boundary strata of the space of multi-scale differentials.

\subsection{Prong-matchings}
\label{sec:PM}

Let $(X,\omega)$ be a smooth complex curve with a meromorphic 1-form.  If a differential~$\omega$ has a zero of order $m \geq 0$ at~$q \in X$, then there exists a local coordinate (and, in fact, there are $m+1$ such choices)~$z$ on~$X$ centered at~$q$ such that locally in this coordinate $\omega=z^mdz$; similarly, for a pole of order $m\le -2$ at~$q\in X$, one can find a local coordinate such that $\omega=(z^m+r/z)dz$. Given such a local coordinate, a ({\em complex}\,) {\em prong} is one of the $2|m+1|$ vectors in $T_qX$ of the form $\zeta^j\tfrac{\partial}{\partial z}$ in this local coordinate~$z$, where $\zeta$ is a primitive root of unity of order $2|m+1|$; see \cite[Definition~5.4]{LMS}. In what follows we will mostly care about the set of $P_q^{\rm out}$ of \emph{outgoing} prongs at the zeros of a differential, that is, the set of $m+1$ prongs
there where the exponent~$j$ of $\zeta^j$ is even, and the set $P_q^{\rm inc}$ of \emph{incoming} prongs at the poles of the differential, that is, those $|m+1|$ prongs there where the exponent~$j$ is odd.

Now let $X$ be a stable curve with a node~$q$ corresponding to a vertical edge $e\in E^v(\Gamma)$ where two components of $X$ meet, and suppose these components $X_1$ and $X_2$ come with differential forms $\omega_1$ and~$\omega_2$ having a zero and a pole, respectively, at the respective preimages $q^+\in X_1$ and $q^-\in X_2$ of~$q$.  A ({\em local}\,) {\em prong-matching} at the node~$q$ is a cyclic order-reversing bijection $\sigma_e\colon P^{\rm in}_{q^-} \to P^{\rm out}_{q^+}$ between the incoming prongs at~$q^-$ and the outgoing prongs at~$q^+$.

Now let $(X,\bfz,\Gamma,\bfomega)$ be a pointed stable curve with an enhanced level graph~$\Gamma$, and let $\bfomega = (\omega_{(i)})_{i \in L^\bullet(\Gamma)}$ be a \emph{twisted differential of type~$\mu$} compatible with~$\Gamma$, possibly except for the global residue condition.  Following \cite{BCGGM1}, this means a collection of meromorphic differentials~$\omega_v$ for each vertex~$v$, vanishing to order~$m_i$ at each of the marked points~$z_i$, vanishing to order~$\kappa(h)-1$ at the preimages of nodes associated to the half-edges~$h\in H'(\Gamma)$,  and such that the residues at the two sides of a horizontal node add up to zero.  Grouping objects level-wise, we denote by $\omega_{(i)}$ the tuple of differentials~$\omega_v$ for all vertices~$v$ on level~$i$.

Given a twisted differential, we have defined above local
prong-matchings
for each vertical edge. Packaging such a choice for each vertical edge
$e \in E^v(\Gamma)$, we call the collection $\bfsigma = (\sigma_e)_{e \in E^v(\Gamma)}$
a {\em global prong-matching}.

\medskip
There is an alternative viewpoint on prong-matchings, which can be generalized to germs of families $\calX \to B$, where a node~$q$ corresponding to an edge~$e$ in the dual graph of the special fiber persists over the base.  In the normalization of the family, there are two components $X^\pm$ (as the edge is vertical, necessarily $X^+\ne X^-$) that admit sections~$q^\pm$ that specify the two preimages of the node~$q$.  We let
\be \label{eq:defNe}
\calN_e^\vee \,\coloneqq\, (q^+)^*\omega_{X^+}\otimes (q^-)^*\omega_{X^-}.
\ee
A \emph{local prong-matching} (see \cite[Definition~5.6]{LMS}) is then a section $\sigma_e$ of $\calN^\vee_e$ such that for any pair $(v^+,v^-)$ of an incoming and an outgoing horizontal prong, the equation $\sigma_e(v^+\otimes v^-)^{\kappa_e} = 1$ holds. To see the equivalence, given $\sigma_e$, we assign to $v^-$ the prong $v^+$ given by the condition $\sigma_e(v^+\otimes v^-) = 1$. Conversely, given a bijection $s\colon P^{\rm in}_{q^-} \to P^{\rm out}_{q^+}$ of incoming and outgoing prongs, we define $\sigma_e\in \calN_e^\vee$ by setting $\sigma_e(p\otimes s(p))=1$ for any $p\in P^{\rm in}_{q^-}$. The fact that $s$ is order  reversing implies that this $\sigma_e\in \calN_e^\vee$ is well defined, which justifies using the same notation $\sigma_e$ for both viewpoints on a prong-matching. A \emph{global prong-matching} is a collection of local prong-matchings for each persistent node (as will be defined formally in~\ref{sec:germfam}) in the family.

We give another reformulation that eliminates the dependence on the choice
of a preferred (`horizontal') direction. Let $U^\pm$ be neighbourhoods
of the points $q^\pm$ in the
normalization of~$\calX$. Suppose the edge~$e$ joins level~$i$ to the lower level~$j$.
Then $\omega_{(i)}$ extends uniquely to a section of $\omega_{U^+/B}(-(\kappa_e-1) q^+)$
and $\omega_{(j)}$ to a section of $\omega_{U^-/B}((\kappa_e+1) q^-)$.
Restricting these to $q^+$ and $q^-$, respectively, yields canonical elements
\bes
\tau^+ \in \omega_{U^+/B}(-(\kappa_e-1) q^+)|_{q^+} = T_{q^+}^{\otimes -\kappa_e} \quad
\text{and} \quad
\tau^- \in \omega_{U^-/B}((\kappa_e+1) q^-)|_{q^-} = T_{q^-}^{\otimes  \kappa_e}
\ees
(where we use the residue isomorphism for the equalities).
We define
\bes
\tau_e \coloneqq (\tau^+)^{-1} \otimes (\tau^-) \in \left(T_{q^+} \otimes T_{q^-}\right)^{\otimes \kappa_e}
\= \calN_e^{\otimes \kappa_e}.
\ees

\begin{lemma}\label{lem:prong_matching_comparison}
In the notation of the previous definition, let $v^+$ and $v^-$ be some horizontal
prongs at $e$. Then $(v^+ \otimes v^-)^{\otimes \kappa_e} \in \calN_e^{\otimes \kappa_e}$
is independent of the choice of prongs and of the direction to be called horizontal, and we have
\begin{equation}
\tau_e \= (v^+ \otimes v^-)^{\otimes \kappa_e}.
\end{equation}
\end{lemma}
\begin{proof} 
For a fixed direction, the different choices of prongs~$v^+$ differ by $\supth{\kappa_e}$ roots of unity, and likewise for~$v^-$. Thus the formula for~$\tau_e$ implies that it does not depend on these prong choices. On the other hand, changing the direction from horizontal to direction~$\theta$ multiplies~$v^+$ by $e^{2\pi i\theta}$ and $v^-$ by $e^{-2\pi i\theta}$, and thus preserves $v^+\otimes v^-$. The equality is obvious, as can be seen by writing it out in any local coordinate that puts the differentials in normal form.
\end{proof}

This implies that the earlier definitions of prong-matching agree with the following.

\begin{df} \label{df:PMfinal}
  A \emph{local prong-matching} is a section $\sigma_e$ of $\calN^\vee_e$
  such that $\sigma_e^{\kappa_e}(\tau_e) = 1$.
\end{df}

\subsection{Level rotation tori}
\label{sec:LRT}

To an enhanced level graph, we associate some groups and algebraic tori.  The {\em level rotation group} $R_\Gamma \cong \ZZ^{L(\Gamma)}$ acts on the set of all global prong-matchings, where the $\supth{i}$ factor twists by one (\textit{i.e.}~multiplies~$\sigma_e$ by $e^{2 \pi i/\kappa_e}$) all prong-matchings associated to edges that cross the {\em $\supth{i}$ level passage}, a horizontal line above level~$i$ and below level $i+1$.\footnote{In this paper we index levels and all quantities indexed by them, such as $t_i$, $s_i$, $\delta_i$ below, by negative integers, as in \cite{LMS}, but contrary to several subsequent papers that use this compactification.}  The ({\em vertical}\,) {\em twist group} is the subgroup $\Tw[\Gamma] \subset R_\Gamma$ fixing all the prong-matchings under the above action. The level rotation group also acts (via its $\supth{i}$ component) on the set of prong-matchings of the two-level undegenerations $\delta_i(\Gamma)$.  We define the {\em simple twist group} $\sTw[\Gamma] \subset
\Tw[\Gamma] \subset R_\Gamma$ to be the subgroup that fixes each of the prong-matchings of each $\delta_i(\Gamma)$.

Let $\CC^{L(\Gamma)} \to (\CC^*)^{L(\Gamma)}$ be the universal covering
of the algebraic torus $(\CC^*)^{L(\Gamma)}$; we identify the level rotation group
$R_\Gamma \subset \CC^{L(\Gamma)}$ as the kernel of this covering. As a subgroup of the
level rotation group, the (simple) twist group acts on $\CC^{L(\Gamma)}$, and we define
the {\em level rotation torus} $T_\Gamma \coloneqq \CC^{L(\Gamma)}/\Tw[\Gamma] $,
together with its simple counterpart, the {\em simple level rotation
torus} $T^s_\Gamma \coloneqq \CC^{L(\Gamma)}/\sTw[\Gamma]$. See
\ref{sec:QuotTwist} for an example when these two tori differ.

Next we define the data that provide the model for the toroidal embedding of the
boundary inside the space of multi-scale differentials.
Since $\sTw[\Gamma] = \oplus_i \Tw[\delta_i(\Gamma)]$ has by definition
a direct sum decomposition level by level, the
simple level rotation torus comes with a natural level-wise identification
$T^s_\Gamma \cong (\CC^*)^{L(\Gamma)}$.  The
embedding $\CC^* \hookrightarrow \CC$ with respect to these coordinates
defines an embedding $T^s_\Gamma \hookrightarrow \ol{T}^s_\Gamma
\coloneqq\CC^{L(\Gamma)}$. We let
\be \label{eq:aidef}
a_i \coloneqq a_{\delta_i(\Gamma)} \coloneqq \lcm_{e \in \delta_i(\Gamma)} \kappa_e
\ee
be the least common multiple of the enhancements of the edges of~$\Gamma$ that persist in the two-level undegeneration $\delta_i(\Gamma)$.
Then $\sTw[\Gamma] \cong \oplus_i a_i \ZZ \subset R_\Gamma$.
Consequently, $T^s_\Gamma$ is a cover of the
original torus~$(\CC^*)^{L(\Gamma)}$, of degree $\prod_i a_i$. Finally,
we define the \emph{quotient twist group} to be
\be
K_\Gamma \,\coloneqq\, \Tw[\Gamma] /\sTw[\Gamma].
\ee
This group acts on $T_\Gamma^s$ with quotient $T_\Gamma$. In coordinates the
quotient map is given by
\begin{equation}\begin{aligned} \label{eq:simpletotorus}
(\CC^*)^{L(\Gamma)} & \,\lra\, (\CC^*)^{L(\Gamma)} \times (\CC^*)^{E^v(\Gamma)} \\
(q_i) &\,\longmapsto\,
(\{r_i\}_{i \in L(\Gamma)}, \{\rho_e\}_{e \in E^v(\Gamma)}) \,:=\, \left(\left\{q_i^{a_i}\right\}_{i \in L(\Gamma)}, \left\{ \prod_{i=\lbot}^{\ltop-1} q_i^{a_i/\kappa_e} \right\}_{e \in E^v(\Gamma)}\right),
\end{aligned}\end{equation}
where $q_i,r_i,\rho_e$ denote the coordinates on the corresponding tori, and we view ${T}_\Gamma \subset (\CC^*)^{L(\Gamma)} \times (\CC^*)^{E^v(\Gamma)}$
as cut out by the equations
\be \label{eq:rrho}
r_{\lbot} \cdots  r_{\ltop-1} \= \rho_e^{\kappa_e}
\ee
for each~$e$. The action of $K_\Gamma$ on $T^s_\Gamma$  extends to an action on the
closure $\ol{T}^s_\Gamma$, and we let $\ol{T}^n_\Gamma \coloneqq
\ol{T}^s_\Gamma/K_\Gamma$,
which is the normalization of the closure of ${T}_\Gamma \subset (\CC^*)^{L(\Gamma)}
\times (\CC^*)^{E^v(\Gamma)}$ inside $\CC^{L(\Gamma)} \times \CC^{E^v(\Gamma)}$.

All these tori come with their {\em extended versions}, denoted with an extra
dot (\textit{e.g.}, $T_\Gamma^\bullet$), that have an extra $\CC^*$-factor. This factor
will act on differentials of all levels simultaneously by multiplying all differentials by a common factor, and lead to the {\em projectivized}
version of the corresponding quotient functor.

\subsection{Controlled families of generalized multi-scale differentials}
\label{sec:germfam}

We will now use these constructions to define families of generalized multi-scale differentials for families $\calX \to B$ of curves over `small enough' base schemes. The general case will then be treated by gluing. The notion of `small enough' that we will use is that of \emph{controlled curves} from \cite[Section~2.6]{biesel2023fine}. This deviates slightly from the corresponding definition
in \cite[Section~11]{LMS}, where such families are merely defined as germs, and allows to \textit{a priori} say under which morphisms our families can be pulled back without having to pass to suitable representatives of germs.

We do not recall the definition of controlled curves in full, but recall the key properties that we will need. If $\calX \to B$ is a controlled curve, then for every~$p\in B$, the fiber $X_p$ has a well-defined dual graph~$\Gamma_p$. Moreover, there exists a \emph{controlling point} $b \in B$ together with \emph{smoothing parameters} $f_e \in \ca O_B(B)$ for every edge $e$ of $\Gamma_b$, such that $f_e$ vanishes exactly on the locus of $p\in B$ where the corresponding node persists in~$X_p$, and such that the family has the local form $u_e v_e = f_e$ in a neighbourhood of the corresponding node. The dual graph of $X_p$ is obtained from that of $X_b$ by contracting exactly those $e$ such that $f_e(p) \neq 0$. The function $f_e$ is unique up to multiplication by units in $\ca O_B(B)$; we write $[f_e]\in \ca O_B(B)/\ca O_B^\times(B)$ for the equivalence class of the smoothing parameter. Given any family of stable curves $\calX \to B$, over a locally noetherian base, there exists an \'etale cover $\bigsqcup_i B_i \to B$ with each $\calX \times_B B_i \to B_i$ controlled (see \cite[Lemma 2.6.9]{biesel2023fine}).
Our first step is to define `standard' open subsets of controlled families of curves where the collection of rescaled differentials are defined. Let $\mu = (m_1,\ldots, m_n)$ be a tuple of  integers with $\sum_{i=1}^n m_i = 2g-2$. If~$X$ is a stable $n$-marked curve with enhanced level graph structure on the dual graph, then for any level~$i$, the subcurve $X^{[i]} \coloneqq X_{(\le i)}\setminus(X_{(>i)} \cup \bfz^\infty)$ is open in $X$, where $X_{(\leq i)}$ is the subcurve at and below level~$i$, $X_{(>i)}$ is the subcurve above level~$i$, and $\bfz^\infty$ is the union of those markings with $m_j < 0$.

We now fix a stable $n$-marked controlled curve $\calX\to B$ and a controlling point $b \in B$, and suppose we are given an enhanced level graph structure on the dual graph of $X_b$. If $p \in B$ is another point, we have an edge contraction $\Gamma_b \to \Gamma_p$, which naturally induces on $\Gamma_p$ the structure of an enhanced level graph (see \cite[Section~5.1]{LMS} for details). In particular, the set~$L_b$ of levels of~$\Gamma_b$ naturally surjects onto the set~$L_p$ of levels of~$\Gamma_p$: we have $l_p\colon L_b \twoheadrightarrow L_p$. The reader may check that the following sets~$U_i\subset\calX$ are indeed open.

\begin{df}
Given a stable $n$-marked controlled curve $\calX\to B$ and a level $i\in L_b$, we define $U_i \subseteq \calX$, the \emph{standard open set at level~$i$}, to be the union over all $p \in B$ of $X_p^{[l_p(i)]}$.
\end{df}

We say that a node $e$ is \emph{persistent} in the family~$\calX$ if $f_e=0\in\calO_B$. If the dual graph~$\Gamma_b$ has been provided with an enhanced level graph structure, we say that a node $e$ is \emph{semi-persistent} if $f_e^{\kappa_e}=0$. The notion of prong-matchings makes sense for a persistent node~$q$.

\medskip
For our families of multi-scale differentials, we need to include an explicit choice of smoothing parameters~$f_e$ into our data. This can be achieved via a section of the partial compactification $\Tsnorm[\Gamma]$ of the simple level rotation torus. Indeed, given the coordinates $(r_i, \rho_e)$ on the torus closure from~\ref{eq:rrho}, a morphism $R^s\colon B\to \Tsnorm[\Gamma]$ determines for each vertical edge~$e$ a function $f_e\in \calO_{B}$, and for each level~$i$ a function $s_i\in \calO_{B}$, defined as the compositions $f_e= \rho_e\circ P \circ R^s$ and $s_i = r_i\circ P \circ R^s$, where $P\colon \Tsnorm[\Gamma] \to \Tnorm[\Gamma]$ is the canonical morphism.  If an edge~$e$ joins levels $j<i$, then by \cref{eq:rrho} these functions satisfy
\begin{equation} \label{eq:f2sj} 
  f_e^{\kappa_e} \= s_j \cdots s_{i-1}.
\end{equation}
The following definition makes precise the notion that a morphism~$R^s$ as above defines a compatible system of node-smoothing parameters.

\begin{df} \label{def:RescEns}
A \emph{simple rescaling ensemble} is a morphism $R^s\colon B\to \Tsnorm[\Gamma_b]$ such that the parameters $f_e\in \calO_{B}(B)$ for each vertical edge~$e$ determined by~$R^s$ lie in the equivalence class~$[f_e]$ determined by the family $\pi\colon\calX\to B$.  A \emph{rescaling ensemble} is a morphism $R\colon B\to \Tnorm[\Gamma_b]$ which arises as the composition $\pi\circ R^s$ for some simple rescaling ensemble $R^s$.
\end{df}
 
The $s_i$ and $f_e$ will be called the \emph{rescaling parameters} and
\emph{smoothing parameters} determined by $R$ or $R^s$. 
The composition of $R^s$ with the coordinate projections gives functions~$t_i$
such that $s_i = t_i^{a_i}$. We refer to those $t_i$ as the \emph{level parameters}.

Recall that the adjective `generalized' in the following definition refers to the fact that
the global residue condition has been dropped, compared to~\cite{LMS}. For an
illustration of some elements of the definition, see~\ref{fig:rescaleddiff}.
The well-definedness of the period in the following definition is checked (in any
characteristic), \textit{e.g.} in \cite[Lemma~1.8]{Boj}.

\begin{df} \label{def:collRD}
A \emph{collection of generalized rescaled differentials of type~$\mu$} on the stable $n$-pointed controlled curve $(\pi\colon \calX \to B, \bfz)$ with a controlling point $b\in B$ is a collection of sections~$\omega_{(i)}$ of~$\omega_{\calX/B}$ defined on the standard open subsets $U_i$ of~$\calX$, indexed by the levels~$i\in L^\bullet(\Gamma)$ of the enhanced level graph. The irreducible components of~$X_b$ on a level strictly below~$i$ are called \emph{vertical zeros}, and those on a level strictly above~$i$ are called \emph{vertical poles} of~$\omega_{(i)}$.  We require the collection to satisfy the following constraints:
\begin{enumerate}
\item\label{d:c-1} There exist sections $s_i \in H^0(B, \mathcal{O}_B)$ with $s_i(b)=0$ such that for any levels $j<i$, the differentials satisfy $\omega_{(i)} \= s_j \cdots s_{i-1} \omega_{(j)}$ on $U_i\cap U_j$.
\item\label{d:c-2} For any edge $e$ joining levels $j<i$, for any $p \in B$ over which $e$ persists, and for some (equivalently, any) choice of functions $u_e, v_e$ on~$\calX$ and $f_e$  on $B$
such that the family has local normal form $u_e v_e = f_e$, there exists a unit $\lambda$ in the strict Henselization\footnote{In the complex analytic category, this would be the germ of a non-vanishing function around the node.} $(\ca O_B[u_e, v_e]/(u_ev_e - f_e))_{(p,0,0)}^{\rm sh}$ of the local ring at $(p, u_e=0, v_e=0)$ such that
\begin{equation}\label{eq:new_normal_form}
\omega_{(i)} \= \lambda u_e^{\kappa_e} \frac{du_e}{u_e} \quad \text{and} \quad     \omega_{(j)} \= - \lambda v_e^{-\kappa_e}\frac{dv_e}{v_e}. 
\end{equation}
\item\label{d:c-3} The $\omega_{(i)}$ have order~$m_k$ along the sections $\calZ_k$ of the $\supth{k}$ marked point that meet the level $i$ subcurve of $X_b$; these are called {\em horizontal zeros and poles} (where $\calZ^{\infty}$ records the horizontal poles). Moreover, $\omega_{(i)}$ is holomorphic and non-zero away from its horizontal and vertical zeros and poles.
\end{enumerate}

If the rescaling and smoothing parameters $s_i, f_e$ for the collection $\omega_{(i)}$ agree with those of a rescaling ensemble~$R^s$ or~$R$, we call them {\em compatible}. We denote the collection by $\bfomega = (\omega_{(i)})_{i \in L^\bullet(\Gamma)}$. 
\end{df}

The unit $\lambda$ is unique since $\omega_{(i)}$ is a generating section of $\omega$ after inverting $u_e$ and $\omega_{(j)}$ is a generating section of $\omega$ after inverting $v_e$.  We then apply the following result from commutative algebra.

\begin{lemma}
Let $R$ be a ring, $f\in R$, and $A = R[u,v]/(uv-f)$. Then $h\colon A \to A[u^{-1}] \times A[v^{-1}]$ is injective.
\end{lemma}

\begin{proof} Assume $h(a) = 0$. Any element $a \in A$ has a unique normal form
$$a \= e_0 + c_1 u + \cdots + c_n u^n + d_1 v + \cdots + d_m v^m.$$
Under the unique isomorphism $A[u^{-1}] \to R[u, u^{-1}]$ sending $v$ to $f/u$, the element $a$ is mapped to
$$e_0 + c_1 u + \cdots + c_n u^n + d_1 f u^{-1} + \cdots + d_m f^m u^{-m}.$$ Therefore, the vanishing of the first component of $h$ implies $e_0 = c_1 = \cdots = c_n = 0$, as can be seen just by reading off the coefficients of the $R$-basis $\{u^t : t \in \mathbb Z\}$. Similarly, the vanishing of the second component implies $d_1 = \cdots = d_m =0$.
\end{proof}

Applying this to $R=\calO_B$, $u=u_e$, $v=v_e$, and $f=f_e$, we first conclude that the map
\[
\ca O_B[u_e, v_e]/(u_ev_e - f_e) \lra \ca O_B[u_e, v_e, u_e^{-1}]/(u_ev_e - f_e) \times \ca O_B[u_e, v_e,v_e^{-1}]/(u_ev_e - f_e)
\]
is injective. Localization and strict Henselization are flat. Therefore, the above injectivity is preserved. 

The reader comparing with the definition in \cite{LMS} will realize that there in item~(2) it is required that for any edge $e$ joining levels $j<i$
of $\Gamma$, there are functions $u_e, v_e$ on~$\calX$ and $f_e$  on $B$,
such that the family has local normal form $u_e v_e = f_e$, and in these
coordinates
\begin{equation}\label{eq:5}
\omega_{(i)} \= \left(u_e^{\kappa_e} + f_e^{\kappa_e}r_{e, (j)}\right)\frac{du_e}{u_e}
\quad\text{and}  \quad
    \omega_{(j)} \= - \left( v_e^{-\kappa_e} + r_{e, (j)}\right) \frac{dv_e}{v_e},
\end{equation}
where $\kappa_e$ is the enhancement of $\Gamma_b$ at $e$. The two normal forms are equivalent at least if~$f_e^{\kappa_e} \neq 0$: Equation~\ref{eq:5} implies~\ref{eq:new_normal_form} by taking $\lambda = 1 + r_{e,(j)} (f_e/u_e)^{\kappa_e}$. Conversely, given~\ref{eq:new_normal_form} we may change $v_e$ to a coordinate that is in the normal form~\ref{eq:5} by \cite[Theorem~4.3]{LMS}.  The form~\ref{eq:new_normal_form} is the one we can directly associate with a rubber differential; see~\ref{subset:rescaled_from_log}.

\medskip
\def\surfacehole#1{
\begin{scope}[shift={#1}, scale=0.7]
\draw[thick] plot [smooth, tension = 0.5,xshift=-1.4cm, yshift=2.15cm] coordinates {(1.1,-2.1) (1.2,-2.2) (1.4,-2.3) (1.6,-2.2) (1.7,-2.1)};
\draw[thick] plot [smooth, tension = 0.8,xshift=-1.4cm, yshift=2.15cm] coordinates { (1.2,-2.2) (1.4,-2.13) (1.6,-2.2)};
\end{scope}
}

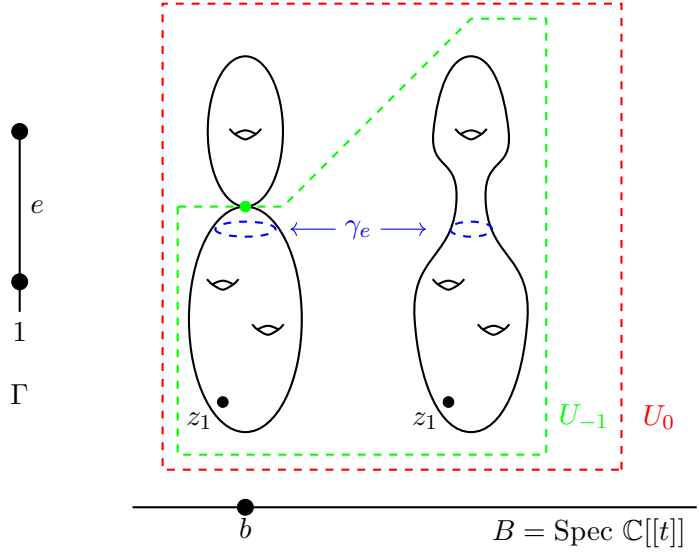
\begin{figure}[htb]
    \centering
\[
\begin{tikzpicture}
\filldraw (0,3) circle (3pt);
\filldraw (0,5) circle (3pt);
\draw[thick] (0,3) to node[midway, right]{$e$} (0,5);
\draw[thick] (0,3) -- (0,2.6) node[below] {$1$};
\draw (0,1.5) node {$\Gamma$};

\draw[thick] (1.5,0) -- (9,0) node[below left]{$B = \on{Spec}(\mathbb{C}[[t]])$};
\filldraw (3,0) circle (3pt) node[below]{$b$};

\draw[thick] plot [smooth cycle, tension = 1, xshift=0 cm, yshift = 0cm] coordinates { (3,4) (2.5,5) (3,6) (3.5,5) };
\draw[thick] plot [smooth cycle, tension = 1, xshift=0 cm, yshift = 0cm] coordinates { (3,4) (2.25,2.5) (3,1) (3.75,2.5) };
\surfacehole{(3,5)};
\surfacehole{(2.7,3)};
\surfacehole{(3.3,2.4)};
\filldraw (2.7,1.4) circle (2pt) node [below left]{$z_1$};

\draw[thick] plot [smooth cycle, tension = 1, xshift=3  cm, yshift = 0cm] coordinates { (3.2,4) (3.5,5) (3,6) (2.5,5) (2.8,4) (2.25,2.5) (3,1) (3.75,2.5) };
\surfacehole{(6,5)};
\surfacehole{(5.7,3)};
\surfacehole{(6.3,2.4)};
\filldraw (5.7,1.4) circle (2pt) node [below left]{$z_1$};

\draw[thick, blue, dashed] plot [dotted, smooth cycle, tension = 1] coordinates { (3,3.8) (2.6,3.7) (3,3.6) (3.4,3.7) };
\draw[thick, blue, dashed,xshift=3 cm] plot [dotted, smooth cycle, tension = 1] coordinates { (3,3.8) (2.72,3.7) (3,3.6) (3.28,3.7) };
\draw[blue] (4.5,3.7) node {$\gamma_e$};
\draw[blue, ->] (4.2,3.7) -- (3.6,3.7);
\draw[blue, ->] (4.8,3.7) -- (5.4,3.7);

\filldraw[green] (3,4) circle (2pt);
\draw[thick, green, dashed] plot [dotted, smooth cycle, tension = 0.0] coordinates { (3,4) (2.1,4) (2.1,0.7) (3,0.7) (7,0.7) (7,6.5) (6,6.5) (3.5,4) };
\draw[green] (7.5,1.2) node {$U_{-1}$};
\draw[thick, red, dashed] plot [dotted, smooth cycle, tension = 0.0] coordinates { (1.9,0.5) (8,0.5) (8,6.7) (1.9,6.7) };
\draw[red] (8.5,1.2) node {$U_{0}$};

\end{tikzpicture}
\]    
\caption{The underlying curve for a family of generalized rescaled differentials of type $\mu=(4)$, with neighbourhoods $U_{0}$, $U_{-1}$ (in red, green) and the vanishing cycle $\gamma_e$ (in blue).}
    \label{fig:rescaleddiff}
\end{figure}

\begin{remark}\label{rem:induced_prong_matching}
Let $\bfomega$ be a collection of generalized rescaled differentials with
a compatible rescaling ensemble~$R^s$ or $R$. For any non-semi-persistent
edge~$e$,  denote by $B_e\subset B$ the vanishing locus of $f_e$. Then there is a natural \emph{induced prong-matching} $\sigma_e$
over~$B_e$, which is determined by the choice
of the rescaled differentials $\omega_{(i)}$ and the rescaling ensemble.
This prong-matching~$\sigma_e$ is defined explicitly in local coordinates by
writing it as $\sigma_e=d u_e \otimes dv_e$ when restricting to the nodal
locus corresponding to $e$, where $u_e$ and $v_e$ are as in~\ref{eq:5}
with~$f_e$ prescribed by the rescaling ensemble.  Any two possible choices
of $u_e$ and $v_e$ are of the form~$\alpha_e u_e$ and $\alpha^{-1}_e v_e$ for
some unit $\alpha_e \in \calO_{B}^*$ (see \cite[Section 4]{LMS}), 
so the induced prong-matching does not depend on this choice.
\end{remark}

We can now package everything into our main notion.

\begin{df} \label{def:germMSD}
Given a controlled family of pointed stable curves $(\pi\colon\calX\to B, \bfz)$, a ({\em controlled}\,) {\em family of generalized {simple} \msds} of type~$\mu$ over $B$ consists of the following data:
\begin{enumerate}
\item\label{d:gMSD-1} the structure of an enhanced level graph on the dual graph $\Gamma_b$ of the fiber $X_b$;  
\item\label{d:gMSD-2} a simple rescaling ensemble $R^s\colon B \to \Tsnorm[\Gamma_b]$, compatible with
\begin{enumerate}
\item\label{d:gMSD-3} a collection of generalized rescaled differentials $\bfomega = (\omega_{(i)})_{i \in L^\bullet(\Gamma_b)}$ of type~$\mu$, and
\item\label{d:gMSD-4} a collection of prong-matchings $\bfsigma = (\sigma_e)_{e \in E^v(\Gamma)}$, where $\sigma_e$ is a section of~$\calN_e^\vee$ over~$B_e$, the vanishing locus of $f_e$.  If~$e$ is a non-semi-persistent node, $\sigma_e$ is required to agree with the induced prong-matching defined in~\Cref{rem:induced_prong_matching}.
  \end{enumerate}
\end{enumerate}
\end{df}

A section of the simple level rotation torus $T^s_{\Gamma_b}(\calO_B)$, that is, a morphism $\sltelement \colon B \to T^s_{\Gamma_b}$, acts on all of the above data via
\[
\sltelement \cdot \left(\omega_{(i)}, R^s,
\sigma_e\right) \= \left(\sltelement \cdot \omega_{(i)}, \sltelement^{-1}\cdot R^s,\sltelement \cdot \sigma_e \right).
\]
Here, for $\sltelement \in T^s_{\Gamma_b}(\calO_B)$ mapping to $((r_i)_{i\in L(\Gamma_b)}, (\rho_e)_{e \in E^v(\Gamma_b)})$
under the quotient map~\ref{eq:simpletotorus}, the action is defined by
\[
\sltelement\cdot \omega_{(i)} \= \left(\prod_{\ell \geq i} r_\ell\right)
\omega_{(i)}, \quad \sltelement\cdot \sigma_e \= \rho_e \sigma_e,
\]
and  $\sltelement^{-1}\cdot R^s$ denotes the post-composition of $R^s$ with the multiplication by~$\sltelement^{-1}$.\footnote{Most of the checks that this action is well defined are straightforward. To verify part~\eqref{d:gMSD-2} of~\ref{def:germMSD}, assume we are given local coordinates $u,v$ around a node associated to $e \in E^v(\Gamma_b)$ such that the differentials have the normal form~\ref{eq:5}. Then the rescaled differential is put in the required normal form using the new coordinates $\widehat u = (\prod_{\ell \geq i} r_\ell )^{1/\kappa_e} u_e$ and  $\widehat v = (\prod_{\ell \geq j} r_\ell )^{-1/\kappa_e} v_e$.}

An \emph{isomorphism between two controlled families of generalized simple \msds} over the same base $B$ and with the same controlling point $b$ \be \label{eq:defmorphismMSD} (\pi'\colon\calX'\lra B, \bfz', \Gamma_{b}, (R^s)', \bfomega', \bfsigma') \,\longrightarrow\, (\pi\colon\calX\lra B, \bfz, \Gamma_{b}, R^s, \bfomega, \bfsigma) \ee consists of an isomorphism $\phi\colon \calX'\to \calX$ and an element $\sltelement \in T^s_{\Gamma_{b}}(\calO_{B})$ such that
\begin{enumerate}
\item[i)] $\phi$ defines an isomorphism of families of pointed
  stable curves,
\item[ii)] the induced isomorphism of dual graphs $ \Gamma'_{b}\to \Gamma_b$ is also an isomorphism of enhanced level graphs,
\item[iii)] the action of $\sltelement$ sends $((R^s)', \bfomega', \bfsigma')$
to $\phi^* (R^s, \bfomega,\bfsigma)$.
\end{enumerate}

Pullbacks of families of controlled generalized \msds are defined as in \cite[Section~11.2]{LMS}. This step requires some care, since the number of levels, the nodes where the prong-matching is an induced prong-matching, and the target of the rescaling ensemble map change. With this in hand, we can define a family of generalized \msds over any scheme locally of finite type over $\bb C$ by sheafifying the notion already defined for controlled families, using that controlled families form a base for the \'etale topology in the sense of \cite[Definition A.3]{biesel2023fine}. This definition is analogous to that in \cite[Section~11.3]{LMS}, and can be seen as a groupoid-version of the constructions worked out in \cite[Appendix A]{biesel2023fine}.

\begin{df} \label{def:simpleMSD}
We let $\GSMS$ be the groupoid of families of
generalized simple multi-scale differentials.
\end{df}

There are two variants of this definition. First, replacing $T^s_{\Gamma_b}(\calO_{B, b})$ with the extended level rotation torus $\Textd[\Gamma_b](\calO_{B, b})$ in the definition of a morphism, we obtain projectivized generalized simple \msds. Here the additional torus factor acts by scaling the differential on all levels simultaneously, including level~$0$. These are relevant to get compact spaces. Here we compare the unprojectivized definitions and will not elaborate further on this.

Second, there is a `non-simple' variant that we need to compare to the relative coarse moduli space. The remarks above about pullback and sheafification apply here as well.

\begin{df} \label{def:germnonsimpleMSD}
A {\em family of \emph{controlled} generalized \msds} of type~$\mu$ is
defined as in~\ref{def:germMSD}, replacing~\eqref{d:gMSD-2} by a rescaling ensemble $R\colon B \to \Tnorm[\Gamma_b]$. A morphism of such controlled families consists of $(\phi, \tilde{\phi}, \sltelement)$ as above, except that now we allow $\sltelement \in T_{\Gamma_{b'}}(\calO_{B'})$. We let $\GMS$ be the resulting groupoid of families of generalized multi-scale differentials.
\end{df}

Modifying \Cref{def:collRD} by additionally imposing  the global residue condition gives a
groupoid that we denote by $\LMS$ for the simple version (\ref{def:germMSD})
and by $\MSgrp$ for the non-simple version (\ref{def:germnonsimpleMSD}).
We state the comparison to the objects defined in \cite{LMS}.

\begin{proposition} \label{prop:smoothDM}
The stack $\LMS$ is a smooth DM-stack. The stack $\MSgrp$ is a stack with finite quotient singularities and agrees with the normalization of the orderly blowup of the normalized incidence variety compactification; see \cite[Section~14]{LMS}.
\end{proposition}

The paper \cite[Section~14.2]{LMS} also defines a smooth stack denoted
by~$\LMS$, patched
locally from quotients of stacks with a Teichm\"uller marking. The full proof
that this stack is isomorphic to the stack with the same symbol defined here
would require recalling the lengthy definitions of level-wise real blowup
and Teichm\"uller marking from \cite[Section~12]{LMS}.
This identification directly implies
the second statement of the proposition. The proof given here
provides the main content of the proposition, the smoothness of this stack,
without using the smoothness results from~\cite{LMS}.

\begin{proof}
Recall from \cite[Section~8.1 and Theorem~10.1]{LMS} that a versal deformation space $B$ of $\MSgrp$ is given by a product $B = \Tnorm[\Gamma_b] \times B_0$, where $\Tnorm[\Gamma_b]$ gives a parameterization of possible rescaling ensembles $R$ (which have values in $\Tnorm[\Gamma_b]$), and where $B_0$ parameterizes the remaining data (deformations of the components $\calX_v$ for $v \in V(\Gamma_b)$ and twisted differentials on these components).  In fact, this local structure is given for the model space in \cite[Section~8.1]{LMS}. This model space is locally isomorphic to the Dehn space by the plumbing construction given in \cite[Theorem~10.1]{LMS}, and \cite[Proposition~12.5]{LMS} shows that every family can locally be lifted to the Dehn space.

Consider the fiber product
$$
\begin{tikzcd}
\widehat{B}\coloneqq B \times_{\MSgrp} \LMS \arrow[r] \arrow[d] & \LMS \arrow[d]\\
B \arrow[r] & \MSgrp\rlap{.}
\end{tikzcd}
$$
We claim that $\widehat{B}$ is equal to the stack quotient $[\overline{T}_\Gamma^s /
K_\Gamma]$ times the product~$B_0$ of the other factors. Then the maps $\widehat{B}
\to \LMS$ provide a smooth cover by spaces which are smooth themselves, which
implies the claimed smoothness of $\LMS$.

To show that $\widehat{B}$ is equal to $[\overline{T}_\Gamma^s / K_\Gamma] \times B_0$, we write down explicitly the maps $\widetilde B \to \widehat{B}$, where $\widetilde B$ carries a controlled curve. For this, recall\footnote{For a reminder on fiber products of stacks, we recommend the excellent paper \cite{Fantechi-Stacks-for-everybody}.} that a morphism to a fiber product as above is given by a triple 
\[
\left(\widetilde{B} \to \LMS, \widetilde{B} \to B, G\right),
\]
where $G$ is a $2$-isomorphism between the compositions 
\[
\widetilde{B} \lra \LMS \lra \MSgrp\quad\text{and}\quad\widetilde{B} \lra B \lra \MSgrp.
\]
Inserting the definitions of the moduli stacks, this data above is equivalent to a triple of
\begin{itemize}
    \item a controlled family $(\pi\colon\calX\to \widetilde{B}, \bfz, \Gamma_{b}, R^s \colon \widetilde B \to \Tsnorm[\Gamma_b], \bfomega, \bfsigma)$ of generalized simple multi-scale differentials,
    \item morphisms $s_T\colon \widetilde B \to \Tnorm[\Gamma_b]$ and $s_0\colon \widetilde B \to B_0$ (which together can be thought of as a morphism $(s_T,s_0)\colon \widetilde B\to \Tnorm[\Gamma_b]\times B_0=B$), 
    \item an isomorphism ($\calX \cong \calX'$, $\sltelement \in T_{\Gamma_b}(\calO_{\widetilde{B}})$) of generalized (non-simple) multi-scale differentials, sending the family $(\pi\colon\calX\to \widetilde{B}, \bfz, \Gamma_{b}, R, \bfomega, \bfsigma)$ to the family $(\pi'\colon\calX'\to \widetilde{B}, \bfz', \Gamma_{b}, R', \bfomega', \bfsigma')$ induced by $(s_T, s_0)\colon \widetilde B \to B$.
\end{itemize}
By identifying the families of curves $\calX \cong \calX'$, we can act on the
pair $(s_T, s_0)$ with  the section $\sltelement$ of the level rotation torus.
Replacing $(s_T, s_0)$ by this modified pair, we obtain a new, equivalent, triple of data, where the isomorphism in the last bullet point is taken as the identity. 
But then we see that such a triple is uniquely determined by the pair
\[
\left(R^s \colon \widetilde{B} \lra \Tsnorm[\Gamma_b],\ s_0\colon \widetilde{B} \lra B_0\right),
\]
by taking $s_T$ in the second bullet point as the composition $\widetilde{B} \to \Tsnorm[\Gamma_b] \to \Tnorm[\Gamma_b]$ and taking the data $(\pi,\bfz,\Gamma_b,\bfomega,\bfsigma)$ in the first bullet point that is determined by the non-simple generalized multi-scale differential from $(s_T, s_0)\colon \widetilde B \to B$.

Above we have found that any morphism $\widetilde{B} \to \widehat{B}$ can be described by a morphism $(R^s, s_0) \colon \widetilde{B} \to \Tsnorm[\Gamma_b] \times B_0$. Two such morphisms are $2$-isomorphic if they can be related by compatible isomorphisms for the stacks $B$ and $\LMS$ in the fiber product. Since $B$ is a scheme, the only such isomorphisms come from sections $\sltelement \colon \widetilde{B} \to T^s_{\Gamma_b}$ leaving the underlying non-simple generalized multi-scale differential fixed. These are exactly identified with sections $\sltelement \colon \widetilde{B} \to K_{\Gamma_b}$, which act in a natural way on the first morphism $R^s \colon \widetilde{B} \to \Tsnorm[\Gamma_b]$. 
Since $\widetilde{B}$ is connected, the section $\sltelement$ is necessarily constant, so that we have identified\footnote{For the second equality below, we use that for a finite group $K$ acting on a scheme $\overline{T}$, the morphisms $\widetilde B \to [\overline{T}/K]$ can be identified with the groupoid $\{\widetilde{B} \to \overline{T}\}/K$, perhaps after replacing $\widetilde B$ by an \'etale cover, which is harmless for the argument. This itself uses the definition of the quotient stack together with the fact that all $K$-torsors over a scheme $\widetilde B$ as above are \'etale-locally trivial.}
\[
\mathrm{Mor}\left(\widetilde{B}, \widehat{B}\right) \= \mathrm{Mor}\left(\widetilde{B}, \Tsnorm[\Gamma_b] \times B_0\right)/K_\Gamma \= \mathrm{Mor}\left(\widetilde{B},\left[\Tsnorm[\Gamma_b]/K_\Gamma\right] \times B_0\right).
\]
This proves the isomorphism $\widehat{B} \cong [\Tsnorm[\Gamma_b]/K_\Gamma] \times B_0$. Since both the quotient stack $[\Tsnorm[\Gamma_b]/K_\Gamma]$ and $B_0$ are smooth, this finishes the proof.
\end{proof}
\begin{proof}[Proof of~\ref{intro:mainiso}, second part]
  Assuming the first part of the theorem, the proof of the second part is completed by showing that the map $\GSMS \to \GMS$ is the relative coarse moduli space over $\Mbar_{g,n}$. First, we observe that the map $\GMS \to \Mbar_{g,n}$ is  representable. Indeed, the stabilizers $(\varphi, \widetilde \varphi, \xi)$ of a family of generalized multi-scale differentials lying over the identity morphism
  $\varphi=\mathsf{id}_B$, $\widetilde \varphi=\mathsf{id}_X$ of the underlying stable curves are those $\xi \in T_{\Gamma_b}(\calO_{B})$ fixing both the differentials $\bfomega$ and the prong-matchings~$\bfsigma$. By the definition of the level rotation torus, this forces $\xi$ to be trivial, so that indeed the stabilizers of $\GMS$ inject to the stabilizers of $\Mbar_{g,n}$.

By the definition of the relative coarse space, we then have a factorization
\[
\GSMS \lra \GSMS^{\sf{coarse}} \lra \GMS,
\]
and we show that the second map is an isomorphism. For this, let $B \to \GMS$ be associated to a controlled family of generalized multi-scale differentials. Then we have a commutative diagram, where we \emph{define} the diagrams on the right to be cartesian:
\[
\begin{tikzcd}
{[\overline{T}_{\Gamma_b}^s/K_\Gamma]} \arrow[dd] & \GSMS_B \arrow[r] \arrow[l, dashed] \arrow[d] & \GSMS \arrow[d]\\
 & \GSMS_B^{\sf{coarse}}  \arrow[r] \arrow[d]  & \GSMS^{\sf{coarse}} \arrow[d]\\
\overline{T}_{\Gamma_b}^n & B \arrow[l,"R"] \arrow[r] & \GMS\rlap{.}
\end{tikzcd}
\]
For the family $B \to \GMS$ given by a tuple $(\pi\colon\calX\to {B}, \bfz, \Gamma_{b}, R, \bfomega, \bfsigma)$ and a map $s\colon \widetilde B \to B$, we claim that after we shrink $\widetilde B$ in the \'etale topology, the morphisms $\widetilde B \to \GSMS_B$ lying over $s$ are precisely given by the data of
\begin{equation}  \label{eqn:mod_normal_form_stacky_fiber_product}
R^s \colon \widetilde B \lra \Tsnorm[\Gamma_b] \quad\text{such that } \left(\Tsnorm[\Gamma_b] \to \Tnorm[\Gamma_b]\right) \circ R^s \= R \circ s\,
\end{equation}
with the automorphisms of this data described by the (necessarily locally constant) sections $\sltelement \colon \widetilde{B} \to K_{\Gamma_b}$.  To see this, one repeats the analysis of the stacky fiber product from \Cref{prop:smoothDM} for $\GSMS_B$.\footnote{In that proof we did use the normal form $B=\Tnorm[\Gamma_b] \times B_0$ to split the map $s \colon \widetilde B \to B$ as $(s_T, s_0)$, where $s_T = R \circ s$. However, the only place where this was actually used was in observing that the data $(s=(s_T, s_0), R^s)$ satisfying the condition in~\ref{eqn:mod_normal_form_stacky_fiber_product} is equivalent to $(R^s, s_0)$ by setting $s_T = (\Tsnorm[\Gamma_b] \to \Tnorm[\Gamma_b]) \circ R^s$. } From this description, we see that $\GSMS_B$ is also the fiber product of~$B$ with the map $[\overline{T}_{\Gamma_b}^s/K_\Gamma] \to \overline{T}_{\Gamma_b}^n$, so that the dotted arrow on the top left makes the left diagram cartesian.

To conclude, we first note that by \cite[Proposition 3.4]{Abramovich2011Twisted-stable-} the space $\GSMS_B^{\sf{coarse}}$, which was defined as a fiber product, is in fact a relative coarse space for $\GSMS_B$ over $\Mbar_{g,n}$.  But since the map $\GSMS_B \to \Mbar_{g,n}$ factors through the representable map $B \to \Mbar_{g,n}$, the space $\GSMS_B^{\sf{coarse}}$ is \emph{also} a coarse space of $\GSMS_B$ over $B$, by an application of~\ref{lem:annoyingcoarsespaces} below to $\calX = \GSMS_B$, $\calY'=B$, and $\calY=\Mbar_{g,n}$.

On the other hand, since $\overline{T}_{\Gamma_b}^n$ is the coarse space of $[\overline{T}_{\Gamma_b}^s/K_\Gamma]$ (over $\mathrm{Spec}(\mathbb{C})$), applying \cite[Proposition 3.4]{Abramovich2011Twisted-stable-} again shows that $B$ \emph{itself} is the coarse space of $\GSMS_B$. This proves that the map $\GSMS_B^{\sf{coarse}} \to B$ is an isomorphism. Since we prove this for any $B \to \GMS$, we conclude that $\GSMS^{\sf{coarse}} \to \GMS$ is an isomorphism, as desired.
\end{proof}

\begin{lemma} \label{lem:annoyingcoarsespaces}
Consider a sequence of morphisms $\calX \to \calY' \to \calY$ of algebraic stacks, locally of finite presentation, and assume the relative inertia $I(\calX /\calY) \to \calX$ is finite. Then if\, $\calY' \to \calY$ is representable, the relative coarse space $\calX^{\sf{coarse},\calY}$ of\, $\calX$ over $\calY$ is isomorphic to the relative coarse space $\calX^{\sf{coarse},\calY'}$ of\, $\calX$ over $\calY'$.
\end{lemma}
\begin{proof}
Consider the solid diagram of morphisms
\[
\begin{tikzcd}
\calX \arrow[d, equal] \arrow[r] & \calX^{\sf{coarse},\calY'} \arrow[r] \arrow[d, dashed, "g", xshift=0.2cm] & \calY' \arrow[d]\\
\calX \arrow[r] & \calX^{\sf{coarse},\calY} \arrow[r] \arrow[u, dashed, "f", xshift=-0.2cm] & \calY\rlap{.}
\end{tikzcd}
\]
Then by the properties of relative coarse spaces (see \cite[Theorem 3.1(2)]{Abramovich2011Twisted-stable-}),  there exists a morphism $f$ as above, since $\calX^{\sf{coarse},\calY}$ is initial among factorizations of $\calX \to \calY$ via a representable map. But then the induced map $\calX^{\sf{coarse},\calY} \to \calY'$ via $f$ is representable, so by the same universal property we obtain the map $g$, and one verifies that $f,g$ are inverse isomorphisms.
\end{proof}

\begin{remark} \label{rem:PMEqcount}
In practice it is often relevant to determine the number of projectivized multi-scale differentials on a given pointed curve with twisted differential $(X,\bfz,\Gamma, \bfomega)$.  By the definition of the above equivalence relation, this is the number of \emph{prong-matching equivalence classes}, \textit{i.e.}~the number of orbits of the set of global prong-matchings under the action of the level rotation group $R_\Gamma$. Determining this number is complicated in general, but, for a two-level graph with prongs $\kappa_1,\ldots,\kappa_s$, there are $\prod \kappa_i / \lcm(\kappa_i)$ prong-matching equivalence classes.
\end{remark}

\subsection{Quotient twist group and rescaling ensembles in a worked example}
\label{sec:QuotTwist}

Consider the triangle graph $\Gamma$ with three levels, each containing one vertex, and
three edges forming a triangle, as illustrated in~\ref{fig:X_L_pic} (to which
we also refer for the labelling of the edges).
\begin{figure}[ht]
  \begin{tikzpicture}[very thick]
    \begin{scope}
\fill (0,0) coordinate (x0) circle (2.5pt); \node [above] at (x0) {$X_{(0)}$};
\fill (1,-1) coordinate (x1) circle (2.5pt); \node [below right] at (x1) {$X_{(-1)}$};
\fill (-1,-2) coordinate (x2) circle (2.5pt); \node [below right] at (x2) {$X_{(-2)}$};

 \draw[] (x0) -- node[right]{$\kappa_1$} node[left]{} (x1);
 \draw (x0) --node[left]{$\kappa_3$} node[right]{} (x2);
 \draw (x1) -- node[below]{$\kappa_2$} node[above]{} (x2);
  \end{scope}
\end{tikzpicture}
\quad \quad \quad
\begin{tikzpicture}[very thick]   \begin{scope}
\fill (0,0) coordinate (x0) circle (2.5pt); \node [above] at (x0) {$X_{(0)}$};
\fill (1,-1) coordinate (x1) circle (2.5pt); \node [below right] at (x1) {$X_{(-1)}$};
\fill (-1,-2) coordinate (x2) circle (2.5pt); \node [below right] at (x2) {$X_{(-2)}$};
\fill (-0.5,-1) coordinate (x3) circle (2.5pt); \node [below right] at (x3) {};
 \draw[] (x0) -- node[right]{$\kappa_1$} node[left]{} (x1);
 \draw (x0) --node[left]{$\kappa_3$} node[right]{} (x3);
  \draw (x3) --node[left]{$\kappa_3$} node[right]{} (x2);
 \draw (x1) -- node[below]{$\kappa_2$} node[above]{} (x2);
 \end{scope}
\end{tikzpicture}
\quad \quad
\caption{The triangle graph (the generic fiber~$X$, left) and its subdivision
(the special fiber $X_L$, right) where the extra vertex stands for a semistable rational component.}
\label{fig:X_L_pic}
\end{figure}
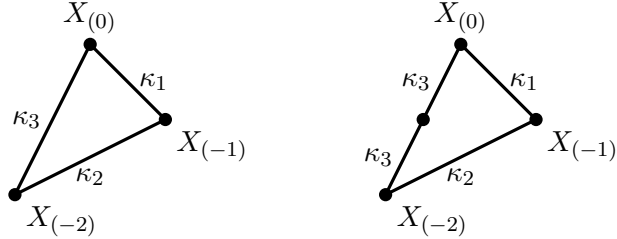
For simplicity we restrict to the case $\kappa_1 = 1 = \kappa_2$ and $\kappa_3 = n$. In this case the simple twist group is $\sTw[\Gamma] = n\ZZ \oplus n\ZZ$. The full twist group is generated by the simple twist group and the element $(1,-1)$. In particular, we note that the {quotient twist group} is
\be \label{eq:KGammaexample}
K_\Gamma \= \Tw[\Gamma]/\Tw[\Gamma]^s \= \ZZ/n\ZZ.
\ee
To work explicitly with invariants, we specialize to the case $n=3$ in what follows.  

The simple level rotation torus is isomorphic to $(\CC^*)^2$, hence $\ol{T}^s_\Gamma \cong \CC^2$, and a morphism $R^s\colon B \to \ol{T}^s_\Gamma$ is given by two functions $(t_{-1},t_{-2})$. Consequently,
\begin{align*}
  \ol{T}^n_\Gamma \= \ol{T}^s_\Gamma / K_\Gamma &\= \left\{(s_{-1}, s_{-2}, f_{1}, f_{2}, f_{3})
  : f_1^1 \= s_{-1}, f_{2}^1 \= s_{-2}, f_{3}^{3} \= s_{-1} s_{-2} \right\}\\
  &\= \left\{(f_{1}, f_{2},  f_{3}) : f_{3}^{3} \= f_{1} f_{2}\right\}, 
\end{align*}
where $s_{-1}=f_{1}=t_{-1}^3$, $s_{-2}=f_2= t_{-2}^3$, and $f_3= t_{-1}t_{-2}$.  In this case the rescaling ensemble $R$ induced by $R^s$ is given by the composition of $R^s$ with the quotient map $\ol{T}^s_\Gamma \to \ol{T}^s_\Gamma / K_\Gamma$, and has coordinates
$$
(s_{-1}, s_{-2}, f_1, f_2, f_3) \= \left(t_{-1}^3, t_{-2}^3, t_{-1}^3, t_{-2}^3, t_{-1} t_{-2}\right).
$$

Let $\bfomega = (\omega_0, \omega_{-1},\omega_{-2})$ be a twisted differential on some pointed stable automorphism-free curve~$(X,\bfz)$ compatible with the~$\Gamma$ discussed here. By plumbing (see \cite[Section~12]{LMS}), we get a family of curves $\calX \to \ol{T}^n_\Gamma$ with an underlying collection of rescaled differentials
\bes
\omega_{(0)} \= \omega_0, \quad \omega_{(-1)} \= s_{-1} \omega_{-1},\quad
\omega_{(-2)} \= s_{-1}s_{-2} \omega_{-2}\,
\ees
and the rescaling ensemble~$R$.\footnote{In fact, replacing the initial datum~$\bfomega$ by the universal equisingular deformation inside the appropriate stratum of differentials and taking as new base $\ol{T}^n_\Gamma$ times the base of the equisingular deformation, we obtain the universal family.}

To summarize: \emph{near $(X,\bfz,\Gamma,\bfomega)$ as above, $\GMS = \MSgrp$,
since there are no global residue conditions; both functors are representable by an algebraic variety;
this algebraic variety is singular with a quotient singularity given by the group~$K_\Gamma$.}

\medskip
Finally, we remark that as illustrated in~\ref{fig:X_L_pic}, a geometric way to think of the $[\mathbb{P}^1/\mathbb{G}_m]$ subdivision is to modify the definition of level graphs by eliminating all long edges (\textit{i.e.}~edges crossing more than one level passage), and instead inserting semistable rational vertices at each level crossed by a long edge, with the same number of prongs. Then the corresponding twist group, level rotation torus, and rescaling ensemble have only their `simple' versions.  To see this concretely, suppose $uv = f$ is the local model of a node corresponding to a long edge crossing $k$ level passages, where $f^\kappa = s_{i-k}\cdots s_{i-1}$
as in~\ref{eq:f2sj}. Introduce new parameters $u_j, v_j, f_j$ for $i-k\leq j\leq i-1$ satisfying $u_jv_j = f_j$, $f_j^{\kappa} = s_{j}$, $v_ju_{j-1} = 1$, $u_{i-1} = u$, and $v_{i-k} = v$. Then $v_j$ and $u_{j-1}$ are coordinates on the inserted semistable rational vertex at each intermediate level~$j$ that can subdivide the long edge into $k$ short edges, where the differential on the semistable rational vertex is $u_{j-1}^{\kappa} (du_{j-1}/u_{j-1}) = - v_j^{-\kappa} (dv_j/v_j)$ and the number of prongs at each node in the inserted semistable rational curves remains equal to $\kappa$.  

A logarithmic version of this construction (the replacement of an edge crossing $k$ levels by a chain of $k-1$ projective lines) can be described in terms of the \emph{divided tropical lines} of \cite{MarcusWiseLog}. We will not use this in what follows, but we give a brief sketch. The PL function $\beta$ with maximum value $0$ can be seen as a map from $X$ to the `tropical line' $[\bb P^1/\bb G_m]$. Write $\ca P$ for the subfunctor of $\on{Hom}_{\cat{LogSch}}(-, [\bb P^1/\bb G_m])$ consisting of log maps to $[\bb P^1/\bb G_m]$ such that the images of the vertices of $X$ form a totally ordered set in the characteristic monoid. Then $\ca P \to [\bb P^1/\bb G_m]$ is a quotient by $\bb G_m$ of a contraction map from a chain of rational curves to a single rational curve. The fiber product $X \times_{[\bb P^1/\bb G_m]} \ca P$ is then the `divided' curve constructed above, all of whose edges cross at most one level. 

\section{The underlying algebraic stack of Rub}\label{sec:minimal_log_str}

The category $\cat{Rub}$ is naturally fibered over $\cat{LogSch}$. Our goal in this section is to understand its underlying algebraic stack (a fibered category over $\cat{Sch}$). We use the notion of minimal log structures from \cite{Gillam} and \cite[Appendix~B]{Wise2016Moduli-of-morph}. We describe here the minimal log structures on points of $\cat{Rub}$, as a variation on the description of the minimal log structures on $\cat{Div}$ given in the proof of
\cite[Theorem 4.2.4]{MarcusWiseLog}.

Throughout this section we work with $\cat{Rub}_{\ul 0}$ in place of $\cat{Rub}$, as it is notationally slightly simpler, and fits better with what we do in the rest of the paper. The interested reader will check that the results go through essentially unchanged for $\cat{Rub}$. 

\subsection{Brief recap on minimal log structures}
\label{sec:minimal_ls_for_rub}

This is taken from \cite[Appendix~B]{Wise2016Moduli-of-morph}, based
on \cite{Gillam}. The purpose of minimal log structures is to understand how to pass from a category fibered in groupoids (CFG) $X$ over $\cat{LogSch}$ to a CFG over $\cat{Sch}$. The wide subcategory of $\cat{LogSch}$ with morphisms the strict morphisms is a CFG over $\cat{Sch}$, so one could just take the corresponding wide subcategory of $X$ and composite. However, this is the `wrong' way to extract the underlying CFG of $X$ over $\cat{Sch}$. For an elementary example, let $X \coloneqq (\pt, \bb N^2)$ be a point with log structure~$\bb N^2$. Then there are very many maps from $Y \coloneqq (\pt, \bb N)$ to $X$: one can choose both the underlying monoid map $\bb N^2 \to \bb N$ and the lift to the log structure giving a $\bb C^*$ parameter. Hence if we take the CFG over $\cat{LogSch}$ associated to $X$ and view it as a CFG over $\cat{Sch}$ via taking strict morphisms and then the forgetful functor, we will get a very large and complicated object,\footnote{For example, the fiber over $\pt \in \cat{Sch}$ is the category of pairs of a log structure $M$ on $\pt$ and an associated log morphism $(\pt, M) \to X$. } when what we really wanted was a point!

However, given a map $T \to \pt$ of schemes, there exist a unique log structure $M$ on $T$ and a morphism $(T,M) \to X = (\pt, \mathbb{N}^2)$ such that any other log morphism $(T, M') \to X$ factors through $(T,M) \to X$. Namely,~$M$ is simply the pullback of the log structure $\mathbb{N}^2$ on $\pt$ under $T \to \pt$. Such a log structure is called \emph{minimal}, and if we take the full subcategory of log schemes over $X$ given by minimal objects, then view it as a CFG over $\cat{Sch}$ via the forgetful functor, we recover exactly what we wanted, namely a point. 

In the next two subsections, we will apply the same machinery to the CFG $\cat{Rub}_{\ul 0}$ over $\cat{LogSch}$. An object $(X/B, \beta)$ of $\cat{Rub}_{\ul 0}$ is called \emph{minimal} if every solid diagram in $\cat{Rub}_{\ul 0}$
\begin{equation}\label{eq:def_of_minimal}
 \begin{tikzcd}
  (X'/B', \beta') \arrow[rr]\arrow[dr] && (X/B, \beta)\rlap{,}  \\
  & (X''/B'', \beta'') \arrow[ur, dashed]
\end{tikzcd}
\end{equation} 
with the induced maps $\ul B' \to \ul B$ and $\ul B' \to \ul B''$ on underlying schemes being isomorphisms, admits a unique dashed arrow.

Gillam's results imply that the full subcategory of $\cat{Rub}_{\ul 0}$ consisting of minimal objects, together with its natural forgetful functor to $\cat{Sch}$, is (equivalent to) the underlying algebraic stack of $\cat{Rub}_{\ul 0}$. Thus, objects are those log points of $\cat{Rub}_{\ul 0}$ for which the log structure is minimal, and morphisms are simply the usual morphisms of log objects.

As such, if we want to understand the relative inertia of $\cat{Rub}_{\ul 0}$ over $\mathfrak M$, we need to understand not only the minimal objects and their morphisms, but also all possible ways of equipping a schematic object of $\cat{Rub}_{\ul 0}$ with minimal log structure.

\begin{remark} \label{Rmk:setoid_warning}
A warning: suppose that one starts with a CFG over $\cat{LogSch}$ which is equivalent to a category fibered in setoids (CFS), and which has enough minimal objects. It is then representable by an algebraic stack with log structure, but this \emph{need not} be equivalent to a category fibered in setoids over schemes (in other words, it can still have non-trivial stacky structure). The most elementary example of this is perhaps the subdivision of $\bb G_m^\trop$ at $0$, which is certainly a category fibered in setoids over $\cat{LogSch}$, but whose underlying algebraic stack is $[\bb P^1/\bb G_m]$. This is because a given schematic point can admit two (or more) different minimal log structures, which can have several isomorphisms between them even if we have a CFS over $\cat{LogSch}$; the fiber over \emph{any given} log scheme can still have no non-trivial automorphisms.
\end{remark}

\subsection{Minimal log structures for $\cat{Rub}_{\ul 0}$}

Let $(X/B, \beta)$ be a point of $\cat{Rub}_{\ul 0}$ with $X/B$ nuclear, where $\M_B$ is the sheaf of monoids on~$B$. Recall that from this family and any $b \in B$, we obtain
\begin{itemize}
    \item the stable graph $\Gamma$ describing the shape of $X_b$;
    \item the length maps $\delta\colon E(\Gamma) \to \ghost_{B,b}$, which we extend to a monoid homomorphism 
\bes
\delta\colon \bb N\Span{E(\Gamma)} \lra \ghost_{B,b};
\ees
\item the value map $\beta: V(\Gamma) \to \ghost_{B,b}^\gp$ at vertices, whose image is totally ordered, inducing the level map
\[
\ell \colon V(\Gamma) \lra \{0,-1, \ldots, -N\} \= \{0\} \sqcup L(\Gamma); 
\]
\item the slopes $\kappa \colon H(\Gamma) \to \mathbb{Z}$ at half-edges, where, given an edge $e \in E(\Gamma)$ consisting of half-edges $h,h'$, we set $\kappa_e = |\kappa(h)| = |-\kappa(h')|$ and denote by $E^v = \{e \in E(\Gamma) : \kappa_e > 0\}$  the set of vertical edges and by $E^h = \{e \in E(\Gamma) : \kappa_h = 0\}$ the set of horizontal edges.
\end{itemize}
For $i\in L(\Gamma)$, we define
with~\ref{eq:aidef}
\bes
a_i \,\coloneqq\, \on{lcm}_e \kappa_e, 
\ees
where the $\on{lcm}$ runs over the set of all edges $e$ such that $\ell(e^-) \le i<\ell(e^+)$ (we say such an edge~$e$ \emph{crosses level~$i$}).
We let $\tilde P \coloneqq \bb N\Span{p_{-1},\dots, p_{-N}}$ be the free monoid on $N=|L(\Gamma)|$ generators. Then we can define a map
$g\colon \vedges \to \tilde P$ by
\begin{equation}\label{eq:g_delta_e}
   g(e) \coloneqq \sum_{i = \ell(e^-)}^{\ell(e^+) - 1}\frac{a_i}{\kappa_e}p_i,
\end{equation}
and extend this map additively to a map 
$g\colon \bb N\Span{\vedges} \to \tilde P$. 
Finally, we let 
$$
\sigma_i \coloneqq \beta(v_i) - \beta(v_{i-1}) \in \ghost_{B,b},
$$
where $v_i$ is any vertex of level $i$. 

\begin{lemma}
The difference $\sigma_i$ is divisible by $a_i$ in $\ghost_{B,b}$.
\end{lemma}
\begin{proof}
Showing that $\sigma_i$ is divisible by $a_i$ is exactly equivalent to showing that it is divisible by $\kappa_e$ for every edge $e$ crossing level $i$ (since we work with saturated monoids, if an element is divisible by two integers, then it is also divisible by their least common multiple). But this is exactly condition~\eqref{p:r_t-3} in~\ref{prop:rub_translation}.
\end{proof}

Set $\tau_i\coloneqq \sigma_i /a_i\in \ghost_{B,b}$ (noting that division in $\ghost_{B,b}$ is unique since $\ghost_{B,b}$ is sharp, integral, and saturated), and define a monoid homomorphism
\begin{equation}\label{eq:hom_psi}
\psi\colon \tilde P \lra \ghost_{B,b}, \quad \psi\colon p_i \longmapsto \tau_i. 
\end{equation}

\begin{lemma} \label{lem:gdeltapsi}
The triangle 
\begin{equation}
 \begin{tikzcd}
\bb N\Span{\vedges} \arrow[r, "g"] \arrow[dr, "\delta"] & \tilde P \arrow[d, "\psi"]\\
& \ghost_{B,b}
\end{tikzcd}
\end{equation}
commutes. 
\end{lemma}
\begin{proof}
We compute  
\begin{equation}
\psi(g(e)) \= \psi\left(\sum_i \frac{a_i}{\kappa_e} p_i\right) \= \sum_i \frac{a_i}{\kappa_e} \tau_i \= \frac{1}{\kappa_e}\sum_i \sigma_i \= \frac{1}{\kappa_e} \left(\beta(v_+) - \beta(v_-)\right) \= \delta_e,
\end{equation}
where the last equality comes from the fact that $\beta$ is a PL function. 
\end{proof}

\begin{df}
\label{def:basicness}We say $(X/B, \beta)$ for~$B$ nuclear is \emph{basic} if the natural map 
$$\psi \oplus \delta|_{E^h} \colon \tilde P \oplus \bb N \Span{E^h} \lra \ghost_{B,b}$$
is an isomorphism. In general we say a point of $\cat{Rub}_{\ul 0}$ is \emph{basic} if it is so on a nuclear cover. 
\end{df}
Our motivation for introducing this definition lies in~\ref{lem:basic_is_minimal}. The intuition behind the definition is that $\ghost_{B,b}$ is precisely big enough to contain the elements that are necessary to accommodate the images of the maps $\delta$, the differences of images of $\beta$, and roots of these differences whose existence is required by condition~\eqref{p:r_t-3} in~\ref{prop:rub_translation}. 

\begin{lemma}
Every point of\, $\cat{Rub}_{\ul 0}$ comes with a map to a basic object. 
\end{lemma}
\begin{proof}
For $(X/B, \beta)$ a nuclear point of $\cat{Rub}_{\ul 0}$, we define a sheaf of monoids $P$ on $B$ as the fiber product
\begin{equation}
P \coloneqq \left(\tilde P\oplus \bb N \Span{E^h}\right) \times_{\ghost_B} \M_B. 
\end{equation}
This $P$ comes with a map $P \to \calO_B$, namely the composition of the
projection to the second factor $\M_B$ and the old log structure map
$\M_B \to \calO_B$,  making it into a log structure. 

This uses that, for \emph{any} nuclear point $(X/B, \beta)$, the map $\psi \oplus \delta|_{E^h}$ from the definition above satisfies that the preimage of $0 \in \ghost_B$ is $0 \in \tilde P\oplus \bb N \Span{E^h}$. From this it also follows that, at any point $b$ in the closed stratum of $B$, the stalk at $b$ of the ghost sheaf $\overline{P}$ of $P$ is given by 
\[
\overline{P}_b \= \tilde P\oplus \bb N \Span{E^h}.
\]
Now we make $(\ul B, P)$ into a point of $\cat{Rub}_{\ul 0}$: we take the underlying family $\ul X / \ul B$ of curves, and equip $\ul X$ with a log structure making it a log curve over $(\ul B, P)$ with length map
\[
\widetilde \delta \colon E(\Gamma) \lra \tilde P\oplus \bb N \Span{E^h}, \quad e \longmapsto 
\begin{cases}
\left(\sum_{i=\ell(e^-)}^{\ell(e^+)-1} \frac{a_i}{\kappa_e} p_i, 0 \right) & \text{ for }e \in E^v,\\
\left(0, e \right) & \text{ for }e \in E^h.
\end{cases}
\]
With this we obtain a family of log curves $(\widetilde X/(\ul B, P))$. Using~\ref{prop:rub_translation} we then lift to a $(\ul B, P)$-point of $\cat{Rub}_{\ul 0}$ by specifying the combinatorial PL function
\begin{equation}
  \beta\colon V(\Gamma) \lra \left(\tilde P\oplus \bb N \Span{E^h}\right)^\gp,
  \quad v \longmapsto -\sum_{j = \ell(v)}^{-1}a_j p_j.
\end{equation}
The construction gives a map from $(X/B, \beta)$ to this basic object $(\ul B,P) \to \cat{Rub}_{\ul 0}$.
\end{proof}

\begin{lemma}
  \label{lem:basic_is_minimal} The $\cat{Rub}_{\ul 0}$-point $(X/B, \beta)$ is basic
  if and only if it is minimal. 
\end{lemma}

\begin{proof}
Our proof mimics closely that of \cite[Theorem 4.2.4]{MarcusWiseLog}. To show that basic objects are minimal, consider a diagram as in~\ref{eq:def_of_minimal}; the problem comes down to showing there is a unique dashed arrow making the diagram 
\begin{equation}
    \begin{tikzcd}
        \ghost_{B'} & & \tilde P \oplus \bb N\Span{E^h} \ar[ll] \ar[dl, dashed]\\
       &  \ghost_{B''}\ar[ul] &
    \end{tikzcd}
\end{equation}
commute. The existence
of this arrow comes from the fact that we can apply the same formula~\ref{eq:hom_psi} to define the map, and the arrow is unique since the image of $\bb N\Span{E}$ has finite index in $\tilde P\oplus \bb N\Span{E^h}$, and division is unique in sharp, integral, saturated monoids. Conversely, applying~\ref{eq:hom_psi} shows that every minimal monoid admits a map from one which is basic (and hence minimal), and the definition of minimality furnishes an inverse to this map.
\end{proof}

\begin{df}
Let $\cat{Rub}_{\ul 0}' $ be the full subcategory of $\cat{Rub}_{\ul 0}$ whose objects have minimal log structure, viewed as a fibered category over $\cat{Sch}$ via forgetting the log structure and the curve.
\end{df}

As explained in~\ref{sec:minimal_ls_for_rub}, Gillam's minimality machinery immediately yields the main theorem of this section, slightly refining the results of \cite{MarcusWiseLog}.

\begin{theorem}\label{thm:underlying_stack_of_rub}
The underlying algebraic stack of\, $\cat{Rub}_{\ul 0}$ is given by $\cat{Rub}_{\ul 0}'$.
\end{theorem}

\subsection{Smoothness of $\cat{Rub}_{\ul 0}$}\label{sec:smoothness_of_Rub}

With the preparations above, we can now prove~\ref{thm:rub_smooth}, stating that the algebraic stack $\cat{Rub}_{\ul 0}$ is smooth.

\begin{proof}[Proof of~\ref{thm:rub_smooth}]

We write $k$ for the base ring (which the reader may take to be $\bb C$, but we are careful to make this proof work in any characteristic). We equip $\on{Spec} (k)$ with trivial log structure. 

We write $\cat{Rub}^{\MW}_{\ul 0}$ for the variant of $\cat{Rub}_{\ul 0}$ considered by Marcus and Wise \cite{MarcusWiseLog}; this is the same as our space except that they drop condition~\eqref{d:r-2} of the definition. The map $\cat{Rub}^{\MW}_{\ul 0} \to \mathfrak M$ is proven in \cite[Lemma 4.2.5 and Corollary 5.3.5]{MarcusWiseLog} to be log \'etale; hence $\cat{Rub}^{\MW}_{\ul 0}$ is log smooth (over $\on{Spec} (k)$ with trivial log structure). 

Let $p$ be a geometric point of $\cat{Rub}_{\ul 0}$ mapping to a geometric point $p'$ of $\cat{Rub}^{\MW}_{\ul 0}$, and write $\bar P$ and $\bar Q$ for the stalks of the respective characteristic monoids. The map $\bar Q \to \bar P$ is injective and has finite cokernel; it corresponds to taking roots of suitable parameters in order to make condition~\eqref{d:r-2} in the definition of $\cat{Rub}_{\ul 0}$ be satisfied. 

The log smoothness of $\cat{Rub}^{\MW}_{\ul 0}$ implies that there exist a scheme $U$ and smooth strict morphisms $f\colon U \to \cat{Rub}^{\MW}_{\ul 0}$ and $g\colon U \to \on{Spec} (k[\bar Q])$
such that $p'$ lies in the image of $f$. Since $\cat{Rub}_{\ul 0}$ is obtained from $\cat{Rub}^{\MW}_{\ul 0}$ by taking the roots that transform $\bar Q$ into $\bar P$, we have a diagram of pullback squares
\begin{equation}
 \begin{tikzcd}
 \cat{Rub}_{\ul 0}\ar[d] & V \ar[l] \ar[d] \ar[r] & \on{Spec} (k[\bar P])\ar[d] \\
 \cat{Rub}^{\MW}_{\ul 0} & U \ar[r] \ar[l] & \on{Spec} (k[\bar Q]). 
\end{tikzcd}
\end{equation}
Now~\ref{def:basicness},~\ref{lem:basic_is_minimal}, and~\ref{thm:underlying_stack_of_rub} together imply that the stalks of the characteristic monoid of $\cat{Rub}_{\ul 0}$ are free monoids of finite rank; in other words, $\on{Spec} (k[\bar P])$ is an affine space over $k$, in particular is smooth over $k$. 
\end{proof}

Note that the base-change $\cat{Rub}_\calL$ is \emph{not} in general smooth, except in genus zero (when the map $\mathfrak M \to \on{Pic}$ is an open immersion, hence smooth). For example, it can have many non-reduced irreducible components; see \cite{HolmesSchmitt}. In particular, the smoothness of the main component of $\cat{Rub}_{\calL_\mu}$ (proven in \cite{LMS} granting the verification that the spaces named $\LMS$ there and in~\ref{prop:smoothDM} indeed agree) does not follow directly from~\ref{thm:rub_smooth} outside of genus zero.

\subsection{Relative automorphisms}\label{sec:log_automorphism_example}

As a log stack, $\cat{Rub}_{\ul 0}$ has trivial automorphisms relative to the stack of log
curves. But (as discussed in \ref{Rmk:setoid_warning}) this does not mean that the underlying
algebraic stack of minimal objects has trivial automorphisms. Rather, they come
from automorphisms of the log structure; the following remark makes this precise. 
\begin{remark}
In general, given a map $\calX \to \calY$ of log stacks with underlying stacks $\ul{\calX}, \ul{\calY}$ and a point $\ul x \colon \mathrm{Spec}(\bb C) \to \ul{\calX}$, we can ask: what is the relative inertia of $\ul x$ over $\ul y = (\ul{\calX} \to \ul{\calY}) \circ \ul x$? For this, let $(\mathrm{Spec}(\bb C), \M_x) \to \calX$ and $(\mathrm{Spec}(\bb C), \M_y) \to \calY$ be the minimal log structures lifting $\ul x, \ul y$. Then by minimality of $\M_y$, the composition $(\mathrm{Spec}(\bb C), \M_x) \to \calX \to \calY$ must factor through a map
\[
f \colon (\mathrm{Spec}(\bb C), \M_x) \lra (\mathrm{Spec}(\bb C), \M_y).
\]
Such a map is uniquely described by a monoid map $\M_y \to \M_x$ over $(\bb C, \times) = (\calO_{\mathrm{Spec}(\bb C)}, \times)$. Then the desired group of automorphisms is just the group of those automorphisms of $\M_x$ that are constant on the image of $\M_y$ and commute with the map to $(\bb C, \times)$.
\end{remark}

Returning to our situation, the `tropical'
part of the log structure (the ghost sheaf $\ghost$) has no non-trivial
automorphisms relative to the stack of log curves. Thus the automorphisms all arise from automorphisms of the log
structure $\M$ that are trivial on $\ghost$ and trivial on the structure sheaf.
So they are really automorphisms of the extension structure of $\M$.

\subsection{The worked example again}
Let $(X/\bb C, \beta\in \ghost^\gp_X)$ be a point of $\cat{Rub}_{\ul 0}$ with the
underlying enhanced level graph given by~\ref{fig:X_L_pic}, still
restricting to the case $\kappa_1 = \kappa_2 = 1$ and $\kappa_3 = n$. We would
like to understand the relative inertia of this point of $\cat{Rub}_{\ul 0}$
over~$\mathfrak M$. 

The minimal monoid on $\bb C$ for the curve $X/\bb C$ is just $\bb N\Span{E} = \bb N\Span{e_1, e_2, e_3}$, and the minimal monoid as a point in $\cat{Rub}_{\ul 0}$ is given by $\tilde P = \bb N\Span{p_{-1}, p_{-2}}$, with one generator $p_i$ for each \emph{level} $i$ (there are no horizontal edges in this example; otherwise, they should also appear in this minimal monoid). The natural map is then given by
\bes
g\colon \bb N\Span{E} \lra \tilde P, \quad e_1 \longmapsto np_{-1},\; \; e_2 \longmapsto np_{-2}, \;\; e_3 \longmapsto   p_{-1} + p_{-2}.
\ees
To see this, note that $a_1 = a_2 = n$, and then apply Formula~\ref{eq:g_delta_e}. Note that there are no non-trivial automorphisms of $\tilde P$ that act as the identity on the image of $g$. The map $g$ extends in the obvious manner to a map on the stalks of the log structures
\bes
\bb N\Span{E} \oplus \bb C^* \lra P =  \tilde P \oplus \bb C^*,
\ees
and the relative inertia is then given by the automorphisms of $\tilde P\oplus \bb C^*$ which act as the identity on the image of $\bb N\Span{E} \oplus \bb C^*$, and which lie over the identity map on $\tilde P$ (since any automorphism of $\tilde P$ constant on the image of $g$ must be the identity). Such an automorphism is defined on $ ((1,0),1)$ and $ ((0,1),1)$ by
\begin{equation*}
((1,0),1) \longmapsto ((1,0),u) \quad \text{ and } \quad ((0,1),1) \longmapsto ((0,1),v)
\end{equation*}
for some $u$, $v \in \bb C^*$ satisfying
\begin{enumerate}
\item
$u^n = 1$, because $n((1,0),1) = ((n,0), 1^n)$ lies in the image of
$\bb N\Span{E} \oplus \bb C^*$ and is thus fixed;
\item $v^n = 1$ for the analogous reason;
\item $uv = 1$, because $((1,1),1)$ lies in the image of $\bb N\Span{E} \oplus \bb C^*$ and is thus fixed.
\end{enumerate}
Such a choice of $u$, $v$ evidently determines such an automorphism.
(Or, more precisely, there are two canonical isomorphisms with the roots
of unity, one coming from `above' and the other from `below' on the graph,
and the composite of these isomorphisms is the inversion map on the group
of roots of unity).

We conclude that the \emph{relative inertia for this triangle graph is
  equal to the group~$K_\Gamma$ computed in~{\rm\ref{eq:KGammaexample}}.}

\section{From logarithmic to multi-scale}\label{sec:log_to_GMS}

In this section we construct the morphism of functors~$F\colon \cat{Rub}_{\calL_\mu} \to \GMS$ whose existence is claimed in~\ref{intro:mainiso}, and then prove that theorem. At the end of the section, we include two related results about describing the multi-scale space as a Zariski closure and describing a morphism from the rubber space to the Hodge bundle, which can be of independent interest.
 
Let $(X/B, \beta\in \Gamma(X,\ghost_X^\gp), \phi\colon \calO_X(\beta) \isom {\calL_\mu}) \in \cat{Rub}'_{{\calL_\mu}}$.  Recall that the prime on $\cat{Rub}$ indicates that we are working with the minimal log structure as described in~\ref{sec:minimal_log_str}, and that we always work with \emph{saturated} log structures.

We assume moreover for now, and for most of this section, that $X/B$ is nuclear,
and explain at the end why the functor glues to general families.

\subsection{The enhanced level graph} \label{eq:enhLG}

The first item to build the $F$-image of $(X/B,\beta,\phi)$ is an
enhanced level graph.
As the underlying graph~$\Gamma$, we simply take the dual graph of the curve fiber over the closed stratum
of~$B$. The level structure, given  in terms of a normalized level function,
comes from $\beta\in \ghost_X^\gp(X)$ as explained in~\ref{eq:defell}. The definition
of the enhancement~$\kappa$ is given in~\ref{eq:defkappa}, where the divisibility
required for this definition has been proven in~\Cref{le:divisibility}. The stability
condition just comes from the fact that we work with stable curves.

Given a vertex $v$ and the corresponding component $X_v$ of the central fiber, the
admissibility of $\kappa$ comes down to the equalities
  \be \label{eq:slopeadmiss}
  2g(v) - 2 + \#H'(v)  - \sum_{j\mapsto v} m_j  \= \deg\left({\calL_\mu}|_{X_v}\right)
  \=   \sum_{h \mapsto v}  \kappa_h
  \ee
(recall that $H'(v)$ denotes the set of non-leg half-edges attached to a vertex $v$). 
  The first equality is immediate from the definition of $\ca L_\mu$, and the second comes from the isomorphism $\phi$ and a computation of the degree of $\ca O_X(\beta)$ on the component $X_v$ presented in~\ref{lem:Obetarestriction}.

The dual graph $\Gamma_{b'}$ of the fiber over a general $b'$ (possibly outside the closed stratum) comes with a level structure obtained from $\Gamma$ by undegeneration (as defined in~\ref{sec:tori}), by the Key Property~\eqref{Kprop-4} of nuclear log curves from~\ref{sec:unpacking_rub_definition}. Constructing the rest of the data of a generalized multi-scale differential requires more work, which we now begin.

\subsection{Logarithmic splittings and rotations}

We write $\tilde P = \bb N\Span{p_{-1}\dots, p_{-N}}$ as
in~\ref{sec:minimal_log_str}.

\begin{df}\label{df:logsplitting}
A \emph{log splitting} is a map
\begin{equation}
\lspl \colon \tilde P \lra \M_B 
\end{equation}
whose composition with the canonical map $\M_B \to \ghost_{B,b}$ is
the map $\psi\colon \tilde P \hra \ghost_{B,b}$ from~\ref{eq:hom_psi}
(recall that we work throughout this section with minimal objects). 

The \emph{simple log level rotation torus} $T^s_{log}$, abbreviated as \emph{simple LLRT}, is the set of log splittings.
\end{df}

\begin{remark}
Let us unpack the simple log level rotation torus definition a bit. Recall our key exact sequence~\ref{eq:stdsequence}. The presence of the $\gp$ is not so important, as we always work with integral monoids (\textit{i.e.}~monoids which inject into their groupifications).  Consequently, a choice of a splitting is essentially a choice of an invertible function on $B$ (which we think of as a scalar) for every level below $0$. Pre-composing~$\tilde \psi$ with the map $g$ from~\ref{eq:g_delta_e} and using~\ref{lem:gdeltapsi}, we then also obtain a lift of the map $\delta$, \textit{i.e.}~a choice of a scalar for each edge. These must satisfy appropriate compatibility equations, and the saturation condition also imposes the existence of certain roots.
\end{remark}

\begin{df}
The \emph{simple log rotation group} is the group 
\[\on{Hom}_{\rm mon}\left(\tilde P, \calO_B^\times(B)\right)
\= \on{Hom}_{\gp}\left(\tilde P^{\gp}, \calO_B^\times(B)\right),\]
where the identification stems from the universal property of the groupification.\footnote{Note that there is also a (non-simple) log rotation group, consisting of the set of compatible choices of elements in $\calO_B^*(B)$ for all $e \in E^v$ and the elements $\sigma_i = \beta(v_i) - \beta(v_{i-1})$. Since this non-simple group will not be needed in the following, we do not give a formal definition.}
\end{df}

We define an action of an element $\phi$ of the simple log rotation group on
the simple log rotation torus by the formula
\be \label{eq:logrotaction}
\left(\phi \cdot \lspl\right)(p) \coloneqq \phi(p)\lspl(p) \quad \text{for $p \in
  \tilde P$}.
\ee

\begin{lemma} \label{le:pseudotorsor}
Via the action~{\rm\ref{eq:logrotaction}}, the simple LLRT is
either empty, or a torsor for the simple log rotation group. After possibly shrinking $B$, we can ensure the existence of a log splitting. 
\end{lemma}

Recall that a pseudo-torsor is a space with a free transitive action, but
unlike a torsor, it may be empty (here, if the base~$B$ is too large
to support the appropriate sections). Thus \Cref{le:pseudotorsor} says that the
simple LLRT is a pseudo-torsor.

\begin{proof}
In the exact sequence
\[
1 \lra H^0\left(B, \calO_B^\times\right) \lra H^0\left(B,\M_B^\gp\right) \lra \underbrace{H^0\left(B, \ghost_B^\gp\right)}_{\=\ghost_{B,b}^\gp} \lra H^1(B, \calO_B^\times) \lra \cdots, 
\]
if all elements $\psi(p_i)=\tau_i \in \ghost_{B,b}$ map to zero in $H^1(B, \calO_B^\times)$, then they have preimages in $H^0(B,\M_B)$; by the freeness of $\ghost_{B,b}$, this implies the existence of a log splitting. Any such choices of preimages differ precisely by elements in $H^0(B, \calO_B^\times)$, which together define an element of the simple log rotation group. Thus the action of this group is free and transitive. 

Finally, if the elements $\tau_i \in \ghost_{B,b}$ do \emph{not} map to zero in $H^1(B, \calO_B^\times) = \Pic(B)$, we can always find an open neighbourhood $B_0$ of $b \in B$ on which these line bundles are trivial after all. Then on $B_0$, the long exact sequence and the argument above show the existence of a lift, finishing the proof.
\end{proof}

\subsection{Log viewpoint on smoothing and rescaling parameters}

In this subsection we construct the rescaling ensemble from the choice of
a log splitting, and provide auxiliary statements about the smoothing and
rescaling functions contained in the ensemble.

Let $\lspl \colon \tilde P \to \M_B$ be a log splitting. Recall the definitions 
of the maps $g\colon \bb N\Span{E^v} \to \tilde P$ from~\ref{eq:g_delta_e} and
 $\alpha\colon \M_B \to \calO_B$ from the definition of a log scheme.

\begin{df}
The \emph{smoothing parameter associated
to a vertical edge $e\in E^v(\Gamma)$ by the log splitting $\lspl$} is
\begin{equation}\label{eq:logfedef}
f_e \coloneqq (\alpha \circ \widetilde{\psi} \circ g)(e). 
\end{equation}
Fix a level $i\in L(\Gamma)$. The \emph{level parameter} and \emph{rescaling
parameter} associated to $i$ by $\lspl$ are
\begin{equation}\label{eq:logstdef}
t_i \coloneqq (\alpha \circ \widetilde{\psi})(p_i) \quad \text{and} \quad
s_i \coloneqq (\alpha \circ \widetilde{\psi})(a_i p_i).  \qedhere
\end{equation}
\end{df}
The collection of functions $\bft = (t_i)_{i\in L(\Gamma)}$ defines a map
$R^s\colon B \to \ol{T}^s_\Gamma$ to the closure of the simple level rotation
torus, which is just $\CC^N,$ and a rescaling parameter $s_i = r_i \circ
\pi \circ R^s$ in the notation of~\ref{sec:germfam}.

\begin{lemma} \label{lem:getsimplescalingens}
The morphism $R^s\colon B \to \ol{T}^s_\Gamma$ defined above is a simple
rescaling ensemble.
\end{lemma}
\begin{proof}
By~\ref{def:RescEns} we must verify that the functions $f_e$ from~\ref{eq:logfedef} are indeed smoothing parameters for their respective nodes, lying in the correct equivalence class in $\calO_B / \calO_B^\times$. To see this, consider the following diagram: 
\[
\begin{tikzcd}
\bb N \Span{E^v} \arrow[r,"g"] \arrow[ddd,"\delta"] & \tilde P \arrow[dd,"\tilde \psi"] & & \\
& & \calO_B^\times \arrow[ld] &\calO_B^\times\rlap{.} \arrow[ld] \\
& \M_B \arrow[r,"\alpha"] \arrow[dl] & \calO_B \arrow[dl] &\\
\ghost_{B,b} \arrow[r,"\overline \alpha"] & \calO_B / \calO_B^\times & & 
\end{tikzcd}
\]
What we must show is that $f_e = (\alpha \circ \widetilde{\psi} \circ g)(e) \in \calO_B$ maps to the class of a smoothing parameter in $\calO_B / \calO_B^\times$. Now the commutativity of the upper-left quadrilateral follows from~\ref{lem:gdeltapsi} and the assumption that $\tilde \psi$ lifts the map $\psi\colon \tilde P \to \ghost_{B,b}$. On the other hand, the map $\overline \alpha$ is just \emph{defined} to make the lower quadrilateral commute. Then we have
\[
  [f_e] \= (\alpha \circ \widetilde{\psi} \circ g)(e) \=
  \overline \alpha(\delta(e)) \in \calO_B / \calO_B^\times.
\]
The fact that $\overline\alpha$ maps $\delta(e)$ to a smoothing parameter for the node associated to $e$ is then just a basic property of families of log curves; see Key Property~\eqref{Kprop-2} of~\ref{sec:unpacking_rub_definition}. 
\end{proof}

\subsection{The collection of rescaled differentials}
\label{subset:rescaled_from_log}

By the definition of lying in $\cat{Rub}_{\calL_\mu}$, we are given an isomorphism
\begin{equation}
\phi\colon \omega_{X/B}\left(-\sum_{k=1}^n m_k z_k\right)\longisom \calO_X(\beta).
\end{equation}

On the other hand, it follows from the definition of $\psi$ that the element $-\sum_{j=i}^{-1} a_j p_j \in \tilde P^\gp$ maps to $\beta(v_i) \in \ghost_{B,b}^\gp$ under $\psi$, where $v_i \in V(\Gamma)$ is any vertex on level $i$. Using the log splitting $\tilde \psi$, we obtain the elements
\[
o_i \coloneqq \tilde \psi\left(- \sum_{m=i}^{-1} a_m p_m \right) \in \M_B^\gp
\]
in the preimage of $\beta(v_i)$. Since this preimage can be identified as the complement of the zero section in $\calO_B(\beta(v_i))$, we can see $o_i$ as a nowhere-vanishing section of $\calO_B(\beta(v_i))$.

Adapting the convention from~\ref{def:collRD} to the family $X/B$, we define $X_{(>i)} \subseteq X$ to be the closed subset of components of fibers $X_{b'}$ whose closure in $X$ does not intersect any component of the central fiber~$X_b$ at a level less than or equal to $i$.
Then we define $U_i = X \setminus (X_{(>i)} \cup Z^\infty)$, where $Z^\infty \subseteq X$ is the image of sections associated to marked poles. Then we claim that there is a well-defined map
\begin{equation} \label{eqn:betavitobeta}
w_i\colon \pi^* \calO_B(\beta(v_i))|_{U_i} \lra \calO_X(\beta)|_{U_i}.
\end{equation}
Indeed, the left-hand side is the line bundle on $U_i$ associated to the piecewise linear function which is \emph{constant}, equal to $\beta(v_i)$. Since we removed $X_{(>i)}$, this function dominates the function $\beta$ on the right, so we have a map as desired. Thus $w_i(\pi^* o_i)$ gives a section of $\calO_X(\beta)$ on $U_i$, and we define
\be \label{eq:defomegafromlogspl}
\omega_{(i)} \,\coloneqq \, \phi^{-1} w_i(\pi^* o_i) \in H^0\left(U_i, \omega_{X/B}\left(-\sum_{k=1}^n m_k z_k\right) \right).
\ee

We check that $\omega_{(i)}$ satisfies the conditions in~\ref{def:collRD} and that the smoothing and rescaling parameters  $f_e$ and $s_i$ defined in~\ref{eq:logfedef} and~\ref{eq:logstdef} (and thus the simple rescaling ensemble~$R^s$) are compatible with these generalized rescaled differentials.
\begin{enumerate}
\item
For any levels $j < i <0 $, there is a natural map of line bundles
$\calO_B(\beta(v_i)) \to \calO_B(\beta(v_j))$. 
On the level of sections, we then have 
\[
o_i \= \tilde \psi\left(- \sum_{m=i}^{-1} a_m p_m\right) \= \tilde \psi\left(- \sum_{m=j}^{-1} a_m p_m\right) \cdot \prod_{m=j}^{i-1} \tilde \psi(a_m p_m) \longmapsto o_j \cdot \prod_{m=j}^{i-1} \tilde \psi(a_m p_m). 
\]
Via the isomorphism $\varphi^*$, and using that $s_m = \alpha(\tilde \psi(a_m p_m))$, this becomes the desired equality $\omega_{(i)} = \omega_{(j)} \cdot \prod_{m=j}^{i-1} s_m$. The fact
that $s_i$ vanishes at the closed point of $B$ comes from the fact that the
map of line bundles is the zero map when restricted to the fibers over the
closed point of $B$.
\item Choose local coordinates $u_e$ and $v_e$ so that $X$ is locally given by $u_ev_e = f_e$. Then the isomorphism $\phi$ corresponds near the node to a non-vanishing section of $\omega_{X/B}(-\beta)$. Now $\omega$ has a local generating section which is given after inverting $u_e$ by $\frac{du_e}{u_e}$ and after inverting $v_e$ by $-\frac{dv_e}{v_e}$, and (perhaps after adjusting the choices of coordinate $u_e$ and $v_e$) $\ca O_X(-\beta)$ has a local generating section which is given after inverting $u_e$ by $u_e^{\kappa_e}o_i$ and after inverting $v_e$ by $o_jv_e^{-\kappa_e}$. By dividing the just-described local generating section by the one given by the isomorphism $\phi$, we obtain a local unit $\lambda$ such that 
\begin{equation}
\omega_{(i)} = \lambda u_e^{\kappa_e}  \frac{du_e}{u_e} \quad \text{and} \quad \omega_{(j)} = - \lambda v_e^{-\kappa_e}  \frac{dv_e}{v_e}\, . 
\end{equation}

\item  On the normalization $Y_i$ of the union of all components of the special fiber $X_b$ sitting at level $i$, we have (see~\ref{lem:Obetarestriction})
\bes
\calO_X(\beta)|_{Y_i} \= \pi^*\calO_B(\beta(v_i)) \otimes_{\calO_{Y_i}}
\calO_{Y_i}\left(\sum_h \kappa_h h\right),
\ees
where the sum is taken over all non-leg half-edges~$h$ attached to the vertices
at level $i$.
Pulling back via $\varphi^*$, the line bundle on the left becomes
\[
\omega_{X_b}\left(-\sum_{k=1}^n m_k z_k\right)|_{Y_i} \=
\omega_{Y_i}\left(-\sum_{k=1}^n m_k z_k + \sum_h h\right).
\]
Tensoring with $\calO_{Y_i}(-\sum_h \kappa_h h)$ on both sides, we then get
\[
\omega_{Y_i}\left(-\sum_{k=1}^n m_k z_k - \sum_h (\kappa_h -1) h\right) \cong \pi^*\calO_B(\beta(v_i)).
\]
Seeing $\omega_{(i)}$ as a meromorphic section on the left, it then corresponds to the nowhere-vanishing section $\pi^* o_i$ on the right. Thus it extends to all of $Y_i$ on the left-hand side. But then this extension seen as a meromorphic section of $\omega_{Y_i}$ has the desired order
$m_k$ at the marked points $z_k$ and $\kappa_h-1$ at the preimage of the node associated to $h$. 
\end{enumerate}

\subsection{Prong-matchings}

To recall the notion of a prong-matching, consider a vertical edge $e \in E^v$,
and let $B_e \hookrightarrow B$ be the closed subscheme of $B$ over which the
node~$e$ persists, \textit{i.e.}~the vanishing locus of the smoothing parameter~$f_e$.

The sections~$q^\pm$ of the two preimages of the node identify $B_e$ as
a subscheme of the blowup of $X \times _B B_e$ along the section corresponding
to $e$. Recalling~\ref{eq:defNe}, we let $\calN^\vee_e \= (q^+)^*\omega_{X_+}\otimes (q^-)^*\omega_{X_-}$ be the corresponding line bundle on $B_e$. Then a local prong-matching at $e$ is a section
$\sigma_e$ of $\calN_e^\vee$ such that $\sigma_e^{\kappa_e}(\tau_e)=1$ for the section $\tau_e \in \calN_e^{\kappa_e}$ defined in~\ref{lem:prong_matching_comparison}.

To identify this notion in the logarithmic context, recall that we have the element $\delta(e) \in \ghost_{B,b}$. Then the bundle $\calN_e^\vee$ has an interpretation as follows. 

\begin{lemma} \label{le:nbiso}
There are canonical isomorphisms of line bundles
\begin{equation}\label{eq:O_delta_eq_W}
\calO_B(\delta(e))|_{B_e} \= \calN^\vee_e
\end{equation}
for each edge~$e$.  Moreover, let $\mathfrak{f} \in \M_B$ be an element mapping to $\delta(e) \in \ghost_{B,b}$, so that we can see it as a section of $\calO_B(\delta(e))$.  Then the function $f=\alpha(\mathfrak{f}) \in \calO_B$ is a smoothing parameter for the node associated to $e$. Let $u,v$ be local coordinates around the node on $X$ such that the Henselized local ring at the node is the Henselization of $\calO_B[u,v]/(uv-f)$. Then the isomorphism~{\rm\ref{eq:O_delta_eq_W}} sends the section $\mathfrak{f}|_{B_e} \in \calO_B(\delta(e))|_{B_e}$ to
$$
du \otimes dv \in \calN_e^\vee \= (q^+)^*\omega_{X_+}\otimes (q^-)^*\omega_{X_-}.
$$
\end{lemma}

\begin{proof}
Since both sides commute with base-change, it is enough to check this in the
universal case, in which the log structure is divisorial coming from the
boundary (and the map $\alpha$ of the log structure is injective, so there
are no non-trivial automorphisms of the log structure). Over a versal
deformation $R$, the local equation of the node is given by $R[u,v]/(uv-f)$,
where $f \in R$ is an element corresponding to $\delta(e)$. So $\calO_B(\delta(e))$
is canonically identified with the ideal sheaf generated by~$f$ in $R$
(\textit{cf.} the appendix). On the other hand, $\ca{N}^\vee_e$
is canonically identified with the conormal bundle in $R$ to the locus $f=0$
(see \cite[Section~XIII.3]{acgh2}) and thus agrees with $\calO_B(\delta(e))|_{Z_e}$. Tracing through the constructions of these canonical identifications yields the second part of the lemma; alternatively, this can be seen as a very slight generalization of \cite[Section 4]{edixhoven1998neron}, where his $c(x)$ corresponds to the element $du \otimes dv$ and his $\pi^{x(e)}$ to the element $f$. 
\end{proof}

Let $\lspl\colon \tilde P \to \M_B$ be a log splitting, and let $e$ be a
vertical edge. By~\ref{lem:gdeltapsi} the element $(\lspl \circ g)(e) \in \M_B$
maps to $\delta(e) \in \ghost_{B,b}$ and hence lies in $\calO_B^\times(\delta(e))
\sub \M_B$ (by the definition of this bundle via~\ref{eq:stdsequence}).
Applying the isomorphism
of~\ref{eq:O_delta_eq_W}, we thus obtain a section of $\calN^\vee_e$.
\begin{df}  \label{def:logPM}
We call the section $\sigma_e = (\lspl \circ g)(e)|_{B_e} \in H^0(B_e, \calN^\vee_e)$ the \emph{local prong-matching $\sigma_e = \sigma_e(\lspl)$
at~$e$ determined by the log splitting}. The collection
$\bfsigma = (\sigma_e)_{e \in E^v(\Gamma)}$ of these is called the
\emph{global prong-matching determined by the log splitting}.
\end{df}

There are two compatibility statements to check for this definition: to
get a prong-matching, see the discussion after~\ref{eq:defNe}, and to make this
part of a multi-scale differential, see~\ref{def:germMSD}\eqref{d:gMSD-4}.

\begin{lemma}
The prong-matching $\bfsigma$ determined by any log splitting is indeed a
prong-matching in the sense of \ref{df:PMfinal}. 
\end{lemma}

\begin{proof}
Assume that the vertical edge $e$ connects levels $i>j$ in $\Gamma$. From~\ref{df:PMfinal}, we need to show that $\sigma_e^{\kappa_e}(\tau_e) = 1$, where~$\tau_e$ is the
section of $\calN_e^{\kappa_e}$ defined as
$\tau_e = (q^+)^* \omega_{(i)}^{-1} \otimes (q^-)^* \omega_{(j)}$.

On the other hand, the differentials $w_{(i)}$ and $w_{(j)}$ are also determined
in~\ref{eq:defomegafromlogspl} by the formulae
\[
\omega_{(i)} \= \phi^* w_i\left(\pi^* \tilde \psi\left(- \sum_{m=i}^{-1} a_m p_m\right)\right) \quad
\text{and} \quad\omega_{(j)} \= \phi^* w_j\left(\pi^* \tilde \psi\left(- \sum_{m=j}^{-1} a_m p_m\right)\right).
\]
We put this into the formula for $\tau_e$; the pullbacks $(q^+)^*, (q^-)^*$ cancel the pullback $\pi^*$. Interpreting~$\tau_e$ as a section of $\calO_B(-\kappa_e \delta(e))$ via~\ref{eq:O_delta_eq_W}, we thus have
\[
\tau_e \= \tilde \psi \left(\sum_{m=i}^{-1} a_m p_m - \sum_{m=j}^{-1} a_m p_m \right)
\= \tilde \psi \left(- \sum_{m=j}^{i-1} a_m p_m \right)
\= \tilde \psi \left(- \kappa_e g(e) \right) \= \sigma_e^{- \kappa_e}.
\]
Here in the second to last equality, we used the definition of $g$ from~\ref{eq:g_delta_e}. This finishes the proof that $\sigma_e^{\kappa_e}(\tau_e)=1$, and thus that $\sigma_e$ is a local prong-matching.
\end{proof}

\begin{lemma}
Let $\lspl\colon \tilde P \to \M_B$ be a log splitting and $e$ a non-semi-persistent vertical node $($i.e.~$f_e^{\kappa_e} \neq 0)$. Then the local prong-matching determined by $\lspl$ is equal to that induced in~\ref{rem:induced_prong_matching}. 
\end{lemma}

\begin{proof}
The local prong-matching $\sigma_e$ of~\ref{rem:induced_prong_matching} is constructed
by writing the local equation of the node as $uv=f_e$ and setting 
$$\sigma_e \,\coloneqq \,du \otimes dv \in \calN_e^\vee
\= (q^+)^*\omega_{X_+}\otimes (q^-)^*\omega_{X_-}. $$
On the other hand, the local prong-matching $\widetilde\sigma_e$ associated to~$e$
by $\lspl$ is given by applying the isomorphism in~\ref{le:nbiso} to the
element $(\lspl \circ g)(e)$. 

Recalling that $f_e = (\alpha \circ \lspl \circ g)(e)$, we see that the desired equality $\sigma_e=\widetilde\sigma_e$ is then the second part of~\ref{le:nbiso}. 
\end{proof}

\subsection{Morphism of functors from rubber to multi-scale}
\label{sec:functor_rub_to_gms}

We put the above together to build a morphism of functors after restricting to base schemes which are locally of finite type over $\bb C$ (this restriction is harmless since $\cat{Rub}_{\calL_\mu}$ is representable by a locally finite-type Deligne--Mumford stack).  We first make this construction in a local situation, then globalize. We start with a family $(X/B, \beta\in \ghost_X(X), \phi)$, which we take to have minimal saturated log structure, and which is both nuclear and controlled. This immediately gives us the structure of an enhanced level graph. We \emph{choose} a log splitting $\lspl \colon \tilde P \to \M_B$ (perhaps after shrinking $B$).  This determines a simple rescaling ensemble, a collection of rescaled differentials, and induces local prong-matchings at each node. Hence we have a simple multi-scale differential.

We next claim that a different choice of log splitting yields an isomorphic simple multi-scale differential, together with a choice of isomorphism. Indeed, by~\ref{le:pseudotorsor} any two log splittings differ by the action of the simple LLRT, and one checks easily that the action of the simple LLRT corresponds to the action of the simple level rotation torus. 

For general $B$ locally of finite type, a family of multi-scale differentials is defined as a family of multi-scale differentials on a nuclear controlled cover, compatible on overlaps. 

It is clear from the constructions that the above map is independent of
choices and is compatible with shrinking the base~$B$; more precisely, given a map $B' \to B$ and a family of log differentials on $B$, we can either first apply our map (obtaining a family of multi-scale differentials on $B$) and then pull back to $B'$, or first pull back and then apply our map; unravelling the definitions yields that the results are canonically isomorphic. By descent we obtain a global morphism of functors $F\colon\cat{Rub}_{\calL_\mu}
\to \GSMS$. 

\subsection{Showing the map of functors induces an isomorphism}\label{sec:ess_surj}

The above construction gives a morphism  from the logarithmic space to the
multi-scale space. In this section we complete the proof of~\ref{intro:mainiso}
by showing that this functor induces an isomorphism. 

\begin{theorem}
The morphism
\begin{equation}
F \colon \cat{Rub}_{{\calL_\mu}} \lra \GSMS. 
\end{equation}
is an isomorphism. 
\end{theorem}
\begin{proof}
Given a map $B \to \GSMS$ with the implicit stable curve over $B$ being controlled, we show that there exists a unique map $B \to \cat{Rub}_{\calL_\mu}$ making the following diagram commute:  
\begin{equation}
    \begin{tikzcd}
    B \arrow[d, dashed] \arrow[dr] & \\
    \cat{Rub}_{\calL_\mu} \arrow[r] & \GSMS\rlap{.}
    \end{tikzcd}
\end{equation}
The claimed isomorphism in the global situation then follows by descent. 
Let $(\pi\colon X \to B, \bfz, \Gamma, R^s, \bfomega, \bfsigma)$ be the simple multi-scale differential corresponding to $B \to \GSMS$. Given $i \in L(\Gamma)$, we write $t_i \in \calO_B(B)$ for the composition with the appropriate coordinate projections $B \to \overline T^s \to \bb C$. 

Let $\M_B$ be the minimal log structure making $X/B$ into a log curve; in particular, its characteristic monoid $\ghost_{B,b}$ is canonically identified with the free monoid $\bb N \Span{E}$ on the edges of $\Gamma$. For each edge $e$, the local prong-matching $\sigma_e$ at $e$ is by~\ref{le:nbiso} a section of the fiber of the bundle $\ca O(\delta_e)$ over the locus $B_e$ where the node persists, and we choose a section in $\M_B(B)$ lifting $f_e$ and restricting over $B_e$ to $\sigma_e$, yielding a splitting
\begin{equation*}
    \mathfrak f \colon \ghost_{B,b} \lra \M_B. 
\end{equation*}

Denote by  $\tilde P \coloneqq \Span{p_{-1}, \dots, p_{-N}}$ the free monoid on the levels, as usual, and define 
\begin{equation}\label{eq:t_map}
    t\colon \tilde P \lra \calO_B, \quad p_i \longmapsto t_i, 
\end{equation}
and 
\begin{equation}
    t'\colon \tilde P \oplus \bb N\Span{E^h}\lra \calO_B,
\end{equation}
acting as $t$ on the first summand and as $\mathfrak f$ on the second. 

Then let 
\begin{equation}
    g'\colon \bb N\Span{E} \lra \tilde P \oplus \bb N\Span{E^h}
\end{equation}
be the map given by $g$ on the vertical edges and by the identity on the horizontal edges. 

The equalities
\begin{equation}\label{eq:f_and_t}
    f_e \= \prod_{i = \ell(e^-)}^{\ell(e^+) - 1}  t_i^{\frac{a_i}{\kappa_e}},
\end{equation}
which come from~\ref{eq:simpletotorus} (where coordinates were denoted by $f_e=\rho_e$ and $t_i=q_i$), imply that the diagram
\begin{equation}
    \begin{tikzcd}
    \M_B \arrow[r, "\alpha"] & \calO_B\\
    \ghost_B \arrow[u, "\mathfrak f"] \arrow[r, "g'"] & \tilde P \oplus\bb N\Span{E^h} \arrow[u, "t'"]
    \end{tikzcd}
\end{equation}
commutes. 

Now we define a sheaf of monoids $P$ as the pushout
\begin{equation} \label{eqn:minlogpushout}
    \begin{tikzcd}
    \M_B \arrow[r] & P\\
    \ghost_B \arrow[u, "\mathfrak f"] \arrow[r, "g'"] & \tilde P \oplus\bb N\Span{E^h}\rlap{,} \arrow[u]
    \end{tikzcd}
\end{equation}
which by the commutativity of the previous diagram comes with a map $\alpha_P\colon P \to \calO_B$. One checks easily that $P$ is in fact a log structure on $B$, with characteristic sheaf $\overline P = \tilde P \oplus \bb N \Span{E^h}$ at a point $b \in B$ in the closed stratum. The map $\M_B \to P$ gives $X/(B, P)$ the structure of a log curve, and mapping a vertex $v$ of level $i$ to the element
\begin{equation}
    \left(-\sum_{j = i}^{-1} a_j p_j, 0 \right) \in \left(\tilde P \oplus \bb N \Span{E^h}\right)^\gp
\end{equation}
defines a map $\beta\colon V \to \overline P^\gp$ so that the pair $(X/B, \beta)$ is a (minimal) point of $\cat{Rub}$. 

To lift this point to a point of $\cat{Rub}_{\calL_\mu}$, we need to build an isomorphism of line bundles 
\begin{equation}
\calO_X(\beta) \longisom \omega_{ X /B}\left(-\sum_{i=1}^n m_i z_i\right). 
\end{equation}
We first define this map on the smooth locus; let $p \in B$ and let $x \in X_p$ be a smooth point of $X_p$, lying in the component associated to a vertex $v \in \Gamma$.  Then the image of $\beta$ in $\ghost_{X, x}^\gp = \overline P_p^\gp$ is given by $\beta(v)$.  Our splitting $\overline{P} \to P$ from~\ref{eqn:minlogpushout} extends to $\overline{P}^\gp \to P^\gp$, and thus $\beta(v)$ maps to a unique section of $\calO_B(\beta(v)) \subseteq P$. Then we define
\begin{equation}\label{eq:iso_open}
\calO_X(\beta)_x \longisom \omega_{X/B, x}
\end{equation}
to be the unique map sending this section to the differential $\omega_{(\ell(v))}$. Next we check that this isomorphism extends over the nodes. We treat only the case of a vertical node $e$, say passing from a vertex $v_i$ of level $i$ to a vertex $v_j$ of level $j$ with $j < i$; then the map near a horizontal node can be constructed just as in the smooth case. Write $\mathfrak f'$ for the natural map $\tilde P \to P$, so that 
\begin{equation}
\beta(v_i) \= \mathfrak f'\left(- \sum_{m=i}^{-1} a_m p_m\right) \quad\text{and}\quad \beta(v_j)\= \mathfrak f'\left(- \sum_{m=j}^{-1} a_m p_m\right).
\end{equation}
Setting $\mathfrak f_e \coloneqq \mathfrak f'(g'(\delta_e))$,  we have
\begin{equation}
\beta(v_i) \= \beta(v_j) \cdot \mathfrak f_e^{\kappa_e} \in P. 
\end{equation}
 We choose local coordinates~$u$ and~$v$ near the node, say~$v$ vanishes on the upper level component corresponding to~$v_i$. The line bundle~$\omega$ has a local generating section which is given after inverting~$u$ by~$\frac{du}{u}$ and after inverting~$v$ by $-\frac{dv}{v}$. The line bundle $\ca O(-\beta)$ has (perhaps after adjusting the local coordinates $u$ and $v$) a local generating section which is given after inverting~$u$ by $u^{\kappa_e}\beta(v_i)^{-1}$ and after inverting $v$ by $\mathfrak f_e^{\kappa_e} v^{-\kappa_e}\beta(v_j)^{-1}$. As such, the bundle $\omega(-\beta)$ has a local generating section that is given after inverting~$u$ by $u^{\kappa_e}\beta(v_i)^{-1}\frac{du}{u}$ and after inverting $v$ by $-\mathfrak f_e^{\kappa_e}v^{-\kappa_e} \beta(v_j)^{-1}\frac{dv}{v}$. The isomorphism~\ref{eq:iso_open} then corresponds to the section of $\omega(- \beta)$ that is given after inverting~$u$ by $\omega_{(i)}\beta(v_i)^{-1}$ and after inverting~$v$ by $\beta(v_j)^{-1} \omega_{(j)} = \beta(v_i)^{-1}\mathfrak f^{\kappa_e} \omega_{(j)}$. To conclude, we need to show that there exists an invertible function~$\lambda$ near the node such that after inverting~$u$ we have $\omega_{(i)} = \lambda u^{\kappa_e}\beta(v_i)^{-1}\frac{du}{u}$, and after inverting~$v$ we have $\mathfrak f^{\kappa_e} \omega_{(j)} = -\lambda \mathfrak f_e^{\kappa_e}v^{-\kappa_e} \beta(v_i)^{-1} \frac{dv}{v}$. But by condition~\eqref{d:c-2} of~\ref{def:collRD}, we know that there exists an invertible function~$\lambda$ such that 
\begin{equation}
\omega_{(i)} \= \lambda u^{\kappa_e} \frac{du}{u} \quad \text{and} \quad     \omega_{(j)} \= - \lambda v^{-\kappa_e}\frac{dv}{v}, 
\end{equation}
and this $\lambda$ clearly suffices. 

Unravelling the constructions earlier in this section verifies that the constructed point of $\cat{Rub}_{{\calL_\mu}}$ indeed maps to our starting point in $\GSMS$ under~$F$. 
 
 To show that we have constructed an isomorphism of fibered categories, we must finally check that the composites 
 \begin{equation}
 \cat{Rub}_{\calL_\mu}(B)  \lra \GSMS(B)  \lra  \cat{Rub}_{\calL_\mu}(B)
 \end{equation}
 and
 \begin{equation}
 \GSMS(B) \lra  \cat{Rub}_{\calL_\mu}(B)  \lra \GSMS(B)
 \end{equation}
 are isomorphic to the respective identities. This can be done by comparing the actions of the simple LLRT and the simple level rotation torus on the respective spaces; we omit the details. 
\end{proof}

\subsection{The multi-scale space as a Zariski closure}

Fix $g$, $n$, and define $\ca L_\mu$ on the universal curve over $\Mbar_{g,n}$ as before.

\begin{dfx} We define $\cat{Rub}_{\ca L_\mu}^\trop$
to be the fibered category of $\cat{LogSch}_{\Mbar_{g,n}}$ whose objects are pairs $(X/B, \beta)$, where $X/B \in \Mbar_{g,n}$ and $\beta$ is a PL function satisfying condition~\eqref{d:r-1} of~\ref{def:rub}, and such that the line bundle $\ca L_\mu(-\beta)$ has multi-degree 0 on each geometric fiber.
\end{dfx}

This is a slight variant on $\PP(\cat{Rub}_{\ca L_\mu})$. By ignoring the divisibility condition in~\ref{def:rub}, we are effectively taking the coarse moduli space relative to $\Mbar_{g,n}$, and we only require that $\ca L_\mu(-\beta)$ has multi-degree~0, rather than requiring it to be trivial. Since we in particular do not record the data of an isomorphism, we are effectively also taking a $\bb C^*$-quotient.

The map $\cat{Rub}_{\ca L_\mu}^\trop \to \Mbar_{g,n}$ is birational and representable, but not in general proper. Using stability conditions as in \cite{HMPPS}, we can construct a compactification 
\bes
\cat{Rub}_{\ca L_\mu}^\trop \lra  \PP(\cat{Rub}_{\ca L_\mu}^\theta) \lra \Mbar_{g,n},
\ees
where $\PP(\cat{Rub}_{\ca L_\mu}^\theta) \to \Mbar_{g,n}$ is proper, birational, and representable, and $\cat{Rub}_{\ca L_\mu}^\trop \to  \PP(\cat{Rub}_{\ca L_\mu}^\theta)$ is an open immersion; but we do not pursue this here as it would require substantial additional notation. 

Let $\PP(\ca{MS}^0) \sub \ca M_{g,n}$ be the locus of smooth curves over which $\ca L_\mu$ admits a non-zero global section; this can be seen as the interior of the locus of (projectivized, generalized) multi-scale differentials. 

\begin{theorem}
The Zariski closure of $\ca{MS}^0$ in $\cat{Rub}_{\ca L_\mu}^\trop$ $($or, equivalently, in $\PP(\cat{Rub}_{\ca L_\mu}^\theta))$ is equal to $\PP(\ca{MS}_\mu)$, the projectivized space of\, $($non-generalized$)$ multi-scale differentials. 
\end{theorem}
\begin{proof}
  There is a natural closed immersion $\PP(\cat{Rub}^\coarse_{\ca L_\mu}) \to \cat{Rub}_{\ca L_\mu}^\trop$, and the main component of the space
  $\PP(\cat{Rub}^\coarse_{\ca L_\mu})$ is $\PP(\ca{MS}_\mu)$. 
\end{proof}

One can obtain the stacky version $\Xi \Mbar_{g,n}(\mu)$ (of which $\ca{MS}_\mu$ is the relative coarse moduli space) in a similar fashion, replacing $ \PP(\cat{Rub}_{\ca L_\mu}^\theta)$ with a stacky modification; we leave the details to the interested reader.

\section{The Hodge DR conjecture} \label{sect:universal_bundle}

In this section we present several equivalent constructions of the universal line bundle introduced in~\ref{sec:intro:hodge_DR}, discuss its various  properties, and prove~\ref{thm:HDR_k_is_0}. 

As explained in~\ref{sec:intro:hodge_DR}, the projectivized space of (generalized) multi-scale differentials comes with a map to the projectivized Hodge bundle, obtained by taking the differential at the top level, and allowing it to vanish at all lower levels. Pulling back $\ca O(1)$ from the Hodge bundle gives a line bundle on the generalized multi-scale space. We begin by giving several equivalent versions of this construction. 

First we write out explicitly the objects of the fibered category $\bb P(\cat{Rub})$: 
\begin{equation*}
\bb P(\cat{Rub}) \= \{(\pi\colon X\lra B, \beta, \ca F)\},
\end{equation*}
where $(X/B, \beta)$ is a point of $\cat{Rub}$ as in~\ref{def:rub}, and $\ca F$ is a line bundle on $B$. The Abel--Jacobi map sends such an object to $\pi^*\ca F (\beta)$, giving a proper, see~\cite[Theorem 4.3.2]{MarcusWiseLog}, map $\bb P(\cat{Rub}) \to \Picabs$.

Now fix a line bundle $\ca L$ on $X_{g,n}/\Mbar_{g,n}$, which is of total degree 0 on each fiber. Then we can write explicitly the fibered category of $\bb P(\cat{Rub}_{\ca L})$ as
\begin{equation*}
\bb P(\cat{Rub}_{\ca L}) \= \{(X/B, \beta, \ca F, \phi)\}, 
\end{equation*}
where $(X/B, \beta, \ca F)$ is an object of $\bb P(\cat{Rub})$ with $X/B$ stable of genus $g$, and $\phi\colon \pi^*\ca F (\beta) \to \ca L$ is an isomorphism. 

\subsection*{Construction 1: Tautological bundle}
This is just the bundle $\ca F$ on $\bb P(\cat{Rub})$, or its pullback to $\ca F$ on $\bb P(\cat{Rub}_{\ca L})$ along the tautological map. We denote the {\em dual} of this line bundle by $\eta$. 

\subsection*{Construction 2: Projective embedding}

Let $D$ be an effective divisor on $X_{g,n}$ such that $R^1\pi_*\calL(D)$ vanishes, so in particular $\pi_* \calL(D)$ is a vector bundle on $\Mbar_{g,n}$. Such a $D$ can always be found as an element of the linear system of a sufficiently relatively ample sheaf on $X_{g,n}$ over $\Mbar_{g,n}$. Then over $\bb P(\cat{Rub}_{\ca L})$ we have natural maps
\begin{equation}
\pi^*\ca F \longisom \ca L(-\beta) \lra \ca L \lra \ca L(D), 
\end{equation}
where the first map is induced by $\phi$, the second is induced by the natural map $\ca O(-\beta) \to \ca O$, and the third by the natural map $\ca O \to \ca O(D)$. Adjunction yields a map 
\begin{equation}
\ca F \= \pi_*\pi^*\ca F \lra \pi_*\ca L(D), 
\end{equation}
which is by definition a map 
\begin{equation}\label{lem:map-hodge}
F\colon \bb P(\cat{Rub}_{\ca L}) \lra \bb P_{\bb P(\cat{Rub}_{\ca L})}(\pi_*\ca L(D)).
\end{equation}
Indeed, note that our projectivizations are moduli of subbundles, not quotient bundles, so it is enough to check that this map is universally injective. But the formation of both sides commutes with arbitrary base-change, so it is enough to check this over a point, where it is clear. 

\begin{lemma}\label{lem_eta_comparison_1}
$F^*\ca O(1) = \eta$.
\end{lemma}
\begin{proof}
The equality $F^*\ca O(1) = \ca F^\vee$ is immediate from \cite[\href{https://stacks.math.columbia.edu/tag/0FCY}{Example 0FCY}]{stacks-project};  we obtain $\ca F^\vee$ instead of $\ca F$ because we define the projectivization to be the moduli of rank 1 subbundles, not rank 1 quotient bundles. 
\end{proof}

In particular, we observe that the line bundle $F^*\ca O(1)$ turns out to be independent of the choice of the sufficiently relatively ample divisor $D$. In the case considered in the introduction, we take 
\begin{equation}
\calL = \omega^{\otimes k}_{X_{g,n}/\Mbar_{g,n}}\left(-\sum_{i=1}^n (a_i - k) z_i\right)
\end{equation}
and $D=\sum_{i: a_i  > 0} a_i z_i$.

\subsection*{Construction 3: Pullback from rubber target}
For this construction we restrict to the case where $\ca L = \ca O_X(\sum_i a_i z_i)$ for $k = 0$; put another way, we choose a rational section of $\ca L$ whose locus of zeros and poles is contained in a union of disjoint sections of $X \to B$. 

We write 
\begin{equation*}
    D \= \sum_{i : a_i >0} a_i z_i \quad \text{and} \quad E \= -\sum_{i : a_i <0} a_i z_i. 
\end{equation*}
Since these are effective divisors, we have natural maps
\begin{equation*}
    \ca O_X \lra \ca O_X(D) \quad\text{and}\quad \ca O_X \lra \ca O_X(E), 
\end{equation*}
and combining with the natural map $\ca O_X \to \ca O_X(\beta)$ and the isomorphism $\phi\colon \pi^*\ca F (\beta) \isom \ca O_X(D - E)$ yields maps 
\begin{equation*}
    \ca O_X(-E)(-\beta)  \lra \ca O_X \quad \text{and} \quad \ca O_X(-E)(-\beta) \lra \ca O_X(D - E)(-\beta) \longisom \pi^* \ca F. 
\end{equation*}
The induced map 
\begin{equation*}
    \ca O_X(-E)(-\beta) \lra \ca O_X \oplus \pi^* \ca F\, 
\end{equation*}
is universally injective since the first map is injective around the support of $D$ and the second is injective away from the support of $D$. This induces a map 
\begin{equation*}
    X \lra \bb P(\ca O_B \oplus \ca F). 
\end{equation*}
The cotangent line at $\infty$ to this rubber target is then given by 
\begin{equation}
    \Psi_\infty = \ca F^\vee. 
\end{equation}
We have deduced the following result. 

\begin{lemma}\label{lem_eta_comparison_2}
$\Psi_\infty = \eta$.
\end{lemma}

\begin{remark}
Above we have constructed a rubber target of length 1 (\textit{i.e.}~with no expansions).  This is because we are only interested in what happens near the infinity section, so we do not need to construct the whole expanded chain. The reader who is more comfortable with expansions may verify that the length 1 target we construct here is exactly what is obtained by following through the proof of the expanded target in \cite[Proposition~50]{BHPSS}, and then contracting all except the top component.
\end{remark}

\subsection{Computation of \texorpdfstring{$\boldsymbol{\eta}$ for $\boldsymbol{k=0}$}{$\eta$ for $k=0$}}

Here we prove~\ref{thm:HDR_k_is_0}, which we restate for the convenience of the reader.

\begin{theorem}
\ref{conj:HDR} is true for $k=0$: for any $g,u \geq 0$ and any vector $A \in \mathbb{Z}^n$ with sum $|A|=0$, we have
\[
p_*\left(\left[\PP\left(\cat{Rub}_{\calL_A}\right) \right]^\mathrm{vir} \cdot \eta^u \right) 
\= p_*\left(\left[\Mbar_{g,A}\left(\mathbb{P}^1, 0, \infty\right)^\sim\right]^\mathrm{vir} \cdot \Psi_\infty^u \right)  
\= [r^u]{\rm Ch}_{g,A}^{0, r, g+u}.
\]
\end{theorem}

\begin{proof}
The first equality follows from~\Cref{lem_eta_comparison_1,lem_eta_comparison_2}. For the second equality, we note that the term on the left has been computed in \cite[Corollary 4.3]{FWY} in terms of a slightly modified Chiodo class. Indeed, we define an $r$-shifted version $A(r)$ of $A$ by
\[
A(r)_i \= \begin{cases}
a_i & \text{for }a_i \geq 0,\\
r+a_i & \text{for }a_i < 0.
\end{cases}
\]
In other words, for all indices $i$ with $a_i<0$ (which form a subset $I_\infty \subseteq \{1, \ldots, n\}$), we shift the vector $A$ by $r$ in the $\supth{i}$ entry.
Then the Chiodo class $\mathrm{Ch}_{g,A(r)}^{0,r,d}$ is a polynomial in $r$, for $r$ sufficiently large. Denote by
\[
\mathrm{Ch}_{g,A(r)}^{0,r,\bullet} \= \sum_{d \geq 0} \mathrm{Ch}_{g,A(r)}^{0,r,d}
\]
the associated mixed-degree class. Then in this notation, the formula from  \cite[Corollary 4.4]{FWY} reads as follows:
\begin{align*}
p_*\left(\left[\Mbar_{g,A}\left(\mathbb{P}^1, 0, \infty\right)^\sim\right]^\mathrm{vir} \cdot \Psi_\infty^u \right)  &\= \sum_{\vec e \in \mathbb{Z}_{\geq 0}^{I_\infty}} \prod_{i \in I_\infty} (-a_i \psi_i)^{e_i} \cdot \left[r^{u-|\vec e|}\right] \mathrm{Ch}_{g,A(r)}^{0,r,u+g-|\vec e|}\\
&\= [r^u] \left[\sum_{\vec e \in \mathbb{Z}_{\geq 0}^{I_\infty}} \prod_{i \in I_\infty} \left(- a_i  r  \psi_i\right)^{e_i} \cdot  \mathrm{Ch}_{g,A(r)}^{0,r,\bullet}\right]_{\mathrm{codim}\ g+u}\\
&\= [r^u] \left[\prod_{i \in I_\infty}  \frac{1}{1+ a_i r \psi_i} \cdot  \mathrm{Ch}_{g,A(r)}^{0,r,\bullet}\right]_{\mathrm{codim}\ g+u}\\
&\= [r^u] \left[\mathrm{Ch}_{g,A}^{0,r,\bullet}\right]_{\mathrm{codim}\ g+u}. 
\end{align*}
Here the last step uses \cite[Theorem 4.1(ii)]{DaniloEulerChar}.
\end{proof}

\subsection{(A)symmetry}

Above we gave three constructions of the line bundle $\eta = \eta(\ca L)$ on $\bb P(\cat{Rub}_\ca L)$. We know that the push-forwards to $\Mbar_{g,n}$ of $[\bb P(\cat{Rub}_{\ca L})]^\mathrm{vir}$ and $[\bb P(\cat{Rub}_{\ca L^\vee})]^\mathrm{vir}$ agree. However, once we intersect with the class $\eta$,  things are a little more subtle. The universal curve over $\PP(\cat{Rub})$ carries a PL function $\beta$, totally ordered and with maximum value $0$. The \emph{minimum} value of $\beta$ we denote by $\beta^{\min}$; this is a PL function on $\PP(\cat{Rub})$. We set $\eta(\beta^{\min}) = \eta \otimes_{\ca O} \ca O(\beta^{\min})$.

\begin{lemma}
We have 
\begin{equation}
p_*\left(\left[\bb P\left(\cat{Rub}_{\ca L^\vee}\right)\right]^\mathrm{vir} \cdot c_1(\eta)^u\right)
\= p_*\left(\left[\bb P\left(\cat{Rub}_{\ca L}\right)\right]^\mathrm{vir} \cdot \left(-c_1\left(\eta\left(\beta^{\min}\right)\right)\right)^{u}\right). 
\end{equation}
\end{lemma}

\begin{proof}
There is a natural isomorphism (compatible with the virtual fundamental classes) over $\Mbar_{g,n}$ from $\bb P(\cat{Rub}_{\ca L})$ to $\bb P(\cat{Rub}_{\ca L^\vee})$, given by
\begin{equation}
(X/B, \beta, \ca F, \phi) \longmapsto  \left(X/B, \beta^{\min} - \beta, \left(\ca F\left(\beta^{\min}\right)\right)^\vee, \phi'\right), 
\end{equation}
where $\phi'$ is the composite
\begin{equation}
\pi^*\left(\ca F\left(\beta^{\min}\right)\right)^\vee\left(\beta^{\min} - \beta\right) \= \pi^* \ca F^\vee (-\beta) \xrightarrow{(\phi^\vee)^{-1}} \ca L^\vee. \hfill\qedhere
\end{equation}
\renewcommand{\qed}{}    
\end{proof}

\section{Blowup descriptions}\label{sec:blowup_descriptions}

In this section we give a description of $\PP(\cat{Rub}_{\ca L}^{\sf{coarse}})$ as a global blowup. 

First, in genus zero, we construct an explicit sheaf of ideals on $\overline{\calM}_{0,n}$, such that blowing up $\overline{\calM}_{0,n}$ along this sheaf gives $\PP(\cat{Rub}_{\ca L}^{\sf{coarse}})$. 
In \cite{nguyen} Nguyen described the incidence variety compactification (IVC) in the case of genus zero as an explicit blowup of $\Mbar_{0,n}$.
Note that in genus zero there are no global residue conditions (because any top-level vertex must have a marked pole), and hence in genus zero the rubber space and the space of generalized multi-scale differentials coincide with the space of multi-scale differentials. Our blowup description can thus recover Nguyen's result about the IVC of the strata of meromorphic $1$-forms in genus zero as a blowup of $\overline{\calM}_{0,n}$. We also provide an example demonstrating the difference between the rubber space and the IVC in genus zero.
 
Next, for arbitrary genus, we construct a globally defined sheaf of ideals on the normalization of the incidence variety compactification (NIVC) whose blowup gives the (projectivized) multi-scale moduli space (\textit{i.e.}~the main component of $\cat{Rub}_{\ca L}^{\sf{coarse}}$). Consequently, it follows that the (coarse) space of projectivized (non-generalized) multi-scale differentials is a projective variety for all $g$. Recall that in \cite[Section~14.1]{LMS} the moduli space of multi-scale differentials was described as a local blowup, where the ideals locally defining the center of the blowup can differ by principal ideals on the overlaps of local charts. In particular, the description of \cite{LMS} did not yield projectivity of the space of multi-scale differentials. By constructing an explicit ample divisor class, the projectivity of the moduli space of multi-scale differentials was later established in \cite[Section~3]{ccm}. Our global blowup description thus provides a direct conceptual understanding of this projectivity result. 
 
Besides projectivity, knowing a blowup description of compactified strata of differentials can be helpful for obtaining geometric invariants, such as volumes of the strata, by using intersection theory; see \cite{nguyen}. We also provide a tropical interpretation of our blowup, which sheds further light on the geometry of the construction.

\subsection{The sheaf of ideals in genus zero} 

Let $\Gamma$ be the dual graph of a boundary stratum $D_\Gamma\subset\overline{\mathcal M}_{0,n}$. For each vertex $v\in V(\Gamma)$, let $d(v)$ be the degree of $\calL_\mu$ restricted to $v$ (so $\sum_{v\in V(\Gamma)} d(v) = 0$ by definition). Since $\Gamma$ is a tree, there exists a unique `slope'\footnote{The justification for this terminology is given by~\ref{eq:slopeadmiss}, which shows that the slopes of points $\cat{Rub}_{\ca L_\mu}$ satisfy the same conditions.} function $\kappa \colon H\to \mathbb Z$ from the set $H=H(\Gamma)$ of half-edges of $\Gamma$ such that 
\begin{enumerate}
\item $\kappa$ agrees with $m_i$ at the leg corresponding to a marked point $z_i$;
\item $\kappa(h) + \kappa(h') = 0$ for any $h$ and $h'$ that are opposite halves of an edge;
\item for all vertices $v$, we have $\sum_{h \in H(v)} \kappa(h) = d(v)$, where we sum over all half-edges attached to $v$. 
\end{enumerate}

For every pair of vertices $v$ and $v'$, let $\gamma$ be the unique path from~$v$ to~$v'$ in the tree~$\Gamma$. We view this (directed) path as a sequence of half-edges, where if an edge $e = (h, h')\in E(\Gamma)$ appears in $\gamma$ in the direction going \emph{from} $h$ \emph{to} $h'$, meaning that along the path $\gamma$ in the direction from $v$ to $v'$ the half-edge $h$ appears first, followed by $h'$, then we put (only) $h$ in our sequence of half-edges. We define an ideal locally around the boundary stratum $D_\Gamma\subseteq \overline{\calM}_{0,n}$ by 
\begin{equation*}
I(v, v') \coloneqq \prod_{h \in \gamma} \delta(h)^{\max(\kappa(h), 0)},
\end{equation*}
where we write $\delta(h)$ for the ideal associated to the edge containing $h$ (that is, for the defining equation of the boundary divisor of $\overline{\calM}_{0,n}$ where the corresponding node exists). Define
\begin{equation*}
J(v, v') \coloneqq I(v, v') + I(v', v);
\end{equation*}
this evidently satisfies $J(v, v') = J(v', v)$ and $J(v,v) = (1)$. Finally, we set 
\begin{equation*}
w(v) \coloneqq \on{valence}(v) - 2,
\end{equation*}
which is a positive integer by the stability of the curve, and define
\begin{equation*}
J(\Gamma) \coloneqq\prod_{(v, v') \in V \times V} J(v, v')^{w(v) w(v')}. 
\end{equation*}
A concrete example of this ideal is given in~\ref{ex:blowup-comparison} below.\footnote{If the reader prefers not splitting up the half-edges into ones with increasing and decreasing slopes, alternatively we can define the {\em fractional} ideal $J'(v, v') = \prod_{h\in \gamma} (\delta(h)^{\kappa (h)}, 1)$, and define $J'(\Gamma)\coloneqq \prod_{(v, v') \in V \times V} J'(v, v')^{w(v) w(v')}$. Then $J'$ induces a globally defined fractional ideal whose blowup is the same as that of $J$ (as we will justify for $J$ in the following subsections.)}

\begin{remark}
Note that it is possible to define {\em locally} an ideal that is simpler than $J$ and whose blowup produces the same space. Indeed, taking the product of $J$ and any  principal ideal works. Nevertheless, such local ideals may not always glue to form a {\em global} sheaf of ideals. Blowing up a globally defined sheaf of ideals (from a projective variety) can directly imply the projectivity of the resulting space, while gluing local blowups together does not. Therefore, the ideal $J$ in the above was designed with some delicate exponents to make it a globally defined sheaf of ideals, as we will check in the next section. This idea will be used in~\ref{subsec:blowup-g} to define a global ideal sheaf on NIVC to conclude the projectivity of the multi-scale space for arbitrary genus.
\end{remark}

\subsection{Compatibility under degeneration in genus zero}

To show that the ideals $J(\Gamma)$ defined in the neighbourhood of each stratum $D_\Gamma\subseteq \overline{\calM}_{0,n}$ glue to a global ideal sheaf over $\overline{\calM}_{0,n}$, we need to show that they behave well under degeneration. As any dual graph~$\Gamma$ can be obtained from any other~$\Gamma'$ by a series of operations of inserting and contracting edges, it is enough to check that the ideals glue under contracting a single edge of the graph.

\begin{lemma}\label{lm:Jcompatible}
Let $e$ be an edge of $\Gamma$, and let $\Gamma'$ be the graph obtained from $\Gamma$ by contracting~$e$. Then $J(\Gamma') = J(\Gamma)$, after inverting the ideal $\delta(e)$. 
\end{lemma}
Note that inverting $\delta({e})$ geometrically corresponds to restricting to the locus where the edge $e$ is contracted, \textit{i.e.}~where the corresponding node of the curve is smoothed out.
\begin{proof}
We denote by $c\colon \Gamma \to \Gamma'$ the contraction map,  let $v_1$ and $v_2$ be the endpoints of $e$, and let $v'$ be the vertex of $\Gamma'$ to which $e$ is contracted, so that $d(v') = d(v_1) + d(v_2)$. 

If $v$ is any vertex of $\Gamma$ different from $v_1$ and $v_2$, then clearly $w(v) = w(c(v))$. Furthermore, the slope function on $\Gamma$ clearly restricts to the slope function on $\Gamma'$. Thus for any two vertices~$u_1$ and~$u_2$ of~$\Gamma$ distinct from $v_1$ and $v_2$, we have
\bes
J_\Gamma(u_1, u_2) \sim J_{\Gamma'}(u_1, u_2), 
\ees
where to simplify notation we write $I \sim J$ if the ideal sheaves $I$ and $J$ become equal after inverting $\delta(e)$. Similarly, $J_\Gamma(v_1, v_2) \sim (1)$. 

Based on the above analysis, we only need to consider the pairs of vertices in $\Gamma'$ and in $\Gamma$ that involve $v'$ and $v_1$ or $v_2$, respectively. 
It therefore suffices to show that 
\begin{equation}
\label{eq:sim_1}
\prod_{v \in V(\Gamma')} J(v', v)^{2w(v') w(v)} \sim \prod_{v \in V(\Gamma)} J(v_1, v)^{2 w(v) w(v_1)} J(v_2, v)^{2 w(v) w(v_2)}. 
\end{equation}
Let $V^\circ\coloneqq V(\Gamma)\setminus\{v_1,v_2\}=V(\Gamma')\setminus\{v'\}$. Then~\ref{eq:sim_1} reduces to showing
\begin{equation*}
\label{eq:sim_2}
\prod_{v \in V^\circ} J(v', v)^{w(v') w(v)} \sim \prod_{v \in V^\circ} J(v_1, v)^{w(v_1) w(v)} J(v_2, v)^{w(v_2) w(v)}. 
\end{equation*}
This follows from $w(v') = w(v_1) + w(v_2)$ and the relations 
\begin{equation*}
  J(v', v) \sim J(v_1, v) \sim J(v_2, v)
\end{equation*}
for all $v \in V'$.
\end{proof}

\begin{df}
\label{def:big_ideal_sheaf}
Define $J(\ca L_\mu)$ to be the (global) ideal sheaf on $\Mbar_{0,n}$ that for any boundary stratum $D_\Gamma$ restricts to the ideal~$J(\Gamma)$ on a neighbourhood of $D_\Gamma$.
\end{df}
The existence of $J(\ca L)$ follows from~\Cref{lm:Jcompatible}.

\subsection{A tropical picture in genus zero}
\label{subsec:tropical}

The normalized blowup in the ideal $J(\Gamma)$ corresponds tropically to a subdivision of the positive orthant in the vector space $\bb Q\Span{E}$, where $E=E(\Gamma)$ is the edge set. This subdivision is built by taking a hyperplane (or sometimes the whole space) for every pair of vertices in $\Gamma$: if $\gamma$ is the path from $v$ to $v'$ as above, then the corresponding hyperplane $L(v, v')$ is cut out by the equation
\begin{equation*}
\sum_{h \in \gamma} \kappa(h) e(h) \= 0 ,
\end{equation*}
where $e(h)$ is the edge containing the half-edge $h$, viewed as an element of the group $\bb N\Span{E}$ (and we recall that a half-edge $h$ is said to be contained in a directed path $\gamma$ if $\gamma$ goes via~$h$ before going through the complementary half-edge of the same edge).

These local subdivisions glue to a global subdivision of the tropicalization of $\Mbar_{0, n}$, inducing a proper birational map $\widetilde{\ca M}_{0, n}\to \Mbar_{0,n}$. 

\begin{lemma}\label{lem:blowup_to_subdivision}
The normalization of the blowup of\, $\overline{\calM}_{0,n}$ in the ideal $J(\ca L_\mu)$ is equal to the proper birational map $\widetilde{\ca M}_{0, n}\to \Mbar_{0,n}$ induced by the subdivision above. 
\end{lemma}
\begin{proof}
  The standard dictionary (see \cite[Section~1.3.3, p.~14]{Kato1989Logarithmic-deg})
  between toric blowups and subdivisions implies that the normalized blowup in $J(v, v')$ is equal to that induced by the subdivision in $L(v, v')$. Since $w(v) \ge 1$ (by stability), blowing up in $J(v, v')$ is the same as blowing up in $J(v, v')^{w(v) w(v')}$.  Normalized blowup in a product of ideals corresponds to superimposing their subdivisions. 
\end{proof}

\subsection{Comparing blowups and rubber maps in genus zero}

We are ready to prove our main statement in genus zero.  

\begin{theorem}
\label{thm:blowup-0}
The normalization of the blowup $\widetilde{\ca M}_{0, n}$ of\, $\Mbar_{0,n}$ along the ideal sheaf\, $J(\calL_\mu)$ is the projectivized coarse moduli space of rubber differentials $\PP(\cat{Rub}_{\calL}^{\sf{coarse}})$. 
\end{theorem}
\begin{proof}
Let $X/B$ be a nuclear log curve of genus zero. 

\textit{Claim:} There exists a PL function $\beta$ on $X$ such that $\calL_\mu \cong \on O(\beta)$, and moreover such a $\beta$ is unique up to translation by an element of $\ghost_{B}(B)^\gp$. 

To prove the claim, we use the fact that the graph is a tree to deduce that there is a unique collection of admissible slopes $\kappa_e$. We pick a vertex $v_0$ and let $\beta$ be the unique PL function vanishing on $v_0$ and with slopes given by the $\kappa_e$. The line bundle $\calL_\mu(-\beta)$ has multi-degree zero, and is hence trivial since $X$ has genus zero. This proves the claim. 

Now recall that $\cat{Rub}^\coarse_\ca L$ can be obtained by omitting the divisibility condition~\eqref{d:r-2} from~\ref{def:rub}. In other words, the point $X/B$ lies in $\cat{Rub}^\coarse_\ca L$ if and only if the values of~$\beta$ on the vertices of $\Gamma$ form a totally ordered set. It therefore remains to check that this is equivalent to the map $B \to \Mbar_{0,n}$ factoring via the subdivision described in~\ref{subsec:tropical}. 

If $\gamma$ is a directed path in $\Gamma$, we define 
\begin{equation*}
\phi(\gamma) \coloneqq  \sum_{h \in \gamma} \kappa_h \delta_h. 
\end{equation*}
Since the difference of values of $\beta$ at the two ends of an edge is the slope $\kappa_e$ of that edge (with the appropriate sign), the values of $\beta$ at the two ends of a path~$\gamma$ differ by~$\phi(\gamma)$.

Fix a vertex $v_0$, and write $\gamma_{v}$ for the unique path from $v_0$ to $v$. Then the set $\{\beta(v) : v \in V(\Gamma)\}$ is totally ordered if and only if the set 
\begin{equation*}
\{\phi(\gamma_v) : v \in V(\Gamma)\}
\end{equation*}
is totally ordered. This is in turn equivalent to requiring that for every path $\gamma\subset\Gamma$ (not necessarily a path from $v_0$), the element $\phi(\gamma)$ is comparable to $0$, \textit{i.e.}~either $\phi(\gamma) \in \ghost_S$ or $-\phi(\gamma) \in \ghost_S$. Imposing this condition is equivalent to subdividing $\bb N\langle E\rangle$ in the hyperplane $L(v, v')$ of~\ref{subsec:tropical}, where~$v$ and $v'$ are the endpoints of $\gamma$.
\end{proof}

\subsection{Comparison to Nguyen's blowup in genus zero}

As mentioned, in genus zero Nguyen \cite{nguyen} described the IVC as an explicit blowup of $\Mbar_{0,n}$ (also for the more general case of $k$-differentials in genus zero). Since the rubber/multi-scale space is the normalization of a blowup of the normalization of the IVC,   
our blowup described in~\ref{thm:blowup-0} must dominate the blowup defined by Nguyen. In this subsection we recall Nguyen's construction, provide a viewpoint of his blowup from our setup, and give an alternative proof for Nguyen's result that blowing up $\Mbar_{0,n}$ in his ideal gives the~IVC.

We begin by recalling Nguyen's construction of a sheaf of ideals on $\Mbar_{0,n}$.
Let $X/B$ be a nuclear log curve of genus zero with graph $\Gamma$, and let $\kappa$ be the slope function on the edges of $\Gamma$, \textit{i.e.}~the PL function constructed in the proof of~\ref{thm:blowup-0}. 
For a given vertex $v\in V(\Gamma)$ and an edge $e \in E(\Gamma)$, let $h_v(e)$ be the half-edge of $e$ such that the path from the end of $h_v(e)$ to $v$ passes through $e$. For a vertex $v\in V(\Gamma)$, we define 
\begin{equation}\label{eq:delta_v}
    \delta_v \coloneqq\prod_{e\in E(\Gamma)} \delta_{e}^{\kappa_{v,e}},
\end{equation}
where $\kappa_{v,e} \coloneqq \max (\kappa(h_v(e)),0)$. Let $N(\Gamma)$ be the (local) ideal (in the variables $\delta_e$, as in our setup) generated by the set of elements $\delta_v$ for all vertices $v\in V(\Gamma)$. It was shown in \cite{nguyen} that these $N(\Gamma)$ can be patched together to a sheaf of ideals $N$ globally defined on $\overline{\ca M}_{0,n}$. This can be seen the same way as~\ref{lm:Jcompatible}, and we will discuss this in more generality in~\ref{rem:blowup-ivc} for arbitrary genera. 

Before proceeding, we illustrate Nguyen's ideal and our ideal in the following example.
\begin{example}\label{ex:blowup-comparison}
Consider a (partially ordered) dual graph $\Gamma$ as illustrated in~\ref{fig:blowupgraph}, with all slopes $\kappa=1$.
\begin{figure}[tb]
\[
\begin{tikzpicture}[baseline=0pt, vertex/.style={circle,draw,font=\huge,scale=0.5, thick}]
\node[vertex] (v0) at (0,0) {$v_0$};
\node[vertex] (v1) at (-2,2) {$v_1$};
\node[vertex] (v2) at (0,2) {$v_2$};
\node[vertex] (v3) at (2,2) {$v_3$};

\draw[thick] (v0) to node[near start, below right]{$\kappa_h=-1$} node[near end, below right]{$\kappa_{h'}=1$} node[midway, above left]{$e_3$}(v3);
\draw[thick] (v0) to node[near end, below left]{$e_2$}(v2);
\draw[thick] (v0) to node[near end, below left]{$e_1$} (v1);

\draw (v0) -- ++(270:0.6);
\draw (v1) -- ++(60:0.6);
\draw (v1) -- ++(120:0.6);
\draw (v2) -- ++(60:0.6);
\draw (v2) -- ++(120:0.6);
\draw (v3) -- ++(60:0.6);
\draw (v3) -- ++(120:0.6);
\end{tikzpicture}
\]
    \caption{The graph $\Gamma$ of a stratum in $\Mbar_{0,7}$; the desired slopes $\kappa$ can be obtained, \textit{e.g.}, by using the signature $\mu=(-1^6, 4)$ with the six markings associated to simple poles attached to the vertices $v_1$, $v_2$, and  $v_3$.}
    \label{fig:blowupgraph}
\end{figure}
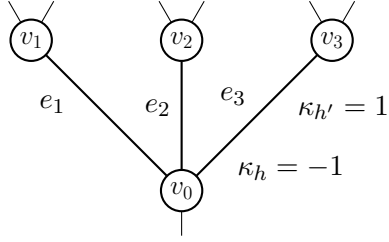
Recalling the definition of $\delta_v$ in~\ref{eq:delta_v} for a vertex $v\in V(\Gamma)$ and writing $\delta_i \coloneqq \delta_{e_i}$ to lighten notation, we obtain $\delta_{v_0} = \delta_1\delta_2\delta_3$, $\delta_{v_1} = \delta_2\delta_3$, $\delta_{v_2} = \delta_1\delta_3$, and $\delta_{v_3} = \delta_1\delta_2$. In this case Nguyen's ideal $N(\Gamma)$ is given by 
$$ N(\Gamma) \= (\delta_1\delta_2, \delta_1\delta_3, \delta_2\delta_3, \delta_1\delta_2\delta_3) \= (\delta_1\delta_2, \delta_1\delta_3, \delta_2\delta_3). $$
In contrast, our ideal $J(\Gamma)$ is given by 
$$ J(\Gamma) \= (\delta_1, \delta_2)^{2} (\delta_1, \delta_3)^{2}(\delta_2, \delta_3)^{2} (\delta_1)^4 (\delta_2)^4 (\delta_3)^4. $$
When we blow up $J(\Gamma)$, each ideal generated by a pair $(\delta_i, \delta_j)$ for $1\le i<j\le 3$ becomes principal, and so does the ideal $N(\Gamma)$. Therefore, the blowup in $J(\Gamma)$ dominates the blowup in $N(\Gamma)$.
\end{example}
Nguyen \cite{nguyen} proved that blowing up $\Mbar_{0,n}$ along the globally defined sheaf of ideals~$N$ gives the IVC. Indeed, in the example above we see explicitly that locally around the boundary stratum with the dual graph~$\Gamma$, the rubber/multi-scale space obtained by blowing up along~$J$ is a further blowup of the IVC.

The situation of this example can also be understood in general, from our viewpoint, which  gives an alternative proof of the result of Nguyen.

\begin{proposition}
\label{prop:blowup-0}
The local blowup of\, $\Mbar_{0,n}$ near $D_\Gamma$ along the ideal $J(\Gamma)$ makes the ideal $N(\Gamma)$ become principal. 

Moreover, in genus zero the blowup of\, $\Mbar_{0,n}$ along the ideal sheaf\, $N$ is the IVC. 
\end{proposition}

Before giving the proof, we first reinterpret $N(\Gamma)$ geometrically as follows. If two vertices $v$ and $v'$ are joined by an edge~$e$, 
and if $\ell(v)>\ell(v')$, then~$\delta_v$ divides~$\delta_{v'}$. Therefore, the ideal $N(\Gamma)$ is the same as the ideal generated only by the elements $\delta_v$ where $v$ ranges over all vertices that are local maxima of $\Gamma$ (in the sense that all edges from $v$ go down~-- recall that this is a partial order on the graph, and the datum of a multi-scale differential upgrades this to a full order). A vertex~$v$ that is a local maximum of~$\Gamma$, such that the corresponding $\delta_v$ generates the ideal $N(\Gamma)$ after the blowup, becomes a global top-level vertex. On the other hand, those local maxima~$v$ whose $\delta_v$ terms do {\em not} generate the principal ideal after blowing up $N(\Gamma)$ may not divide each other, and thus remain unordered. This corresponds to the fact that a point in the IVC records actual differentials merely on top-level vertices where the stable differential is not identically zero, while on any lower vertex the stable differential is identically zero (though the underlying marked zeros and poles of the twisted differential are still remembered).  

\begin{proof}
For the first claim, note that the edge parameter $\delta_e$ appears with the same exponent in the expressions of $\delta_v$ and $\delta_{v'}$ \emph{unless}~$e$ lies in the unique path from $v$ to $v'$, in which case the exponents of $\delta_e$ in $\delta_v$ and $\delta_{v'}$ are the same as those in $I(v,v')$ and $I(v',v)$, respectively. Since blowing up along $J(\Gamma)$ makes the ideal $(I(v,v'), I(v',v))$ principal, it follows that each ideal $(\delta_v, \delta_{v'})$ becomes principal under that blowup. This is to say that after blowing up in~$J(\Gamma)$, one of $\delta_v$ and $\delta_{v'}$ must divide the other. Doing this for all~$v$ and $v'$ shows that after the blowup along $J(\Gamma)$, there is a set of elements $\{\delta_{v_1},\dots,\delta_{v_k}\}$ such that for every $i$ and each $v\in V(\Gamma)$, $\delta_{v_i}$ divides $\delta_v$. 
In particular, such $\delta_{v_i}$ and $\delta_{v_j}$ divide each other and thus differ by multiplication by a unit, and the ideal $N(\Gamma)$, after the blowup along $J(\Gamma)$, is generated by any one of these $\delta_{v_i}$, and hence it becomes principal. 

For the second claim, we will construct the desired morphisms between the blowup and the IVC in both directions that are inverses of each other. Since the blowup and the IVC both admit natural maps onto $\Mbar_{0,n}$, these morphisms will be constructed locally over each boundary stratum $D_\Gamma$ of $\Mbar_{0,n}$. 

The upshot underneath the constructions is that $\delta_v$ for $v\in V(\Gamma)$ is an {\em adjusting parameter} in the sense of \cite[Proposition 11.13]{LMS}, which means that multiplying by $\delta_v^{-1}$ makes the limiting differential become not identically zero on the component corresponding to $v$. To see this, let $D_{e_i}$ be the boundary divisor of $\overline{\calM}_{0,n}$ corresponding to a given edge $e_i$ of $\Gamma$. Contracting all edges of $\Gamma$ except $e_i$ produces a graph with two vertices connected by the  edge $e_i$, and the family of stable differentials over it vanishes on the irreducible component corresponding to the lower-level vertex, with generic vanishing order $|\kappa_{e_i}|$. Given a vertex $v$ of $\Gamma$, if the image of $v$ under this contraction is the lower of the two resulting vertices, then over $D_\Gamma$ the family of stable differentials vanishes identically on the irreducible component corresponding to~$v$. Applying this observation to all edges $e_i$ in $\Gamma$ where $v$ becomes lower after the above edge contractions, it follows that the expression of $\delta_v$ records exactly the total vanishing order over $D_\Gamma$ of the stable differentials on the irreducible component corresponding to $v$. Therefore, multiplying by $\delta_v^{-1}$ makes the limiting differential become not identically zero on $v$.  
By definition, this implies that $\delta_v$ is an adjusting parameter for $v$.

Now we construct a morphism from the IVC to the blowup of $\Mbar_{0,n}$ along~$N$ by using the universal property of the blowup. More precisely, as we blow up (in a neighbourhood of $D_\Gamma$) the ideal generated by all $\delta_v$, it suffices to check that this ideal becomes principal on the IVC. Recall that the IVC parameterizes pointed stable differentials (of prescribed type) that are not identically zero, where  
a stable differential is a section of 
the dualizing sheaf over the stable curve, considered up to an overall scaling by a non-zero constant factor. If a vertex $v$ is not a local maximum of $\Gamma$, \textit{i.e.}~if there exists an edge $e$ going up from $v$, then the (stable) differential on the irreducible component corresponding to $v$ is identically zero. Thus given a (not identically zero)
stable differential, we can declare a local maximum vertex $v$ of $\Gamma$ to be a global maximum if and only if the stable differential on the corresponding irreducible component of the curve is not identically zero. By the preceding discussion, this is precisely to say that all adjusting parameters $\delta_v$ for the global maxima vertices $v$ differ by units, and divide all the other $\delta_v$. Hence the ideal $N(\Gamma)$ pulls back a  principal ideal on the IVC, which induces the map (locally) from the IVC to the blowup of $\Mbar_{0,n}$ along $N(\Gamma)$.  

Next we construct a morphism in the opposite direction, locally near $D_\Gamma$ from the blowup of $\Mbar_{0,n}$ along~$N(\Gamma)$ to the IVC, by using the universal property of the Hodge bundle over $\Mbar_{0,n}$ (twisted by the polar part of the differentials, and projectivized as always).

Consider the universal family of differentials with prescribed zeros and poles over a punctured neighbourhood of $D_\Gamma$ in $\calM_{0,n}$. We claim that this family of differentials extends to a family of stable differentials over the local blowup of $\Mbar_{0,n}$ along $N(\Gamma)$. Indeed, for each point in the preimage of $D_\Gamma$ in the blowup, we know the set of global maxima $v_1,\dots,v_k$ of the graph (with $k\ge 1$), such that any other $\delta_v$ is divisible by one of the adjusting parameters $\delta_{v_1},\dots,\delta_{v_k}$. It follows that the limiting stable differential will be not identically zero precisely on the irreducible components corresponding to $v_1,\dots,v_k$, and thus in particular not identically zero on the stable curve as a whole.
By the universal property of the projectivized Hodge bundle, the blowup along $N(\Gamma)$ carrying a family of (not identically zero) stable differentials 
admits locally a morphism to this bundle. Moreover, since, over the locus of smooth curves, this family of differentials coincides with the family of differentials in a given stratum, it implies that the image of the morphism from the blowup to the Hodge bundle is the closure of the stratum, \textit{i.e.}~the IVC. By construction, it is clear that this map is the inverse of the local morphism in the other direction. As both the blowup of $\Mbar_{0,n}$ along the global ideal sheaf $N$ and the IVC admit forgetful surjective maps to $\Mbar_{0,n}$, and as we proved that they are locally isomorphic near every stratum $D_\Gamma\subset\Mbar_{0,n}$, it follows that they are globally isomorphic.
\end{proof}

\subsection{A blowup description for arbitrary genus} 
\label{subsec:blowup-g}

Recall that the NIVC denotes the normalization of the incidence variety compactification (\textit{i.e.}~of the closure of the stratum in the Hodge bundle), and let $\Gamma$ be a partially ordered level graph of a boundary stratum in the NIVC. For every vertex $v\in V(\Gamma)$, by normality an \emph{adjusting parameter~ $h_v$} exists by \cite[Proposition~11.13]{LMS}. Recall that by definition this means that multiplying by $h_v^{-1}$ makes the limiting differential in a degenerating family not identically zero on the irreducible component of the stable curve corresponding to $v$.  Define an ideal locally around the boundary stratum of the NIVC corresponding to $\Gamma$ by 
$$ J(\Gamma) \coloneqq \prod_{(v, v') \in V(\Gamma) \times V(\Gamma)} (h_v, h_{v'})^{w(v) w(v')}, $$
where the product runs over {\em all ordered} pairs of vertices (including the case $v = v'$) and where $w(v) \coloneqq 2g(v) -2 + {\rm valence} (v)$. Since the blowup in $J(\Gamma)$ makes the adjusting parameters comparable for any two vertices, the (local) blowup of the NIVC along $J(\Gamma)$ is orderly (recall that this means that after the blowup the divisibility relation induces a full order on the set of adjusting parameters; see \cite[Definition 11.15]{LMS}). By the same argument as in the proof of \cite[Theorem~14.8]{LMS}, it follows that the normalization of this blowup is isomorphic to the moduli space of multi-scale differentials.

Finally, we show that the locally defined ideals $J(\Gamma)$ are compatible under degeneration, so that they glue to form a global sheaf of ideals~$J$ on the NIVC. For this, again it is enough to check compatibility under an edge contraction (recalling that unlike in the genus zero case, the edge can be a loop). First, in the case of a loop, by the formula for $w(v)$, we see that contracting a loop does not change $J(\Gamma)$. Now suppose that two distinct vertices $v_1$, $v_2$ of $\Gamma$ connected by an edge $e$ are merged, when $e$ is contracted, to a vertex~$v'$ in the resulting graph $\Gamma'$.  Note that this contraction makes $h_{v_1}\sim h_{v_2}\sim h_{v'}$ modulo units. Moreover, $w(v') = w(v_1) + w(v_2)$. Then for any vertex~$u$  different from $v_1, v_2, v'$, we have 
$$ \left(h_{v_1}, h_u\right)^{2w(v_1)w(u)} \left(h_{v_2}, h_u\right)^{2w(v_2)w(u)} \,\sim\, \left(h_{v'}, h_u\right)^{2 w(v')w(u)},  $$
$$ \left(h_{v_1}, h_{v_1}\right)^{w(v_1)^2} \left(h_{v_2}, h_{v_2}\right)^{w(v_2)^2} \left(h_{v_1}, h_{v_2}\right)^{2w(v_1)w(v_2)} \, \sim \, \left(h_{v'}, h_{v'}\right)^{w(v')^2}. $$ 
It follows that $J(\Gamma')$ specializes to $J(\Gamma)$. Therefore, the local ideals $J(\Gamma)$ can be glued to a global sheaf of ideals $J$. In summary, we have proven the following theorem. 

\begin{theorem}
\label{thm:blowup-g}
The main component $\PP(\mathcal{MS}_\mu)$ of\, $\PP(\GMS)$ is the normalization of the blowup of the NIVC in the ideal sheaf $J$; in particular, its coarse moduli space is a projective variety. 
\end{theorem}

\begin{remark}
\label{rem:blowup-ivc}
For arbitrary genera, one can describe the IVC (and then also the rubber and multi-scale spaces) by blowing up the normalization of the closure of the stratum in the  Deligne--Mumford compactification $\Mbar_{g,n}$, which we denote by NDM. The argument is similar to the one in the proof of~\ref{prop:blowup-0}. Since the NDM is normal, for every vertex $v$ of $\Gamma$, an adjusting parameter $h_v$ for $v$ exists as in \cite[Proposition~11.13]{LMS}. Then the blowup of the NDM along the (local) ideals  $(h_{v_1}, \ldots, h_{v_k})$, where $v_1, \ldots, v_k$ are local maximum vertices of $\Gamma$, carries a family of stable differentials,  and hence it maps to the IVC by the universal property of the Hodge bundle. The inverse map from the IVC to this blowup is similarly obtained by using the universal property of the blowup.  

To see that these local ideals patch together to form a global sheaf of ideals, suppose that a local maximum vertex $v_1$ joins a lower vertex $v_0$ via an edge $e$.  Suppose further that $e$ is contracted so that $v_1$ and $v_0$ merge as one vertex $v'_1$, which makes $h_{v_1} \sim h_{v'_1}$ modulo units.  If $v'_1$ remains a local maximum, then we have $(h_{v_1}, h_{v_2}, \ldots, h_{v_k}) = (h_{v'_1}, h_{v_2}, \ldots, h_{v_k})$ after contracting $e$, so these ideals match. If $v'_1$ is not a local maximum, then there exists another local maximum vertex, say $v_2$, that goes along a path downward to $v'_1$ (in terms of the partial order of $\Gamma$). It follows that $h_{v_2}$ divides $h_{v_1} \sim h_{v'_1}$ and hence $(h_{v_1}, h_{v_2}, \ldots, h_{v_k}) = (h_{v_2}, \ldots, h_{v_k}) $ after contracting $e$, so these ideals still match. 
\end{remark}

\section*{Appendix. Sign conventions}\label{appendix:sign_conventions}
\addcontentsline{toc}{section}{Appendix. Sign conventions}

The sign conventions adopted in \cite{MarcusWiseLog} and in \cite{LMS} are opposite to one another; as this sign plays a more prominent role in \cite{LMS}, we follow that sign convention, which we now explain in the logarithmic language.

Let $(X, \M_X)$ be a log scheme and $\beta \in \ghost_X^\gp(X)$. The preimage of $\beta$ in the exact sequence
\begin{equation*}
1 \lra \ca O_X^\times \lra \M_X^\gp \lra \ghost_X^\gp \lra 1
\end{equation*}
is a $\bb G_m$-torsor that we denote by $\ca O_X^\times(\beta)$, from which we construct a line bundle $\ca O_X(\beta)$ by gluing in the \emph{zero}\footnote{Here Marcus and Wise glue in the $\infty$-section. } section. In particular, 
\begin{enumerate}
\item
if $X$ has divisorial log structure and $\beta \in \ghost_X(X)$, then $\ca O_X(\beta)$ is naturally an ideal sheaf on $X$;
\item
 if $(X,x)$ is a DVR with divisorial log structure at $x$, then the stalk of $\ghost_X$ at $x$ is naturally identified with $\bb N$, and the association $\beta \mapsto \ca O_X(\beta)$ sends $n$ to $\ca O_X(-nx)$; 
\item if $a \le b \in \ghost_X(X)^\gp$, then we have a natural map $\ca O_X(b) \to \ca O_X(a)$. 
\end{enumerate}

If $e\colon u \to v$ is a directed edge of a graph $\Gamma$ of length $\delta_e$, and $\beta$ is a function on the vertices of $\gamma$ with slope $\kappa$ along $e$, then $\beta(v) = \beta(u) + \kappa \cdot \delta_e$. We identify a half-edge $h$ attached to a vertex $e$ with an \emph{outgoing} edge at $e$.

If $(X/B, \beta)$ is a nuclear object of $\cat{Rub}$, then the image of $\beta$
is totally ordered with \emph{largest} element $0$. If the image of $\beta$ has
cardinality $N+1$, then there is a unique isomorphism of totally ordered sets
between $\operatorname{Im}(\beta)$ and $\{0,-1, \dots, -N\}$ (the latter having largest
element~$0$). We denote by $\ell\colon V \to \{0,-1, \dots, -N\}$ the induced map,
in accordance with~\ref{eq:normlev}.

If $e$ is an edge between vertices $u$ and $v$, we define $\ell^+(e)$ and $\ell^-(e)$ to be the unique pair of elements of $\{0,-1, \dots, -N\}$ such that $\ell^+(e) \ge \ell^-(e)$ and $\{\ell^+(e), \ell^-(e)\} = \{\ell(u), \ell(v)\}$.


\newcommand{\etalchar}[1]{$^{#1}$}
\def\cprime{$'$}
\providecommand{\bysame}{\leavevmode\hbox to3em{\hrulefill}\thinspace}


\begin{thebibliography}{HMP{\etalchar{+}}25+++}

\bibitem[AOV11]{Abramovich2011Twisted-stable-}
D.~Abramovich, M.~Olsson, and A.~Vistoli, \emph{Twisted stable maps to tame
  {A}rtin stacks}, J.~Algebraic Geom.\ \textbf{20} (2011), no.~3, 399--477,
  \doi{10.1090/S1056-3911-2010-00569-3}.

\bibitem[AV02]{MR1862797}
D.~Abramovich and A.~Vistoli, \emph{Compactifying the space of stable maps}, J.~Amer.\ Math.\ Soc.\ \textbf{15} (2002), no.~1, 27--75,
  \doi{10.1090/S0894-0347-01-00380-0}.

\bibitem[AW18]{AbraWise}
D.~Abramovich and J.~Wise, \emph{Birational invariance in logarithmic
  {G}romov--{W}itten theory}, Compos.\ Math.\ \textbf{154} (2018), no.~3,
  595--620, \doi{10.1112/S0010437X17007667}.

\bibitem[ACG11]{acgh2}
E.~Arbarello, M.~Cornalba, and P.~Griffiths, \emph{Geometry of algebraic
  curves. {V}olume {II}} (with a contribution by J.\,D.~Harris), Grundlehren math.\ Wiss., Vol.~268, Springer, Heidelberg, 2011, \doi{10.1007/978-3-540-69392-5}. 

\bibitem[BHP{\etalchar{+}}23]{BHPSS}
Y.~Bae, D.~Holmes, R.~Pandharipande, J.~Schmitt, and R.~Schwarz, \emph{Pixton's
  formula and {A}bel--{J}acobi theory on the {P}icard stack}, Acta Math.\
  \textbf{230} (2023), no.~2, 205--319, \doi{10.4310/acta.2023.v230.n2.a1}.

\bibitem[BCG{\etalchar{+}}18]{BCGGM1}
M.~Bainbridge, D.~Chen, Q.~Gendron, S.~Grushevsky, and M.~M{\"o}ller,
  \emph{Compactification of strata of Abelian differentials}, Duke Math.~J.\
  \textbf{167} (2018), no.~12, 2347--2416, \doi{10.1215/00127094-2018-0012}.

\bibitem[BCG{\etalchar{+}}19]{LMS}
\bysame, \emph{The moduli space of multi-scale differentials}, 
  preprint \arXiv{1910.13492} (2019).

\bibitem[BH23]{biesel2023fine}
O.~Biesel and D.~Holmes, \emph{Fine compactified moduli of enriched structures
  on stable curves}, Mem.\ Amer.\ Math.\ Soc.\ \textbf{285} (2023), no.~1416, \doi{10.1090/memo/1416}. 

\bibitem[Boj19]{Boj}
V.~Bojkovi{\'c}, \emph{Mittag-leffler problems on Berkovich curves}, 
  preprint \arXiv{1901.07650} (2019).

\bibitem[Che17]{chensurvey}
D.~Chen, \emph{Teichm\"{u}ller dynamics in the eyes of an algebraic geometer},
In: \emph{Surveys on recent developments in algebraic geometry}, pp.~171--197, Proc.\ Sympos.\ Pure Math., Vol.~95, Amer.\ Math.\ Soc., Providence, RI, 2017, \doi{10.1090/pspum/095/01626}.

\bibitem[CC19]{ChenChenLog}
D.~Chen and Q.~Chen, \emph{Spin and hyperelliptic structures of log twisted
differentials}, Selecta Math.\ (N.S.) \textbf{25} (2019), no.~2, Paper No.~20, \doi{10.1007/s00029-019-0467-x}.

\bibitem[CCM24]{ccm}
D.~Chen, M.~Costantini, and M.~M\"oller, \emph{On the {K}odaira dimension of
  moduli spaces of {A}belian differentials}, Camb.~J.\ Math.\ \textbf{12} (2024), no.~3, 623--752, \doi{10.4310/cjm.241001000503}.


\bibitem[Chi08]{Chiodo_twisted}
A.~Chiodo, \emph{Towards an enumerative geometry of the moduli space of twisted
  curves and {$r$}th roots}, Compos.\ Math.\ \textbf{144} (2008), no.~6,
  1461--1496, \doi{10.1112/S0010437X08003709}.

\bibitem[CMS23]{CMS2023}
M.~Costantini, M.~M\"oller, and J.~Schwab, \emph{Chern classes of linear
  submanifolds with application to spaces of k-differentials and ball
  quotients}, preprint \arXiv{2303.17929} (2023).
  
\bibitem[CMZ23]{CoMoZadiffstrata}
M.~Costantini, M.~M\"{o}ller, and J.~Zachhuber, \emph{{\tt diffstrata}---a
  Sage package for calculations in the tautological ring of the moduli space
  of abelian differentials}, Exp.\ Math.\ \textbf{32} (2023), no.~3, 545--565,
  \doi{10.1080/10586458.2021.1980750}.

\bibitem[CMZ24]{CoMoZa}
\bysame, \emph{The area is a good enough metric}, Ann.\ Inst.\ Fourier
(Grenoble) \textbf{74} (2024), no.~3, 1017--1059,
\doi{10.5802/aif.3592}.

\bibitem[DSvZ21]{DSvZ}
V.~Delecroix, J.~Schmitt, and J.~van Zelm, \emph{admcycles---a {S}age package for calculations in the tautological ring of the moduli space of stable
  curves}, J.~Softw.\ Algebra Geom.\ \textbf{11} (2021), no.~1, 89--112,
  \doi{10.2140/jsag.2021.11.89}.
  
\bibitem[Edi98]{edixhoven1998neron}
  B.~Edixhoven, \emph{On N{\'e}ron models, divisors and modular curves}, J.~Ramanujan Math.\ Soc.\ \textbf{13} (1998), no.~2, 157--194.

\bibitem[FWY21]{FWY}
H.~Fan, L.~Wu, and F.~You, \emph{Higher genus relative {G}romov--{W}itten theory
  and double ramification cycles}, J.~Lond.\ Math.\ Soc.~(2) \textbf{103} (2021), no.~4, 1547--1576, \doi{10.1112/jlms.12417}.

\bibitem[Fan01]{Fantechi-Stacks-for-everybody}
B.~Fantechi, \emph{Stacks for everybody}, In: \emph{European {C}ongress of {M}athematics,  {V}ol. {I}} ({B}arcelona, 2000), pp.~349--359, Progr.\ Math., vol.~201, Birkh\"{a}user Verlag,  Basel, 2001, \doi{10.1007/978-3-0348-8268-2_20}.

\bibitem[FP18]{fapa}
G.~Farkas and R.~Pandharipande, \emph{The moduli space of twisted canonical
  divisors}, J.~Inst.\ Math.\ Jussieu \textbf{17} (2018), no.~3, 615--672,
  \doi{10.1017/S1474748016000128}.

\bibitem[GLN23]{DaniloEulerChar}
A.~Giacchetto, D.~Lewa\'{n}ski, and P.~Norbury, \emph{An intersection-theoretic
  proof of the {H}arer--{Z}agier formula}, Algebr.\ Geom.\ \textbf{10} (2023),
  no.~2, 130--147, \doi{10.14231/AG-2023-004}.

\bibitem[Gil12]{Gillam}
W.\,D.~Gillam, \emph{Logarithmic stacks and minimality}, Internat.~J.\ Math.\
  \textbf{23} (2012), no.~7, 1250069, \doi{10.1142/S0129167X12500693}.

\bibitem[GV05]{GV}
T.~Graber and R.~Vakil, \emph{Relative virtual localization and vanishing of
  tautological classes on moduli spaces of curves}, Duke Math.~J.\ \textbf{130}
  (2005), no.~1, 1--37, \doi{10.1215/S0012-7094-05-13011-3}.

\bibitem[GS13]{Gross2013Logarithmic-Gro}
M.~Gross and B.~Siebert, \emph{Logarithmic {G}romov-{W}itten invariants}, J.~Amer.\ Math.\ Soc.\ \textbf{26} (2013), no.~2, 451--510,
  \doi{10.1090/S0894-0347-2012-00757-7}.

\bibitem[HMOP23]{Holmes2020Models-of-Jacob}
D.~Holmes, S.~Molcho, G.~Orecchia, and T.~Poiret, \emph{Models of {J}acobians
  of curves}, J.~reine angew.\ Math.\ \textbf{801} (2023), 115--159,
  \doi{10.1515/crelle-2023-0031}.

\bibitem[HMP{\etalchar{+}}25]{HMPPS}
D.~Holmes, S.~Molcho, R.~Pandharipande, A.~Pixton, and J.~Schmitt,
  \emph{Logarithmic double ramification cycles}, Invent.\ Math.\ \textbf{240}
  (2025), no.~1, 35--121, \doi{10.1007/s00222-025-01318-z}.

\bibitem[HS21]{HolmesSchmitt}
D.~Holmes and J.~Schmitt, \emph{Infinitesimal structure of the pluricanonical
  double ramification locus}, Compos.\ Math.\ \textbf{157} (2021), no.~10,
  2280--2337, \doi{10.1112/s0010437x21007557}.

\bibitem[JPPZ17]{Janda2016Double-ramifica}
F.~Janda, R.~Pandharipande, A.~Pixton, and D.~Zvonkine, \emph{Double
  ramification cycles on the moduli spaces of curves}, Publ.\ Math.\ Inst.\ Hautes \'{E}tudes Sci.\ \textbf{125} (2017), 221--266,
  \doi{10.1007/s10240-017-0088-x}.

\bibitem[Jar00]{Jarvis_higher_spin}
T.\,J.~Jarvis, \emph{Geometry of the moduli of higher spin curves}, Internat.~J.\  Math.\ \textbf{11} (2000), no.~5, 637--663, \doi{10.1142/S0129167X00000325}.

\bibitem[Kat00]{Kato2000Log-smooth-defo}
F.~Kato, \emph{Log smooth deformation and moduli of log smooth curves},
  Internat.~J.\ Math.\ \textbf{11} (2000), no.~2, 215--232,
  \doi{10.1142/S0129167X0000012X}.

\bibitem[Kat89a]{Kato1989Logarithmic-deg}
K.~Kato, \emph{Logarithmic degeneration and dieudonn{\'e} theory}, preprint,
  available at
  \url{https://www.math.brown.edu/dabramov/LOGGEOM/Kato-Dieudonne.pdf} (1989).

\bibitem[Kat89b]{Kato1989Logarithmic-str}
\bysame, \emph{Logarithmic structures of {F}ontaine-{I}llusie}, In: \emph{Algebraic  analysis, geometry, and number theory} ({B}altimore, {MD}, 1988), pp.~191--224, Johns Hopkins Univ.\ Press, Baltimore, MD, 1989.

\bibitem[Li01]{Li2001Stable-morphism}
J.~Li, \emph{Stable morphisms to singular schemes and relative stable
  morphisms}, J.~Differential Geom.\ \textbf{57} (2001), no.~3, 509--578, \doi{10.4310/jdg/1090348132}.

\bibitem[MW20]{MarcusWiseLog}
S.~Marcus and J.~Wise, \emph{Logarithmic compactification of the
  {A}bel-{J}acobi section}, Proc.\ Lond.\ Math.\ Soc.~(3) \textbf{121} (2020),
  no.~5, 1207--1250, \doi{10.1112/plms.12365}.

\bibitem[MPS23]{MPS}
S.~Molcho, R.~Pandharipande, and J.~Schmitt, \emph{The {H}odge bundle, the
  universal 0-section, and the log {C}how ring of the moduli space of curves},
  Compos.\ Math.\ \textbf{159} (2023), no.~2, 306--354,
  \doi{10.1112/S0010437X22007874}.

\bibitem[MR24]{MR}
S.~Molcho and D.~Ranganathan, \emph{A case study of intersections on blowups of
  the moduli of curves}, Algebra Number Theory \textbf{18} (2024), no.~10,
  1767--1816, \doi{10.2140/ant.2024.18.1767}.

\bibitem[Ngu24]{nguyen}
D.\,M.~Nguyen, \emph{The incidence variety compactification of strata of
  {$d$}-differentials in genus 0}, Int.\ Math.\ Res.\ Not.\ IMRN (2024), no.~10,
8734--8757, \doi{10.1093/imrn/rnad326}.

\bibitem[Ogu18]{Ogus2018Lectures-on-Log}
  A.~Ogus, \emph{Lectures on logarithmic algebraic geometry}, Cambridge Stud.\
 Adv.\ Math., Vol.~178, Cambridge Univ.\ Press, 2018, \doi{10.1017/9781316941614}. 

\bibitem[RSPW19]{RSPWI}
D.~Ranganathan, K.~Santos-Parker, and J.~Wise, \emph{Moduli of stable maps in
  genus one and logarithmic geometry, {I}}, Geom.\ Topol.\ \textbf{23} (2019),
  no.~7, 3315--3366, \doi{10.2140/gt.2019.23.3315}.

\bibitem[Sta25]{stacks-project}
{The Stacks Project Authors}, \emph{Stacks project},
  \url{http://stacks.math.columbia.edu}, 2025.
  
\bibitem[Wis16]{Wise2016Moduli-of-morph}
J.~Wise, \emph{Moduli of morphisms of logarithmic schemes}, Algebra Number
  Theory \textbf{10} (2016), no.~4, 695--735, \doi{10.2140/ant.2016.10.695}. 

\bibitem[Wri15]{wrightsurvey}
A.~Wright, \emph{Translation surfaces and their orbit closures: an introduction
  for a broad audience}, EMS Surv.\ Math.\ Sci.\ \textbf{2} (2015), no.~1,
  63--108, \doi{10.4171/EMSS/9}.

\bibitem[Zor06]{zorichsurvey}
  A.~Zorich, \emph{Flat surfaces}, In: \emph{Frontiers in number theory, physics, and geometry. {I}}, pp.~437--583, Springer-Verlag, Berlin, 2006.

\end{thebibliography}
\end{document}